\documentclass[letterpaper,11pt]{amsart}
\usepackage{amsmath}
\usepackage{amsthm}
\usepackage{mathrsfs}
\usepackage{amssymb,amscd}
\usepackage[all]{xy}
\usepackage{mathtools}
\usepackage{fullpage}
\usepackage{comment}
\usepackage{tikz-cd}
\usepackage{bbm}

\usepackage{xcolor}
\usepackage[colorlinks=true,linkcolor=blue,breaklinks,pagebackref]{hyperref}

\usepackage[lite,abbrev,msc-links,alphabetic]{amsrefs}

\newcommand{\fa}{\mathfrak{a}}
\newcommand{\fb}{\mathfrak{b}}
\newcommand{\fc}{\mathfrak{c}}

\newcommand{\fg}{\mathfrak{g}}
\newcommand{\fh}{\mathfrak{h}}

\newcommand{\fm}{\mathfrak{m}}
\newcommand{\fn}{\mathfrak{n}}
\newcommand{\fo}{\mathfrak{o}}
\newcommand{\fp}{\mathfrak{p}}

\newcommand{\fs}{\mathfrak{s}}

\newcommand{\fu}{\mathfrak{u}}

\newcommand{\fz}{\mathfrak{z}}

\newcommand{\fX}{\mathfrak{X}}

\newcommand{\bA}{\mathbb{A}}

\newcommand{\C}{\mathbb{C}}
\newcommand{\Q}{\mathbb{Q}}
\newcommand{\R}{\mathbb{R}}

\newcommand{\Z}{\mathbb{Z}}

\newcommand{\cA}{\mathcal{A}}
\newcommand{\cB}{\mathcal{B}}

\newcommand{\cE}{\mathcal{E}}
\newcommand{\cF}{\mathcal{F}}

\newcommand{\cH}{\mathcal{H}}

\newcommand{\cO}{\mathcal{O}}

\newcommand{\cS}{\mathcal{S}}
\newcommand{\cT}{\mathcal{T}}
\newcommand{\cU}{\mathcal{U}}

\newcommand{\rd}{\mathrm{d}}

\newcommand{\tS}{\mathtt{S}}

\newcommand{\rB}{\ensuremath{\mathrm{B}}}

\newcommand{\rG}{\ensuremath{\mathrm{G}}}
\newcommand{\rH}{\ensuremath{\mathrm{H}}}

\newcommand{\rR}{\ensuremath{\mathrm{R}}}
\newcommand{\rS}{\ensuremath{\mathrm{S}}}

 \DeclareMathOperator{\supp}{supp}
 \DeclareMathOperator{\vol}{vol}

\DeclareMathOperator{\Gal}{Gal}

\DeclareMathOperator{\GL}{GL}

\DeclareMathOperator{\U}{U}

\DeclareMathOperator{\gl}{\mathfrak{gl}}

\DeclareMathOperator{\Hom}{Hom}

\newtheorem{theorem}{Theorem}
\newtheorem{proposition}[theorem]{Proposition}
\newtheorem{lemma}[theorem]{Lemma}

\newtheorem{`conjecture'}[theorem]{``Conjecture''}
\newtheorem{corollary}[theorem]{Corollary}

\theoremstyle{definition}
\newtheorem{definition}[theorem]{Definition}

\newtheorem{remark}[theorem]{Remark}

\numberwithin{equation}{section}
\numberwithin{theorem}{section}

\renewcommand{\to}{%
   \ifbool{@display}{\longrightarrow}{\rightarrow}%
   }
\let\shortmapsto\mapsto
\renewcommand{\mapsto}{%
   \ifbool{@display}{\longmapsto}{\shortmapsto}%
   }
\newcommand{\hooklongrightarrow}{\mathrel{\mkern 0.5mu\lhook\mkern -3.5mu\relbar\mkern -3mu \rightarrow }}
\newcommand{\inj}{%
   \ifbool{@display}{\hooklongrightarrow}{\hookrightarrow}
   }
\newcommand{\isoarrow}{%
   \ifbool{@display}{\overset{\sim}{\longrightarrow}}{\xrightarrow\sim}%
   }
\newlength{\olen}
\newlength{\ulen}
\newlength{\xlen}
\newcommand{\xra}[2][]{%
   \ifbool{@display}%
      {\settowidth{\olen}{$\overset{#2}{\longrightarrow}$}%
       \settowidth{\ulen}{$\underset{#1}{\longrightarrow}$}%
       \settowidth{\xlen}{$\xrightarrow[#1]{#2}$}%
       \ifdimgreater{\olen}{\xlen}%
          {\underset{#1}{\overset{#2}{\longrightarrow}}}%
          {\ifdimgreater{\ulen}{\xlen}%
             {\underset{#1}{\overset{#2}{\longrightarrow}}}
             {\xrightarrow[#1]{#2}}}}%
      {\xrightarrow[#1]{#2}}
   }
\makeatother
\newcommand{\xyra}[2][]{%
   \settowidth{\xlen}{$\xrightarrow[#1]{#2}$}%
   \ifbool{@display}%
      {\settowidth{\olen}{$\overset{#2}{\longrightarrow}$}%
       \settowidth{\ulen}{$\underset{#1}{\longrightarrow}$}%
       \ifdimgreater{\olen}{\xlen}%
          {\mathrel{\xymatrix@M=.12ex@C=3.2ex{\ar[r]^-{#2}_-{#1} &}}}%
          {\ifdimgreater{\ulen}{\xlen}%
             {\mathrel{\xymatrix@M=.12ex@C=3.2ex{\ar[r]^-{#2}_-{#1} &}}}
             {\mathrel{\xymatrix@M=.12ex@C=\the\xlen{\ar[r]^-{#2}_-{#1} &}}}}}%
      {\mathrel{\xymatrix@M=.12ex@C=\the\xlen{\ar[r]^-{#2}_-{#1} &}}}%
   }
\makeatletter
\newcommand{\xla}[2][]{%
   \ifbool{@display}%
      {\settowidth{\olen}{$\overset{#2}{\longleftarrow}$}%
       \settowidth{\ulen}{$\underset{#1}{\longleftarrow}$}%
       \settowidth{\xlen}{$\xleftarrow[#1]{#2}$}%
       \ifdimgreater{\olen}{\xlen}%
          {\underset{#1}{\overset{#2}{\longleftarrow}}}%
          {\ifdimgreater{\ulen}{\xlen}%
             {\underset{#1}{\overset{#2}{\longleftarrow}}}
             {\xleftarrow[#1]{#2}}}}%
      {\xleftarrow[#1]{#2}}
   }
\newcommand{\lra}{%
   \ifbool{@display}{\longleftrightarrow}{\leftrightarrow}%
   }

\begin{document}

\title{On the Geometric Side of the Jacquet-Rallis Relative Trace Formula}

\author{Weixiao Lu}
\address{Massachusetts Institute of Technology, Department of Mathematics, 77 Massachusetts Avenue, Cambridge, MA 02139, USA}
\email{weixiaol@mit.edu}

\date{\today}

\begin{abstract}

We study some aspects of the geometric side of the Jacquet-Rallis relative trace formula. Globally, we compute each geometric term of the Jacquet-Rallis relative trace formula on the general linear group for regular supported test functions. We prove that it can be described by the regular orbital integral. Locally, we show that the regular orbital integral can be compared with the semisimple orbital integral on the unitary group.

\end{abstract}

\maketitle

\tableofcontents

\section{Introduction}\label{sec:introduction}

The global Gan--Gross--Prasad(GGP) conjecture \cite{GGP12} relates the non-vanishing of period integral on classical groups to the non-vanishing of the central value of certain L-functions. In \cite{JR11}, Jacquet and Rallis proposed a relative trace formula(RTF) approach to the GGP conjecture for the Bessel periods on $\U(n) \times \U(n+1)$. In \cite{Zhang14}, Zhang solved the smooth transfer conjecture and proved the global Gan-Gross-Prasad conjecture for $\U(n) \times \U(n+1)$ under some local conditions as a consequence. 

We briefly review the Jacquet-Rallis RTF on the general linear group here. Let $E/F$ be a quadratic extension of number fields. For $k \in \Z_{\ge 1}$, let
\[
        \rG'_k := \mathrm{Res}_{E/F}\GL_{k,E}. 
    \]
Let $\rG'= \rG'_n \times \rG'_{n+1}$. $\rG'$ has two subgroups $\rH_1,\rH_2$, where
    \begin{equation} \label{eq:subgroup G'}
          \rH_1 = (h,\begin{pmatrix}
            h & \\ & 1
        \end{pmatrix})(h \in \rG'_n), \quad \rH_2 = \GL_{n,F} \times \GL_{n+1,F}.
    \end{equation}
Let $\bA=\bA_F$, the ad\`{e}le ring of $F$ and $[\rG'] := \rG'(F)\backslash \rG'(\bA)$. We write $\eta$ for the quadratic character on $\bA^\times$ associated with $E/F$. By an abuse of notation, we also write $\eta$ for the character on $\rH_2(\bA)$ defined by 
\[
    \eta(h_{2,n},h_{2,n+1}) = \eta(h_{2,n})^{n+1} \eta(h_{2,n+1})^n.
\]
Let $f \in \cS(\rG'(\bA))$ be a Schwartz function and let $K_f(x,y)$ be its automorphic kernel function:
    \[
        K_f(x,y) = \sum_{\gamma \in \rG'(F)} f(x^{-1}\gamma y) ,\quad (x,y) \in [\rG'] \times [\rG'].
    \]
Consider the distribution
    \begin{equation} \label{eq:I(f)}
         I(f) = \int_{[\rH_1]} \int_{[\rH_2]} K_f(h_1,h_2) \eta(h_2) \rd h_1 \rd h_2.
    \end{equation}
 Under some conditions on $f$, this integral is absolutely convergent and has a geometric and spectral expansion.

In ~\cite{Zydor20}, Zydor defined a regularization of the integral ~\eqref{eq:I(f)} for all compactly supported smooth functions $f \in C_c^\infty(\rG'(\bA))$, hence ~\eqref{eq:I(f)} makes sense for such $f$, and it has geometric and spectral expansion. The regularization and geometric/spectral expansions were extended by Beuzart-Plessis, Chaudouard, and Zydor in ~\cite{BPCZ} and ~\cite{CZ21} for general Schwartz functions, hence a full coarse Jacquet-Rallis RTF is established. Together with other techniques developed in ~\cite{Yun},\cite{BP21c},~\cite{Xue19},\cite{CZ21},~\cite{BPLZZ}, the endoscopic case of the global GGP conjecture for $\U(n) \times \U(n+1)$ was proved in ~\cite{BPCZ}. It was extended to certain Eisenstein series and higher corank cases by Beuzart-Plessis--Chaudouard ~\cite{BPC23}. In this article, we will study some aspects of the geometric side of the Jacquet--Rallis relative trace formula.

\subsection{Global Results}
Let $\rB=\rH_1 \backslash \rG' /\rH_2$ be the GIT quotient and let $q:\rG' \to \rB$ be the quotient map, for any $b \in \rB(F)$, we write $\rG'_b$ for the fiber of $b$ under the quotient map. Then there is a $b$-part of the distribution $I$, denoted by $I_b$, which is a regularization of the integral
\[
    \int_{[\rH_1]} \int_{[\rH_2]} \left( \sum_{\gamma \in \rG'_b(F)} f(h_1^{-1} \gamma h_2) \right) \eta(h_2) \rd h_1 \rd h_2.
\]
The geometric expansion of $I$ is then the identity
    \[
        I(f) = \sum_{b \in \rB(F)} I_b(f).
    \]
If $b$ is regular semi-simple, pick any $\gamma \in \rG'(F)$ with image $b$, the distribution $I_b$ can be computed via the orbital integral
    \[
        I_b(f) = \int_{\rH_1(\bA)} \int_{\rH_2(\bA)} f(h_1^{-1} \gamma h_2) \eta(h_2) \rd h_1 \rd h_2.
    \]
But for general $b \in \rB(F)$, there is no easy interpretation of $I_b$. One of the main goals of this paper is to compute $I_b(f)$, for general $b \in \rB(F)$, under the assumption that there is a place $v$ of $F$ such that $f=f_v f^v$, where $f^v \in \cS(\rG'(\bA^v)),f_v \in \cS(\rG'(F_v))$ and $f_v$ is supported in the regular subset $\rG'_{\mathrm{reg}}(F_v)$ of $\rG'(F_v)$, where $\rG_{\mathrm{reg}}'$ is the Zariski open subset of $\rG'$ consisting of elements whose stabilizer under the $\rH_1 \times \rH_2$ action is trivial.

Let $\gamma \in \rG'_\mathrm{reg}(F)$ be a regular element and $f \in \cS(\rG'(\bA))$ (with no support condition), in Subsection ~\ref{subsec:global_orbital_integral}, we will define a distribution $I_\gamma(f)$, which can be written formally as
\begin{equation} \label{eq:intro_regularized_integral}
    I_\gamma(f) = \int_{\rH_1(\bA)} \int_{\rH_2(\bA)} f(h_1^{-1}\gamma h_2) \eta(h_2) \rd h_1 \rd h_2,
\end{equation}
note that this integral is not convergent in general, therefore some regularization is needed. Once the distribution $I_\gamma$ is defined, our main theorem can be summarized as follows:
\begin{theorem}[See Theorem ~\ref{thm:global_II}] \label{thm:intro_global_1}
    If $f \in \cS(\rG'(\bA))$ is of the form $f_vf^v$, and $f_v$ is regular supported, then for any $b \in \rB(F)$
    \[
        I_b(f) = \sum_{\gamma} I_\gamma(f),
    \]
    where in the summation, $\gamma$ runs through any representative of $\rH_1(F) \times \rH_2(F)$ orbits of regular elements in $\rG'_{b}(F)$. 
\end{theorem}

In Subsection ~\ref{subsec:regular_orbits}, we will describe the regular orbits of $\rG'_b(F)$
explicitly. In particular, we will see that the sum above is a finite sum. 

There is a Zariski open subset $\rG'_+$ of $\rG'$ such that $\rG'_+ \subset \rG'_{\mathrm{reg}}$ (see Subsection ~\ref{subsec:symmetric space S}). When $f_v$ is supported in $\rG'_+(F_v)$, we will have a better description of $I_\gamma(f)$ as follows:

\begin{theorem}[See Theorem ~\ref{thm:main}] \label{thm:intro_global_2}
    If $f \in \cS(\rG'(\bA))$ is of the form $f_v f^v$, and $f_v$ is supported in $\rG'_+(F_v)$. Let $\gamma \in \rG'_+(F)$, then we have the following assertions: \begin{enumerate}
        \item  For $s \in \C$ with $\mathrm{Re}(s)$ sufficiently negative, the integral
        \[
    \int_{\rH_1(\bA)} \int_{\rH_2(\bA)} f(h_1^{-1} \gamma h_2) \eta(h_2) \lvert \det h_1 \rvert^s \rd h_1 \rd h_2
    \]
    is absolutely convergent, and has a meromorphic continuation to $\C$ which is holomorphic at $0$. The value at $s=0$ coincides with $I_\gamma(f)$.
    \item  For each $b \in \rB(F)$, there is a unique $\rH_1(F) \times \rH_2(F)$ regular orbit inside $\rG'_b(F) \cap \rG'_+(F)$. Choose any $\gamma$ in this orbit, then
    \[
        I_b(f) = I_\gamma(f).
    \]
    \end{enumerate}
\end{theorem}

In fact, we will first prove Theorem ~\ref{thm:intro_global_2} and a Lie algebra version of it in Section ~\ref{sec:global}, and then deduce Theorem ~\ref{thm:intro_global_1} from it in Section ~\ref{sec:global_II}.

\subsection{Local Results}
Let $\gamma \in \rG'(F)$ be a regular element, our regularization of the integral ~\eqref{eq:intro_regularized_integral} will make the distribution $I_\gamma(f)$ Eulerian. In fact, to define $I_\gamma(f)$, we first study a local analogue of it in Section ~\ref{sec:local}. Let $v$ be a place of $F$, and let $f \in \cS(\rG'(F_v))$ (with no support condition), we will prove the following result in Section ~\ref{sec:local}.

\begin{proposition} \label{prop:intro_local}[See Subsection ~\ref{subsec:local_group} and Proposition ~\ref{prop:local_group}]
    Let $\gamma \in \rG'_{\mathrm{reg}}(F_v)$, then there is a map
    \[
        \cS(\rG'(F_v)) \to \{\text{meromorphic functions on }\C\}, \quad f \mapsto I_\gamma(f,s)
    \]
    which is a regularization of the integral
    \begin{equation} \label{eq:intro_local_regularization}
       \int_{\rH_1(F_v)} \int_{\rH_2(F_v)} f(h_1^{-1}\gamma h_2) \lvert \det h_1 \rvert^{s} \eta(h_2) \rd h_1 \rd h_2.
    \end{equation} 
    with the following properties:
    \begin{enumerate}
        \item $I_\gamma(f,s)$ only depends on the value of $f$ on the orbit of $\gamma$.
        \item For $x \in \rH_1(F_v)$ and $y \in \rH_2(F_v)$, let $f^{x,y}(\gamma)=f(x\gamma y^{-1}) \in \cS(\rG'(F_v))$, then
        \[
            I_\gamma(f^{x,y},s) = \lvert \det x \rvert^s \eta(y) I_\gamma(f,s).
        \]
        \item If $\gamma \in \rG'_+(F_v)$, then the integral ~\eqref{eq:intro_local_regularization} is absolutely convergent for $\mathrm{Re}(s)$ sufficiently negative, and coincides with $I_\gamma(f,s)$ when convergent. In particular, this holds when $\gamma$ is regular semisimple.
        \item There is an abelian $L$-function $L_\gamma(s)$ such that $I^\natural_\gamma(f,s) := I_\gamma(f,s)/L_\gamma(s)$ is entire for any $f$, and when $f=1_{\rG'(\cO)}$ and everything is unramified, $I_\gamma(f,s)=L_\gamma(s)$.
    \end{enumerate}
\end{proposition}
    
 When $\gamma$ is regular semisimple, the orbital integral $I_\gamma(f,s)$ has been studied in depth. One key step towards the proof of the global GGP conjecture is to show that regular semisimple orbital integral $I_\gamma(f,0)$ can be compared with similarly defined regular semisimple orbital on the unitary groups.

Let $\cH_v$ be the isometric classes of $n$-dimensional non-degenerate $E_v/F_v$ Hermitian spaces. For $V \in \cH_v$, let $V'$ be the $n+1$-dimensional Hermitian space $V \oplus E_ve_0$, the Hermitian form on $V'$ is defined by orthogonal direct sum of $V \oplus h_0$ where $h_0(e_0,e_0)=1$. We put $\rG^V = \U(V) \times \U(V')$ and $\rH^V := \U(V)$ as a diagonal subgroup of $\rG^V$. Let $f^V \in \cS(\rG^V(F_v))$ and let $\gamma \in \rG^V(F_v)$ be a regular semisimple element, the orbital integral $J_\gamma(f)$ is defined by
\[
    J_\gamma(f) = \int_{\rH^V(F_v) \times \rH^V(F_v)} f(x^{-1}\gamma y) \rd x \rd y.
\]
Let $f \in \rG'(F_v)$ and for each $V \in \cH_v$ let $f^V \in \cS(\rG^V(F_v))$, there is a notion of matching between $f$ and $f^V$ (See Subsection ~\ref{subsec:transfer}), which means $f$ and $f^V$ have the ``same" regular semisimple orbital integrals. In Section ~\ref{sec:localII}, we will prove a version of ``local singular transfer", which shows that for $f$ and $(f^V)_{V \in \cH_v}$, if their regular semisimple orbital integrals match (i.e., they are matching), then the more singular terms, i.e. the regular orbital integral $I^\natural_\gamma(f,0)$ and the semisimple orbital integral of $f^V$ also match.

\begin{theorem}\label{thm:intro_local}(See Theorem ~\ref{thm:local_singular_transfer})
    Suppose that $f$ and $(f^V)_{V \in \cH_v}$ match, let $\gamma \in \rG'_{\mathrm{reg}}(F_v)$, then we have
    \[
        I^\natural_\gamma(f,0) =  \sum_{(V,\cO)} c_{\gamma,\cO} J_{\cO}(f^V),
    \]
    where the summation runs through the set $\{(V,\cO)\mid V \in \cH_v, \cO \text{ is a semisimple orbit corresponds to }\gamma \}$, and $c_{\gamma,\cO}$ are some explicit constants independent of $f$ and $f^V$.
\end{theorem}
This result can be regarded as a local analogue of the singular transfer theorem of Chaudouard--Zydor in ~\cite{CZ21}*{Th\'{e}or\`{e}me 1.1.6.1}.

\subsection{Remarks}

In ~\cite{Zhang12}, Zhang proposed an RTF approach to the arithmetic Gan--Gross--Prasad \~(AGGP) conjecture. In ~\cite{DZ23}, a $p$-adic version of the AGGP conjecture is proved.  One key difficulty in \cite{DZ23} is caused by the use of a very specific test function at some split places, and these functions have regular support but not regular semisimple support, in order to $p$-adically interpolate the orbital integral. The result in this article could be applied to \emph{loc.cit.}.

Liu ~\cite{Liu14} proposed an RTF to attack the global GGP conjecture for the Fourier-Jacobi period on unitary groups, and using this RTF, Xue ~\cite{Xue14} proved the global GGP conjectures for Fourier-Jacobi period under some local conditions. Some of the local conditions were removed in ~\cite{Xue19}, and a full coarse RTF was established by Boisseau, Xue, and the author ~\cite{BLX}. Using the coarse RTF, the global GGP conjectures for the Fourier-Jacobi period, together with the Eisenstein and higher corank case, are proved in \emph{loc.cit.}. We will also consider the Fourier-Jacobi analogue of the main results in Section ~\ref{sec:Fourier-Jacobi}.

We give an outline of this article. After some preliminaries in Section ~\ref{sec:prelim}, in Section ~\ref{sec:RTF}, we extend the coarse Jacquet-Rallis RTF on the general linear group and its Lie algebra by allowing for a general character on $\rH_1$. In Section ~\ref{sec:global}, we will prove Theorem ~\ref{thm:intro_global_2}. In Section ~\ref{sec:local}, we will define the local regular orbital integral and study its properties, and in Section ~\ref{sec:global_II} we use it to define the global regular orbital integral and prove Theorem ~\ref{thm:intro_global_1}. In Section ~\ref{sec:localII} we prove the local singular transfer theorem. In Section ~\ref{sec:Fourier-Jacobi}, we give a Fourier-Jacobi analogue of the results. In the appendix, we extend the asymptotic of the modified kernel in ~\cite{BPCZ}*{Theorem 3.3.7.1} to general parabolic subgroups and its Levi subgroups.

\paragraph{\textbf{Acknowledgement}}
The author would like to thank Wei Zhang for the suggestion of this problem and helpful discussion on many parts of this article, he is also very grateful to Daniel Disegni, Ziqi Guo, Linli Shi, Yiyang Wang, Hang Xue, Hongfeng Zhang, and Zhiyu Zhang for helpful comments and discussions.

\section{Preliminaries} \label{sec:prelim}

\subsection{General notation}
\label{subsec:general_notations}
\begin{itemize}
    \item We fix an integer $n>0$ throughout this article.

    \item For a matrix $A$, we write $A^t$ for the transpose of $A$. 

    \item Let $F$ be a number field and let $v$ be place of $F$, we write $F_v$ for the completion of $F$ at the place $v$. In general, if $\tS$ is a finite set of places of $F$, we write $F_\tS := \prod_{v \in \tS} F_v$ and $\bA^\tS_F$ for the restricted product $\prod_{v \not \in \rS}' F_v$. We also write $F_\infty : F \otimes_\Q \R$.  If $E/F$ is a finite extension, we put $E_\tS := E \otimes_F F_\tS$.
    
    \item Let $f$ and $g$ be positive functions on a set $X$, we write $f \ll g$ if there exists $C>0$ such that $f(x) \le Cg(x)$ for all $x \in X$. We write $f \ll_{c_1,\cdots,c_r} g$ if the constant $C$ depends on the parameters $c_1,\cdots,c_r$.
    \item Let $r$ be a real number. We denote by $\cH_{<r}$ (resp. $\cH_{>r}$) the left half plane $\{ z \in \C \mid \mathrm{Re}(s)<r \}$ (resp. right half plane $\{ z \in \C \mid \mathrm{Re}(s)>r \}$).

    \item For a vector space $V$ over a field $k$, we write $V^*:=\Hom_k(V,F)$ for the dual space of $V$. We denote by $k^n$ and $k_n$ the $n$ dimensional column/row vectors respectively.

    \item Let $G$ be a reductive group over a field $k$ of characteristic 0, and assume that $G$ acts on a finite type affine $k$-scheme $X$. Let $X/G$ be the GIT quotient and let $b \in (X/G)(k)$, we denote the fiber of $b$ under the quotient map $X \to X/G$ by $X_b$, which is a closed subscheme of $X$.

    \item Let $R$ be a ring and assume that we have a homomorphism $\chi$ from $R$ to some abelian group $A$. By an abuse of notation, we also denote the homomorphism $\chi \circ \det$ from $\GL_n(R)$ to $A$ by $\chi$.
\end{itemize}

\subsection{Algebraic groups} \label{subsec:algebraic_group}
    In this subsection, we follow the notations in ~\cite{BPCZ} where the reader can find more details there. Let $F$ be a number field and let $G$ be an algebraic group over $F$, we write $\fg_\infty$ for the Lie algebra of the Lie group $G(F_\infty)$, and let $\cU(\fg_\infty)$ be the universal enveloping algebra of $\fg_\infty$. We denote by $[G] := G(F) \backslash G(\bA)$ the ad\`{e}lic quotient of $G$. We fix the Tamagawa measure $dg$ on $G(\bA)$ as described in ~\cite{BPCZ}*{Section 2.3}.
      
      For the remainder of this subsection, we assume that $G$ is connected and reductive. Fix a maximal split torus $A_0$ of $G$ and a minimal parabolic subgroup $P_0$ containing $A_0$. A parabolic subgroup $P$ of $G$ is called standard if $P \supset P_0$, and it is called semi-standard if $P \supset A_0$. Let $P$ be a semi-standard parabolic subgroup of $G$, then $P$ has a unique Levi decomposition $P=M_PN_P$ such that $M_P \supset A_0$, this is called the standard Levi decomposition. When we say Levi decomposition of a semi-standard parabolic subgroup, we will always mean the standard Levi decomposition. 

      Let $W$ be the Weyl group of $(G,A_0)$, that is the quotient by $M_0(F)$ of the normalizer of $A_0$ in $G(F)$.

    For a semi-standard parabolic subgroup $P$ of $G$, define
        \[
            \fa_P^* := X^*(P) \otimes_\Z \R ,\quad \fa_P := \Hom_\Z(X^*(P),\R).
        \]
    We endow $\fa_P$ with the Haar measure such that the lattice $X^*(P)$ has covolume 1.
        
    Let $\fa_0 := \fa_{P_0}$ and $\fa_0^* := \fa_{P_0}^*$. For $P \subset Q$, there is a natural direct sum decomposition
        \[
            \fa_P = \fa_P^Q \oplus \fa_Q, \,\fa^*_P = \fa_{P}^{Q,*} \oplus \fa_Q^*.
        \]
    In particular, we will view $\fa_P$ (resp. $\fa_P^*$) as a subspace of $\fa_0$ (resp. $\fa_0^*$). For $P \subset Q$, we put
        \[
            \epsilon_P^Q := (-1)^{\dim \fa_P - \dim \fa_Q}.
        \]
    If $Q=G$, we put $\fa_P^Q = \fa_P$.
    
    Let $P_0'$ be a semi-standard minimal parabolic subgroup containing $P$, we denote by $\Delta_{P_0'}^P$ the set of simple roots under the $A_0$ action, relative to $P_0'$, on the Lie algebra of $N_P$. We put $\Delta_{P_0'}:=\Delta_{P_0'}^G$. For semi-standard parabolic subgroups $P \subset Q$, define $\Delta_P^Q$ to be the image of $\Delta_{P_0'}^Q \setminus \Delta_{P_0'}^P$ by the projection map $\fa_0^* \to \fa_P^*$. We also have a set of coroots $\Delta_P^{Q,\vee}$, the set of weights $\widehat{\Delta}_P^Q$, and the set of coweights $\widehat{\Delta}_P^{Q,\vee}$. Let $\rho_P \in \fa_P^*$ be the half of the sum of the roots of the action of $A_P$ on $N_P$.

    For any semi-standard parabolic subgroup $P$, and for any $T \in \fa_0$, we define a point $T_P \in \fa_P$, such that for any $w \in W$ such that $wP_0 w^{-1} \subset P$, the point $T_P$ is the projection of $w \cdot T$ on $\fa_P$ under the projection map $\fa_0 \to \fa_P$.

    Let $\tau_P^Q$ be the characteristic function of 
        \[
      \{ X \in \fa_0 \mid \langle \alpha, X \rangle>0, \text{ for all }\alpha \in \Delta_P^Q \},
        \]
    and put $\fa_0^+ := \{ X \in \fa_0 \mid \tau_{P_0}^G(X)=1 \}$.
     We also write $\widehat{\tau}_P^Q$ for the characteristic function of 
         \[
      \{ X \in \fa_0 \mid \langle \varpi, X \rangle>0, \text{ for all }\varpi \in \widehat{\Delta}_P^{Q} \}.
        \]
        
     For a semi-standard parabolic subgroup $P$ of $G$, we put
        \[
            [G]_P := N_P(\bA)M_P(F) \backslash G(\bA).
        \]
    We fix a norm $\| \cdot \|$ on $G(\bA)$ as in ~\cite{BP21}*{Appendix A}. It induces a norm on $[G]_P$ by 
        \[
            \| g \|_P := \inf_{\gamma \in N_P(\bA)M_P(F)} \|\gamma g\|.
        \]
    There is a notion of weight functions on $[G]_P$ as in ~\cite{BPCZ}*{Subsection 2.4.3}. In particular, for any $\alpha \in \fa_0^*$, there is a weight $d_{P,\alpha}$ on $[G]_P$.

     We denote by $A_G^\infty$ the neutral component of real points of the maximal split central torus of $\mathrm{Res}_{F/\Q} G$. For a semi-standard parabolic subgroup $P$ of $G$, let $A_P^\infty := A_{M_P}^\infty$. We also define $A_0^\infty := A_{P_0}^\infty = A_{M_0}^\infty$.

    We fix a maximal compact subgroup $K$ of $G(\bA)$, which is in good position with $P_0$. Hence we have the Iwasawa decomposition $G(\bA)=P(\bA)K$ for all semi-standard parabolic subgroup $P$ of $G$. The map
        \[
            H_P: P(\bA) \to \fa_P, \, p \mapsto \left( \chi \mapsto \log \lvert \chi(g) \rvert  \right)
        \]
    extends to $G(\bA)$, by requiring it trivial on $K$. The map $H_P$ induces an isomorphism $A_P^\infty \cong \fa_P$, we endow $A_P^\infty$ with the Haar measure such that this isomorphism is measure-preserving. Let $G(\bA)_P^1$ be the preimage of 0 under $H_P$. $G(\bA)^1:=G(\bA)^1_G$ is a subgroup of $G(\bA)$. For general $P$, the product in $G(\bA)$ induces a direct product decomposition (of sets)
        \[
            G(\bA) = G(\bA)_P^1 \times A_P^\infty.
        \]
    The subset $G(\bA)_P^1$ descends to a subset of $[G]_P$, we denote this by $[G]_P^1$.
    
    Let $\fX(G)$ be the set of cuspidal datum of $G$, see ~\cite{BPCZ}*{Section 2.9}. Let $P$ be a semi-standard parabolic subgroup, we have the following Langlands decomposition of the $L^2$ space:
        \begin{equation} \label{eq:spectral decomposition}
              L^2([G]_P) = \bigoplus_{\chi \in \fX(G)} L^2_\chi([G]_P).
        \end{equation}
          
     Fix a norm $\| \cdot \|$ on $\fa_0$. We say $T \in \fa_0$ is \emph{sufficiently positive}, if
    there exists $C>0$ and $\varepsilon >0$ such that
    \[
            \inf_{\alpha \in \Delta_0} \langle \alpha,T \rangle \ge \max\{ C,\varepsilon \|T\| \}.
        \]

    For $T \in \fa_0$ and a semi-standard parabolic subgroup $P$ of $G$, let $F^P(\cdot,T)$ be the function on $[G]_P$ introduced in ~\cite{BPC23}*{Subsubsection 2.3.3}. It is a characteristic function of a subset of $[G]_P$, and this subset is compact modulo center.

\subsection{Function spaces} \label{subsec:function space}
    Let $F$ be a number field and let $G$ be an algebraic group over $F$. We say a function $f:G(\bA) \to \C$ is smooth, if there exists a compact open subgroup $J \subset G(\bA_f)$ such that $f$ is invariant under the right translation by $J$, and for all $g_f \in G(\bA_f)$, the function $g_\infty \mapsto f(g_fg_\infty)$ is smooth on Lie group $G(F_\infty)$. We say that the function on $[G]_P$ is smooth if it pulls back to a smooth function on $G(\bA)$. 
    
    Let $C$ be a compact subset of $G(\bA_f)$ and let $J \subset G(\bA_f)$ be a compact subgroup. Let $\cS(G(\bA),C,J)$ be the space of smooth functions $f:G(\bA) \to \C$ which are biinvariant by $J$, supported in the subset $C \times G(F_\infty)$ and such that 
    \[
        \| f \|_{N,X,Y} = \sup_{g \in G(\bA)} \lvert (\mathrm{R}(X) \mathrm{L}(Y)f)(g) \rvert < \infty
    \]
    for any $N>1$ and $X,Y \in \cU(\fg_\infty)$. Let $\cS(G(\bA))$ be a union of all $\cS(G(\bA),C,J)$, it carries a natural strict LF topology. In particular, if $V$ is a vector space over $F$, we have a space $\cS(V(\bA))$ of Schwartz function on $V(\bA)$.
    
    Let $G$ be a reductive group over $F$, and let $P$ be a semi-standard parabolic subgroup of $G$. We denote by $\cS^0([G]_P)$ the space of the rapidly decreasing measurable functions on $[G]_P$. In general, let $X \subset [G]_P$ be a measurable subset, let $\cS^0(X)$ be the set of measurable functions $f$ on $X$ such that
        \[
            \| f \|_{\infty,N} := \sup_{x \in X} \|x\|_P^N \lvert f(x) \rvert
        \]
    is finite for all $N$. Equipped with the semi-norms $\| \cdot \|_{\infty,N}$, $\cS^0(X)$ is a Frech\'{e}t space.

    We denote by $\cT^0([G]_P)$ the space of complex Radon measure $\varphi$ on $[G]_P$ such that there exists $N>0$ making the integral
    \[
        \int_{[G]_P} \| x \|_P^{-N} \lvert \varphi(x) \rvert
    \]
    finite. It carries a natural LF topology.

     Let $w$ be a weight on $[G]_P$, we write $\cS_w([G]_P)$ for the LF space of weighted Schwartz functions on $[G]_P$. It consists of smooth functions $f$ on $[G]_P$ such that there exists $N>0$, such that for all $X \in \cU(\fg_\infty)$ and $r>0$,
    \[
        \| f\|_{\infty,X,r,N} := \sup_{x \in [G]_P} \lvert R(X)f(x) \rvert w(x)^r \|x\|_P^{-N} < \infty.
    \]
    In particular, if $w=\| \cdot \|_P$, $\cS_w([G]_P) := \cS([G]_P)$ is the space of Schwartz functions on $[G]_P$, and if $w=1$, $\cS_w([G]_P) := \cT([G]_P)$ is the space of functions of uniform moderate growth on $[G]_P$. For more details on these global function spaces, see ~\cite{BPCZ}*{Section 2.5}.

    Now let $F$ be a local field of characteristic 0, and let $G$ be an algebraic group over $F$, fix a norm $\| \cdot \|$ on $G(F)$ as in ~\cite{Kottwitz}*{Section 18}. We denote by $\cS(G(F))$ the space of Schwartz function on $F$. If $F$ is non-archimedean, $\cS(G(F))$ consists of compactly supported and locally constant function, if $F$ is archimedean, it consists of smooth functions $f$ on $G(F)$ such that for any $X \in \cU(\fg_\infty)$ and $N \ge 0$
        \[
            \| f \|_{X,N} := \sup_{g \in G(F)} \|x\|^N  \lvert \mathrm{R}(X)f(x) \rvert < \infty.
        \]
 
  The space $\cS(G(F))$ carries a natural Fr\'{e}chet topology when $F$ is archimedean, and we endow $\cS(G(F))$ with the finest locally convex topology if $F$ is non-Archimedean.
    
    \subsection{The symmetric space $\rS$ and its variants} \label{subsec:symmetric space S}
    
    From now on until the end of Section ~\ref{sec:prelim}, we fix a field $F$ of characteristic 0 and let $E$ be a quadratic \'{e}tale algebra over $F$. Let $\mathtt{c}$ be the unique nontrivial evolution of $E$ that fixes $F$. For an $F$ algebra $R$, we will write $a \mapsto a^\mathtt{c}$ for the involution on $R \otimes_F E$ induced by $\mathtt{c}$. For $k \ge 1$, write $\GL_n := \GL_{n,F}$.

    Let $\rS$ be the algebraic variety over $F$ such that
    \[
        \rS(R) = \{ x \in \GL_{n+1}(R \otimes E) \mid x x^\mathtt{c} = 1  \},
    \]
    where $1$ stands for the identity matrix of size $(n+1) \times (n+1)$.

    We regard the group $\GL_n$ as a subgroup of $\GL_{n+1}$ via the embedding $g \mapsto \begin{pmatrix}
        g & \\ & 1
    \end{pmatrix}$. The group $\GL_n$ has a right action on $\rS$ by
        \[
        x \cdot g = g^{-1}xg.
        \]
    There is a $\GL_n$-equivariant isomorphism
        \[
     \nu: \mathrm{Res}_{E/F} \GL_{n+1,E}/\GL_{n+1} \cong \rS , \quad g \mapsto gg^{-1,\mathtt{c}}
        \]
    where $\GL_n$ acts on $\mathrm{Res}_{E/F} \GL_{n+1,E}/\GL_{n+1}$ by left translation. Let $\rB:=\rS/\GL_n$ be the GIT quotient.

    Note that there is a natural identification $\rG'/\rH_1 \times \rH_{2,n+1} \cong \mathrm{Res}_{E/F} \GL_{n+1,E}/\GL_{n+1,F}$ induced by the map $(g_n,g_{n+1}) \mapsto g_n^{-1}g_{n+1}$, therefore $\rB$ can also be identified with the GIT quotient $\rG'/\rH_1 \times \rH_2$. More concretely, the map
    \begin{equation} \label{eq:the_map_alpha}
         \alpha: \rG' \to \rS, \quad (g_n,g_{n+1}) \mapsto \nu(g_n^{-1}g_{n+1})
    \end{equation}
    identifies $\rG'/\rH_1 \times \rH_2$ with $\rB$.
    
    Let $\sigma \in E^\times$ with $\sigma \sigma^\mathtt{c}=1$, and let $\rS^\sigma$ be the Zariski open subset of $\rS$ consisting of $x \in \rS$ such that the matrix $x-\sigma \cdot 1$ is invertible. 
    
    We write $\gl_{n+1}:=\gl_{n+1,F}$ for the vector space of $(n+1) \times (n+1)$ matrices with coefficient in $F$. The group $\GL_n$ has a right action on $\gl_{n+1}$ given by
    \[
        A \cdot g := g^{-1}Ag.
    \]
    Let $\cB := \gl_{n+1}/\GL_n$ be the GIT quotient. From ~\cite{Zhang14}*{Lemma 3.1}, we see that $\cB$ is isomorphic to the affine space $\mathbf{A}^{2n+1}$ of dimension $2n+1$.
    
    Fix $\tau \in E$ such that $\tau^\mathtt{c}=-\tau$, let $\gl_{n+1}^\tau$ be the Zariski open subset of $\gl_{n+1}$ consists of $Y \in \gl_{n+1}$ such that the matrix $Y-\tau \cdot 1$ is invertible.

The Cayley map
    \begin{equation} \label{eq:Cayley_gl_n}
         \fc_\sigma: \gl_{n+1}^\tau \to \rS^\sigma, \quad Y \mapsto -\sigma \frac{1+\tau^{-1}Y}{1-\tau^{-1}Y}
    \end{equation}
defines a $\GL_n$-equivariant isomorphism between $\gl_{n+1}^{\tau}$ and $\rS^\sigma$. The open subsets $\gl_{n+1}^\tau$ and $\rS^\sigma$ is $\GL_n$ invariant and descends to open subsets of $\cB^\tau$ of $\cB$ and $\rB^{\sigma}$ of $\rB$ respectively, and $\fc_\sigma$ induces an isomorphism $\cB^\tau \to \rB^{\sigma}$ which we will still denote by $\fc_\sigma$.

Write an element $x$ of $\rS$ as
    \[
        x = \begin{pmatrix}
        A & b \\ c & d
             \end{pmatrix},
    \]
    where $A,b,c,d$ has size $n \times n,n \times 1,1 \times n,1 \times 1$ respectively. We put 
    \begin{equation} \label{eq:definition_Delta+}
        \Delta^+(x)=\det(b,Ab,A^2b,\cdots,A^{n-1}b), \quad \Delta^-(x)=\det \begin{pmatrix}
            c \\ cA \\ \cdots \\ cA^{n-1}
        \end{pmatrix}.
    \end{equation}

Let $\rS_+$ (resp. $\rS_-$) be the Zariski open subset of $\rS$ where $\Delta^+$ (resp. $\Delta^-$) is non vanishing, and let $\rG'_+$ and $\rG'_-$ be the preimage of $\rS_+$ and $\rS_-$ under the map $\alpha$ (see ~\eqref{eq:the_map_alpha}) respectively. Let $e_{n+1}$ be the column vector $(0,\cdots,0,1)^t$ of size $(n+1) \times 1$, we can directly check that
\begin{equation} \label{eq:reg equivalent}
    \Delta^+(x) = (-1)^n \det(x,xe_{n+1},\cdots,x^ne_{n+1}), \quad \Delta^-(x) = (-1)^n \det \begin{pmatrix}
        e_{n+1}^t \\ \cdots \\ e_{n+1}^t x^n
    \end{pmatrix}.
\end{equation}
Similarly, for $X = \begin{pmatrix}
    A & v \\ u & d 
\end{pmatrix} \in \gl_{n+1}$ with $A \in \gl_n, v \in F^n, u \in F_n,d \in F$, define
    \[
     \delta^+(X) = \det(v,Av,\cdots,A^{n-1}v), \quad \delta^-(X)=\det \begin{pmatrix}
         u \\ uA \\ \cdots \\uA^{n-1}
     \end{pmatrix}.
    \]
Let $\gl_{n+1,+}$ (resp. $\gl_{n+1,-}$) be the Zariski open subset of $\gl_{n+1}$ where $\delta^+$ (resp. $\delta^-$) is non-vanishing. We have
\begin{equation}  \label{eq:reg equivalent Lie algebra}
       \delta^+(X)=(-1)^n \det(e_{n+1},\cdots,X^n e_{n+1}), \quad \delta^-(X)=(-1)^n \det \begin{pmatrix}
        e_{n+1}^t \\ e_{n+1}^t X \\ \cdots \\ e_{n+1}^t X^n
    \end{pmatrix}.
\end{equation}
 
Direct computation shows that for $X \in \gl^\tau_{n+1}(F)$, we have
\begin{equation} \label{eq:Delta_and_Cayley_transform}
    \Delta^+(\fc_\sigma(X)) = (-2 \sigma \tau^{-1})^{\frac{n(n+1)}{2}} \det(1-\tau^{-1}Y)^{-n} \delta^+(X).
\end{equation}
and
\begin{equation} \label{eq:Delta-_and_Cayley_transform}
    \Delta^-(\fc_\sigma(X)) = (-2 \sigma \tau^{-1})^{\frac{n(n+1)}{2}} \det(1-\tau^{-1}Y)^{-n} \delta^-(X).
\end{equation}
As a consequence, for any $Y \in \gl_{n+1}^\tau(F)$, we have
    \begin{equation} \label{eq:non zero equivalent}
        \Delta^+(\fc_\sigma(Y)) \text{ (resp. }  \Delta^-(\fc_\sigma(Y))) \ne 0 \iff \delta^+(Y) \text{ (resp. } \delta^-(Y)) \ne 0 
    \end{equation}

For a vector space $V$ over $F$, put $\widetilde{\gl_V} := \gl(V) \times V \times V^*$, it carries a right action by $\GL(F)$:
    \[
        (A,v,u) \cdot g = (g^{-1}Ag,g^{-1}v,ug).
    \]
We write $\widetilde{\gl_n} := \widetilde{\gl_{F^n}}$. Note that we have a $\GL_n$-equivariant isomorphism
    \begin{equation} \label{eq:isomorphism gl and tilde gl}
        \gl_{n+1} \cong \widetilde{\gl_n} \times F, \quad \begin{pmatrix}
             A & v \\ u & d
        \end{pmatrix} \mapsto \left( (A,v,u),d \right)
    \end{equation}
where $F$ is endowed with the trivial action of $\GL_n$. Thus, we can view $\widetilde{\gl_n}$ as a variant of $\gl_{n+1}$ under the $\GL_n$-action.

Let $\cA_V = \widetilde{\gl_V}/\GL(V)$ be the GIT quotient. We put $\cA := \cA_{F^n}$, by ~\eqref{eq:isomorphism gl and tilde gl}, we have an identification $\cB \cong \cA \times F$. If $\dim V=d$, $\cA_V$ can be identified with the affine space $\mathbf{A}^{2d}$, where the quotient map $q:\widetilde{\gl_V} \to \cA_V$ is given by
    \begin{equation} \label{eq:gl_GIT_quotient}
        (A,u,v) \mapsto \left( (\mathrm{Trace} \wedge^i A)_{1 \le i \le d}, (uA^iv)_{0 \le i \le d-1} \right).
    \end{equation}
 
We see from this description that when $\dim V_1 = \dim V_2$, $\cA_{V_1}$ and $\cA_{V_2}$ are canonically identified. 

For $X=(A,v,u) \in \widetilde{\gl_n}$, we put
    \[
        \delta^+(X) = \det(v,Av,A^2v,\cdots,A^{n-1}v), \quad \delta^-(X) = \det(u,uA,uA^2,\cdots,uA^{n-1}).
    \]

Note that the definition of $\delta^{\pm}$ depends on the choice of basis of $F^n$, so there is no canonical definition of $\delta^{\pm}$ on $\widetilde{\gl_V}$ for a general vector space $V$. Let $\widetilde{\gl_V}_{,+}$ (resp. $\widetilde{\gl_V}_{,-}$) be the Zariski open subset of $\widetilde{\gl_V}$ where $\{v,Av,\cdots,A^{n-1}v\}$ form a basis of $V$ (resp. $\{u,uA,\cdots,uA^{n-1}\}$ form a basis of $V^*$). If $V=F^n$, this is the open subset of $\widetilde{\gl_n}$ where $\delta^+$ (resp. $\delta^-$) is non-vanishing. 

Under the map ~\eqref{eq:isomorphism gl and tilde gl}, we have
\[
    \gl_{n+1,+} = \widetilde{\gl_n}_{,+} \times F, \quad  \gl_{n+1,-} = \widetilde{\gl_n}_{,-} \times F.
\]    
Recall that if we have a reductive group $G$ over $F$ acting on a finite-type $F$-scheme $X$, then $x \in X(F)$ is called regular if its stabilizer $G_x$ has minimal dimension, and $x \in X(F)$ is called semisimple if its (scheme-theoretic) orbit is Zariski closed. We denote by $X_{\mathrm{reg}}$ (resp. $X_{\mathrm{rs}}$) the Zariski open subset of $X$ consisting of regular (resp. regular semisimple) elements. In this article, we are mainly interested in the following four cases of $X$ and $G$.
\begin{equation} \label{eq:four_cases}
    \begin{aligned}
        X&=\rG', &G=\rH_1 \times \rH_2; \\
        X&=\rS,  &G=\GL_n; \\
        X&=\gl_{n+1}, &G=\GL_n; \\
        X&=\widetilde{\gl_n}, &G=\GL_n. 
    \end{aligned}
\end{equation}
In these cases, a rational point is regular if and only if its stabilizer is trivial, and it is known that (see e.g. ~\cite{Zhang14}) $X_\mathrm{rs}=X_+ \cap X_-$.

For $g \in \GL_n$, and $x \in \rS(F)$ (resp. $Y \in \gl_{n+1}(F)$ or $Y \in \widetilde{\gl_n}(F)$), note that $\Delta^\pm(x \cdot g)=\det g \cdot \Delta^\pm(x)$ (resp. $\delta^\pm(Y \cdot g) = \det g \cdot \delta^\pm(Y)$). Therefore for $X$ and $G$ in these four cases ~\eqref{eq:four_cases}, we have $X_+ \cup X_- \subset X_{\mathrm{reg}}$. We remark that for $n \ge 2$, this inclusion is strict.

\subsection{Descent on $\widetilde{\gl_n}$} \label{subsec:descent}

Let $\fz_V := \{ (\lambda \cdot \mathrm{id},0,0) \mid \lambda \in F \} \subset \widetilde{\gl_V}$. Then $\fz_V$ is the center of $\widetilde{\gl_V}$ under the $\GL(V)$ action (i.e. the set of fixed points). The composition $\fz_V \hookrightarrow \widetilde{\gl_n} \to \cA$ induces a closed embedding $\fz_V \hookrightarrow \cA$. We call the image of this embedding the center of $\cA$, an element lying in the center will be called central.

We now recall the descent construction on $\widetilde{\gl_n}$, cf.   ~\cite{Zhang14}*{Section 3.2} or ~\cite{CZ21}*{Section 3.4}. It allows us to approximate any element $a \in \cA(F)$ by regular semisimple elements and central elements in smaller GIT quotients. The construction consists of two steps: first approximate $a$ by $(a_0,a^0)$ where $a_0$ is regular semisimple and $a^0$ is close to being central, and then approximate $a^0$ by central elements $(a_1,\cdots,a_k)$. We now describe them in detail.  

\paragraph{\textbf{Step 1}} For any vector space $V$ over $F$, We have a stratification of $\cA_V$ as described in ~\cite{CZ21}*{Paragraph 3.1.4}. For an integer $r$ with $1 \le r \le \dim V$, and $X=(A,v,u) \in \widetilde{\gl_V}(F)$, let
    \begin{equation} \label{eq:d_r_on_gl}
      d_r(X) = \det(uA^{i+j-2}v)_{1 \le i,j \le r}.
    \end{equation}
Let $\widetilde{\gl_V}^{(r)}$ be the locally closed subscheme of $\widetilde{\gl_V}$ such that $d_r \ne 0$ but $d_s = 0$ for $s>r$. The subscheme $\widetilde{\gl_V}^{(r)}$ is $\GL(V)$-invariant and descends to a locally closed subscheme $\cA_V^{(r)}$ of $\cA_V$. Note that $\cA_V^{(\dim V)}$ is the open subset $\cA_{V,\mathrm{rs}}$ consisting of regular semisimple elements.

Choose an isomorphism $i: F^n \cong F^r \oplus F^{n-r}$, it determines an embedding
    \begin{align} \label{eq:descent_1}
        \begin{split}
           \iota:  \widetilde{\gl_{r}}^{(r)} \times \widetilde{\gl_{n-r}} &\hookrightarrow \widetilde{\gl_n}, \\ 
            (A_1,v_1,u_1),(A_2,v_2,u_2) &\mapsto \left( \begin{pmatrix}
               A_1 & v'u_2 \\ v_2u' & A_2
           \end{pmatrix}, v_1,u_1 \right)
        \end{split}
    \end{align}
where $u' \in F_r$ satisfies
    \[
        u'A_1^iv_1 = \begin{cases}
            0 & \text{ if } 0 \le i < r-1 \\
            1 & \text{ if } i = r-1.
        \end{cases}
    \]
and $v' \in F^r$ satisfies
    \[
        u_1 A_1^iv' = \begin{cases}
            0 & \text{ if } 0 \le i < r-1 \\
            1 & \text{ if } i = r-1.
        \end{cases}
   \]
The map ~\eqref{eq:descent_1} descends to an isomorphism of varieties
    \begin{equation} \label{eq:descent step 1 isomorphism_plus}
    \iota: \cA_{F^r}^{(r)} \times \cA_{F^{n-r}} \cong \cA^{(\ge r)}
    \end{equation}

where $\cA^{(\ge r)} := \cup_{k \ge r}\cA^{(k)}$ is an open subset of $\cA$. The isomorphism $\iota$ is independent of the choice of $i$ and restricts to an isomorphism 
    \begin{equation} \label{eq:descent step 1 isomorphism}
         \cA_{F^r}^{(r)} \times \cA_{F^{n-r}}^{(s)} \cong \cA^{(r+s)}.
    \end{equation}

In particular, for $a \in \cA^{(r)}(F)$, there exists unique $(a_0,a^0) \in \cA_{F^r}^{(r)} \times \cA_{F^{n-r}}^{(0)}$ such that $\iota(a_0,a^0)=a$. Note that $a_0$ is regular semisimple in $\cA_{F^r}(F)$.

\paragraph{\textbf{Step 2}} Fix $a \in \cA^{(r)}(F)$. As we discussed in Step 1, $a$ can be uniquely written as $\iota(a_0,a^0)$. Assume that $a^0$is represented by $(A,v,u) \in \widetilde{\gl_{F^{n-r}}}(F)$, then $uA^iv=0$ for any $i \ge 0$ (see ~\cite{CZ21}*{3.1.4.1}). We therefore see that $a^0$ is determined by the characteristic polynomial $P$ of $A$. Assume $P$ decomposes in the polynomial ring $F[x]$ as $P=P_1^{n_1} \cdots P_k^{n_k}$, where $P_i$ are distinct and irreducible. For $1 \le i \le k$, put $F_i=F[x]/P_i(x)$, which is a finite extension of $F$. Put
    \begin{equation} \label{eq:H and h}
      H_i = \mathrm{Res}_{F_i/F} \GL_{n_i,F_i}, \quad H^0 = \prod_{i=1}^k H_i, \quad \widetilde{\fh}_i = \mathrm{Res}_{F_i/F}  \widetilde{\gl_{F_i^{n_i}}}, \quad \widetilde{\fh}^0 = \prod_{i=1}^k \widetilde{\fh}_i.
    \end{equation}
Let $\cA_{H^0}$ be the GIT quotient $\widetilde{\fh^0}/H^0$, which can be identified with $\prod_{i=1}^k \cA_i$ with 
\[
\cA_i = \mathrm{Res}_{F_i/F} (\widetilde{\gl_{F^i}}/\GL_{n_i,F_i}).
\] 

Choose an isomorphism $i^0: \bigoplus_{i=1}^k \mathrm{Res}_{F_i/F} F_i^{n_i} \cong F^{n-r}$. Then we get the embedding
    \begin{equation} \label{eq:embedding group}
      \iota_{H^0}:  H^0 \hookrightarrow \GL_{n-r}, \, (g_i) \mapsto \oplus g_i
    \end{equation}  
and
    \begin{equation} \label{eq:embedding Lie algebra}
      \iota_{\fh^0}: \widetilde{\fh^0} \hookrightarrow \widetilde{\gl_{F^{n-r}}}, \, (A_i,v_i,u_i) \mapsto (\oplus A_i,\oplus v_i,\oplus u_i),
    \end{equation}
where we have used the identification.
    \[
    \Hom_{F_i}(F_i^{n_i},F_i) \xrightarrow{\sim} \Hom_F( F_i^{n_i},F), \quad f \mapsto (v \mapsto \mathrm{tr}_{F_i/F} f(v))
    \]
The embedding ~\eqref{eq:embedding Lie algebra} induces a natural map $\cA_{H^0} \to \cA_{F^{n-r}}$ and this map is independent of the choice of $i^0$. 

Let $\widetilde{\fh^0}'$ denote the open subset of $\widetilde{\fh^0}$ consisting of $(A_i,u_i,v_i)$ such that $\det Q_i(A_j) \ne 0$, where $Q_i$ is the characteristic polynomial of $A_i$ over $F$. $\widetilde{\fh^0}'$ descends to an open subset $\cA'_{H^0}$ of $\cA_{H^0}$.

We put
\[
    H_0 = \GL_r, \quad H = H_0 \times H^0, \quad \widetilde{\fh_0} = \widetilde{\gl_{F^r}}, \quad \widetilde{\fh} = \widetilde{\fh_0} \times \widetilde{\fh^0}.
\]

The GIT quotient $\cA_H := \widetilde{\fh}/H$ can be identified with $\cA_{F^r} \times \cA_{H^0}$. The isomorphisms $i$ and $i^0$ determine the maps $\iota_{\fh^0}$ in ~\eqref{eq:embedding Lie algebra} and $\iota$ in \eqref{eq:descent_1}, and in terms gives a map $\iota_\fh: \widetilde{\fh_0}^{(r)} \times \widetilde{\fh^0} \to \widetilde{\gl_n}$. Let $\widetilde{\fh}' = \widetilde{\fh_0}^{(r)} \times \widetilde{\fh^0}'$. Then $\widetilde{\fh}'$ is an open subset of $\widetilde{\fh}$, and descends to an open subset $\cA_H'$ of $\cA_H$. The map $\iota_\fh$ restricts to a map $\widetilde{\fh}' \to \widetilde{\gl_V}$, and descends to a map $\iota_{\cA}: \cA_H' \to \cA$. The map $\iota_{\cA}$ is independent of the isomorphisms $i$ and $i^0$ and by ~\cite{Zhang14}*{Appendix B} (see also ~\cite{CZ21}*{Subsection 6.4}), the map $\iota_{\cA}$ is \'{e}tale.

Let $\alpha_i$ be the image of $T$ in $F_i=F[T]/(P_i)$, and let $a_i$ be the image of $(\alpha_i \cdot \mathrm{id},0,0)$ under the quotient map $\widetilde{\gl_{V_i}} \to \cA_i$. Let $a_H := (a_0,a_1,\cdots,a_k) \in \cA_H$. From the definition, we see that $\iota_\cA(a_H) = a$.

By ~\cite{Zhang14}*{Appendix B} (see also ~\cite{CZ21}*{Subsection 6.4}), we have a Cartesian diagram
    \begin{equation} \label{diag:Luna slice}
           \begin{tikzcd}
\widetilde{\fh}' \times^H \GL_n \arrow[r] \arrow[d] & \widetilde{\gl_n} \arrow[d] \\
\cA_H' \arrow[r]                                & \cA                        
\end{tikzcd}
    \end{equation}
 where 
    \begin{itemize}
        \item The right vertical map is the quotient map, and the left vertical map is induced by the quotient map $\widetilde{\fh'} \to \cA_H$, trivial on the second component.
        \item The bottom horizontal map is $\iota_\cA$, the top horizontal map sends $(X,g)$ to $\iota_{\fh}(X) \cdot g$.
    \end{itemize}
By the vanishing of Galois cohomology, we have
    \begin{equation} \label{eq:vanishing_of_Galois_cohomology}
        (\widetilde{\fh}' \times^H \GL_n)(F) = \widetilde{\fh}'(F) \times^{H(F)} \GL_n(F).
    \end{equation}
Hence the map $\iota_{\fh}$ induces a natural bijection between the $H(F)$-orbit of  $\widetilde{\fh}_{a_H}$ and the $\GL_n(F)$-orbit of $\widetilde{\gl_n}_{,a}$. 

We finally remark that the above construction is compatible with base change. More precisely, if $K/F$ is a field extension (not necessarily algebraic) and for $a \in \cA_V(F)$, the above procedure gives $H,\widetilde{\fh}$ and the diagram ~\eqref{diag:Luna slice}. If we regard $a$ as an element of $\cA_V(K)=\cA_{V\otimes_F K}(K)$, assume that the above procedure gives $H_K,\widetilde{\fh_K}$ and the diagram ~\eqref{diag:Luna slice} over $K$. Then one check easily that $H_K = H \times_F K$ is the base change of $F$, and the diagram ~\eqref{diag:Luna slice} over $K$ is also a base change of the corresponding diagram over $F$. We will use this fact for $F$ is a number field and $K=F_v$ for a place $v$ of $F$.

\subsection{Classification of regular orbits}
\label{subsec:regular_orbits}

In either of the four cases in ~\eqref{eq:four_cases}, let $X/G$ be the GIT quotient, and $q:X \to X/G$ be the quotient map. For $b \in X/G(F)$, we denote by $X_b$ the fiber of $b$ (see Subsection ~\ref{subsec:general_notations}). We now use the results in Subsection ~\ref{subsec:descent} to classify the regular orbits. Let us first begin with some special regular orbits. 
\begin{lemma} \label{lem:unique orbit}
    In either of the four cases in ~\eqref{eq:four_cases}, for any $b \in X/G(F)$, $G(F)$ acts simply transitively on $X_b(F) \cap X_+(F)$ and $X_b(F) \cap X_-(F)$. In other words, there exists a unique orbit $O^+(b)$ (resp. $O^-(b)$) of $G(F)$ action on $X_b(F)$ such that this orbit is contained in $X_+(F)$ (resp. $X_-(F)$). 
\end{lemma}

\begin{proof} 
    We prove the result for $X_+$, and the proof for $X_-$ is then similar.

    When $X=\widetilde{\gl_n}$, the lemma follows from ~\cite{Zhang14b}*{Proposition 6.3} \footnote{In ~\cite{Zhang14b}*{Proposition 6.3}  such result is only stated for local field, but the proof works for any field of characteristic 0.}. By ~\eqref{eq:isomorphism gl and tilde gl}, it holds for $X=\gl_{n+1}$. For $X = \rS$, take $\sigma \in E$ such that $\mathrm{Nm}_{E/F}(\sigma)=1$ and $b \in \rB^\sigma(F)$. By ~\eqref{eq:non zero equivalent}, as a set with $\GL_n(F)$ action, $\rS_b(F) \cap \rS_+(F)$ is isomorphic to $\gl_{n+1,\fc_\sigma^{-1}(b)}(F) \cap \gl_{n+1,+}(F)$, therefore the lemma holds for $X=\rS$. Finally, since $\rG'/\rH_1 \times \rH_{2,n+1}$ is isomorphic as a $\GL_n$-variety to $\rS$, the lemma holds for $X=\rG'$. 
\end{proof}

If $b \in (X/G)(F)$ is regular semisimple, then $O^+(b)=O^-(b)=X_b$ (see ~\cite{Zhang12}). We also have the following result

\begin{lemma} \label{lem:regular_nilpotent_orbit}
 For $X=\widetilde{\gl_n}$ and $G=\GL_n$, and $a \in \cA(F)$ is central, then $O^+(a)$ and $O^-(a)$ are the only regular orbits of in $\widetilde{\gl_n}_{,a}$.
\end{lemma}

\begin{proof}
    Let $a$ be the image of $(\lambda \cdot \mathrm{id},0,0) \in \widetilde{\gl_n}(F)$. If $\lambda=0$, then this is proved in ~\cite{Zhang14b}*{Lemma 6.1}. For general $\lambda$, we only need to note that $A \mapsto (A- \lambda \cdot \mathrm{id},0,0)$ gives a $\GL_n$ equivariant bijection between $\widetilde{\gl_n}_{,a}$ and $\widetilde{\gl_n}_{,0}$.
\end{proof}

If $a \in \cA(F)$ is central, the orbit $O^+(a)$ and $O^-(a)$ can be described explicitly. If $a$ is the image of $(\lambda \cdot \mathrm{id},0,0)$ then $O^+(a)$ is the orbit of
\begin{equation} \label{eq:central orbit representative}
      Z^+_\lambda := \left(
    \begin{pmatrix}
        \lambda & 1 & 0 & \cdots & 0\\
        0 & \lambda & 1 & \cdots & 0\\
        0 & 0 & \lambda & \cdots & 0\\
        \cdots & \cdots & \cdots & \cdots & \cdots \\
        0 & 0 & 0 & \cdots & \lambda
    \end{pmatrix}, 
    \begin{pmatrix}
        0 \\ 0 \\ 0 \\ \cdots \\ 1
    \end{pmatrix},
    0
    \right).
\end{equation}
And $O^-(a)$ is the orbit of
    \begin{equation} \label{eq:central_orbit_representative_minus}
             Z^-_\lambda := \left( \begin{pmatrix}
            \lambda & 0 &   \cdots & 0 & 0 \\ 
            1 & \lambda &   \cdots & 0 & 0 \\
            0 & 1 &   \cdots & 0 & 0 \\
            \cdots & \cdots & \cdots & \cdots & \cdots \\
            0 & 0 & \cdots & 1 & \lambda
        \end{pmatrix}, 0, (0,\cdots,0,1)  \right).
    \end{equation}

Now we classify general regular orbits. We first consider the case where $X=\widetilde{\gl_n}$ and $G=\GL_n$. Let $a \in \cA(F)$, we use our notation in Subsection ~\ref{subsec:descent}. We have field extensions $F_i(1 \le i \le k)$ of $F$ and after choosing the isomorphism $i$ and $i^0$, we associate $a \in \cA(F)$ with a map $\iota_\fh: \widetilde{\fh} \to \widetilde{\gl_n}$.

Let $\cE$ be the set of maps from the set $\{1,2,\cdots,k\}$ to $\{+,-\}$. For $\varepsilon \in \cE$, we write $\varepsilon_i := \varepsilon(i)$.

\begin{proposition} \label{prop:regular_orbit_tilde_gl_n}
    Pick any $X_0 \in \widetilde{\fh_0}_{,a_0}(F)$. Then a complete set of representative of orbit of $\GL_n(F)$ action on $\widetilde{\gl_n}_{,a}(F) \cap \widetilde{\gl_n}_{,\mathrm{reg}}(F)$ is given by
    \begin{equation}  \label{eq:list_regular_orbit_tilde_gl} 
    \left\{\iota_\fh(X_0,Z^{\varepsilon_1}_{\alpha_1},\cdots,Z_{\alpha_k}^{\varepsilon_k}) \mid \varepsilon \in \cE \right\}, 
    \end{equation}
    where $\varepsilon$ runs through every element in $\cE$ and we recall that $\alpha_i$ is the image of $T$ in $F_i=F[T]/(P_i)$. Moreover \begin{enumerate}
        \item For each $\varepsilon \in \cE$, the orbit of $\iota_\fh(X_0,Z^{\varepsilon_1}_{\alpha_1},\cdots,Z^{\varepsilon_k}_{\alpha_k})$ is independent of $X_0$ and the choice of the isomorphisms $i$ and $i^0$.
        \item The orbit of $\iota_{\fh}(X_0,Z_{\alpha_1}^{+},\cdots,Z_{\alpha_k}^{+})$ (resp. $\iota_{\fh}(X_0,Z_{\alpha_1}^{-},\cdots,Z_{\alpha_k}^{-})$) is $O^+(a)$ (resp. $O^-(a)$).
    \end{enumerate}
\end{proposition}

\begin{proof}
    By the diagram ~\eqref{diag:Luna slice} and the equality ~\eqref{eq:vanishing_of_Galois_cohomology}, the map $\iota_\fh$ induces a $\GL_n(F)$-equivariant bijection
    \[
        \widetilde{\fh}_{a_H}(F) \times^{H(F)} \GL_n(F) \to \widetilde{\gl_n}_{,a}(F),
    \]
    therefore, $\iota_\fh$ induces a bijection between the $H(F)$ orbit on $\widetilde{\fh}_{a_H}(F)$ and $\GL_n(F)$ orbit on $\widetilde{\gl_n}_{_a}(F)$, and regular orbits corresponds to regular orbit. Since each component $a_i$ of $a_H$ is either regular semi-simple ($i=0$) or central ($i>0$), Lemma ~\ref{lem:regular_nilpotent_orbit} implies that the regular orbits in $\widetilde{\gl_n}_{,a}(F)$ are given by ~\eqref{eq:list_regular_orbit_tilde_gl}. Different choices of $i$ and $i^0$ will yield $\iota_\fh$ composed with an $\GL_n(F)$ action on $\widetilde{\gl_n}(F)$, which certainly does not affect the orbit, therefore (1) is proved. Part (2) follows from ~\cite{CZ21}*{Proof of Lemme 3.2.2.2 and Lemme 3.4.1.1}.
\end{proof}

For $\varepsilon \in \cE$, we call a regular orbit in $\widetilde{\gl_n}_{,a}(F)$ of \emph{type} $\varepsilon$, if it is the orbit of $\iota_\fh(X_0,Z_{\alpha_1}^{\varepsilon_1},\cdots,Z_{\alpha_k}^{\varepsilon_k})$. 

\begin{corollary} \label{coro:regular_orbits}
    In each of the cases in ~\eqref{eq:four_cases}, for $b \in X/G(F)$, there are finitely many regular orbits of $G(F)$-action on $X_b(F)$. Indeed, the number of orbits is a power of 2.
\end{corollary}

\begin{proof}
    The case when $X=\widetilde{\gl_n}$ is proved in Proposition ~\ref{prop:regular_orbit_tilde_gl_n}, the other cases are reduced to this case (see the proof Lemma ~\ref{lem:unique orbit}).
\end{proof}

\section{Coarse relative trace formulae} \label{sec:RTF}

\subsection{Preliminaries}   
     In this section, we denote by $E/F$ a quadratic extension of number fields. Let $\mathtt{c}$ be the non-trivial element in the Galois group $\Gal(E/F)$. For an $F$-algebra $R$, $\mathtt{c}$ will induce an involution on $R \otimes_F E$, we denote this by $a \mapsto a^\mathtt{c}$. Let $\bA:=\bA_F$ and $\bA_E$ be the ring of ad\`{e}les of $F$ and $E$ respectively.
    
      Let $k \ge 1$ be an integer, we write $\GL_k:=\GL_{k,F}$ for the general linear group of rank $k$ over $F$. Let $\rG'_k := \mathrm{Res}_{E/F} \GL_{k,E}$ and let $\rG':=\rG'_n \times \rG'_{n+1}$. $\rG'$ has two subgroups $\rH_1$ and $\rH_2$ as in ~\eqref{eq:subgroup G'}. We put $\rH_{2,n} = \GL_n, \rH_{2,n+1} = \GL_{n+1}$, they are both regarded as subgroups of $\rH_2$. For an element $h_2 \in \rH_2$, we write $h_{2,n}$ and $h_{2,n+1}$ for its corresponding components. Let $K_1$ and $K_2$ be the standard maximal compact subgroup of $\rH_1(\bA)$ and $\rH_2(\bA)$ respectively.

     For $k \ge 1$, let $B_k$ be the upper triangular Borel subgroup of $\rG'_k$. We choose $B_k$ as our fixed minimal parabolic subgroup of $\rG'_k$, and we choose $B_n \times B_{n+1}$ as our fixed minimal parabolic subgroup of $\rG'$. We put $\fa_k:=\fa_{B_{k}}$. Let $B_{k,F}:=B_k \cap \GL_n$ be the upper triangular Borel subgroup of $\GL_n$.
    
    Let $k,a,b$ be integers such that $k \ge 1$ and $0 \le a <b \le k$, let $e_1,\cdots,e_{k}$ be the standard basis of $E^{k}$ and let $E^{a,b}$ be the subspace of $E^k$ generated by $e_{a+1},\cdots,e_b$.
    
    Let $\cF_{\mathrm{RS}}$ be the set of Rankin-Selberg parabolic subgroup of $\rG'$ introduced in ~\cite{BPCZ}*{Subsection 3.1}. It consists of semi-standard parabolic subgroups of $\rG'$ of the form $P_n \times P_{n+1}$, such that $P_n$ is standard and $P_{n+1} \cap \rG_n' = P_n$. When $P_n$ is the stabilizer of the flag
    \begin{equation} \label{eq:flag}
        0=V_0 \subset V_1 \subset \cdots \subset V_r = E^n
    \end{equation} 
    with $V_i=E^{0,a_i}$, then there are $2r+1$ possible choices such that $P_n \times P_{n+1} \in \cF_{\mathrm{RS}}$. It is either the stabilizer of the flag
        \begin{equation} \label{eq:flag 1}
            0=V_0 \subset \cdots \subset V_k \subset V_{k+1} \oplus Ee_{n+1} \subset \cdots \subset V_r \oplus Ee_{n+1}=E^{n+1}
        \end{equation}
    with $0 \le k \le r-1$ or the stabilizer of the flag
        \begin{equation} \label{eq:flag 2}
          0=V_0 \subset \cdots \subset V_k \subset V_k  \oplus Ee_{n+1} \subset \cdots \subset V_r \oplus Ee_{n+1}=E^{n+1}
        \end{equation}
    with $0 \le k \le r$.

    In the case when $P_{n+1}$ stabilizes the flag ~\eqref{eq:flag 1}, the standard Levi subgroup $M_{P_{n+1}}$ of $P_{n+1}$ can be identified with
        \[
            \prod_{\substack{0 \le i \le r-1 \\i \ne k}} \GL(E^{a_i,a_{i+1}}) \times \GL(E^{a_k,a_{k+1}}+Ee_{n+1}).
        \]
    Let $M_{P_{n+1}}(\bA)^{\mathbbm{1}}$ be the subgroup of $M_{P_{n+1}}(\bA)$ consisting of elements
        \[
            \prod_{\substack{0 \le i \le r-1 \\i \ne k}} \GL(E^{a_i,a_{i+1}})(\bA)^1 \times \GL(E^{a_k,a_{k+1}}+Ee_{n+1})(\bA)
        \]
    under the above identification. We denote by $\rG'_{n+1}(\bA)_{P_{n+1}}^\mathbbm{1}$ the product $M_{P_{n+1}}(\bA)^\mathbbm{1} N_{P_{n+1}}(\bA)$.

    In the case when $P_{n+1}$ stabilizes the flag ~\eqref{eq:flag 2}, $M_{P_{n+1}}$ can be identified with
      \[
            \prod_{0 \le i \le r} \GL(E^{a_i,a_{i+1}}) \times \GL(Ee_{n+1}).
    \]
     Let $M_{P_{n+1}}(\bA)^{\mathbbm{1}}$ be the subgroup of $M_{P_{n+1}}(\bA)$ consisting of elements
        \[
            \prod_{0 \le i \le r}  \GL(E^{a_i,a_{i+1}})(\bA)^1 \times \GL(Ee_{n+1})(\bA)
        \]
    under the identification above and we also set $\rG'_{n+1}(\bA)_{P_{n+1}}^\mathbbm{1} = M_{P_{n+1}}^{\mathbbm{1}} N_{P_{n+1}}(\bA)$.

    For $P \in \cF_\mathrm{RS}$, we write $P_{\rH_1} := P \cap \rH_1 \cong P_n$ and $P_{\rH_2} = P \cap \rH_2$. Let $\cF_{\mathrm{RS},F}$ be the set of semi-standard parabolic subgroups of $\rH_2$ of the form $P_n \times P_{n+1}$, such that $P_n$ is standard and $P_{n+1} \cap \GL_n = P_n$. The map $P \mapsto P \cap \rH_2$ induces a bijection between $\cF_{\mathrm{RS}}$ and $\cF_{\mathrm{RS},F}$

   The natural map $\rG'_n \hookrightarrow \rG'_{n+1}$ induces a map $\fa_n \hookrightarrow \fa_{n+1}$. For $P \in \cF_{\mathrm{RS}}$, it induces an embedding $\fa_{P_n} \hookrightarrow \fa_{P_{n+1}}$. We define
    \[
        \fa_{P,\rH_1} := \fa_{P_n} \cap \fa_{P_{n+1}} \cong \fa_{P_{\rH_1}} \cap \fa_P.
    \]
    The natural map $\iota:\fa_{P,\rH_1} \to \fa_{P_{n+1}}^{\rG'_{n+1}}$ induced by $\fa_{P,\rH_1} \hookrightarrow \fa_{P_{n+1}} \twoheadrightarrow \fa_{P_{n+1}}^{\rG'_{n+1}}$ is an isomorphism.

    The subspace $\fa_{P,\rH_1}$ is the Lie algebra of $A_{P,\rH_1}^\infty := A_{P_n}^\infty \cap A_{P_{n+1}}^\infty \cong A_{P_{\rH_1}}^\infty \cap A_P^\infty$. Note that we have the decomposition
    \[
    M_{P_{n+1}}(\bA) = M_{P_{n+1}}^\mathbbm{1}(\bA) \times A_{P,\rH_1}^\infty, \quad \rG'_{n+1}(\bA) = \rG'_{n+1}(\bA)^\mathbbm{1}_{P_{n+1}} \times A_{P,\rH_1}^\infty.
    \]

    For $s \in \C$, define $s \det \in \fa_{P,\rH_1,\C}^*$ by requiring the expression
    \[
        \langle s \det, H_{P_n}(a) \rangle = \lvert \det a \rvert^s
    \]
holds for all $A_{Q,\rH_1}^\infty \subset \rG'_n(\bA)$. Define $\rho_{P_n}$ (resp. $\rho_{P_{n+1}}$) $\in \fa_{P,\rH_1}^*$ be the element pulling back from $\rho_{P_n} \in \fa_{P_n}^{\rG_n',*}$ (resp. $\rho_{P_{n+1}} \in \fa_{P_{n+1}}^{\rG_{n+1}',*}$). Let $\underline{\rho}_{P,s}$ be the element in $\fa_{P,\rH_1,\C}^*$ defined by
    \[
        \underline{\rho}_{P,s} := \rho_{P_{n+1}} - \rho_{P_n} + s \det.
    \]
We also regard $\underline{\rho}_{P,s}$ as an element of $ \fa_{P_{n+1},\C}^{\rG'_{n+1}}$ via the isomorphism $\iota$. Write $c_P \in \R$ for the Jacobian of $\iota$.

Let $V$ be a finite dimensional vector space over $\R$, define $V^*_\C:=\Hom_\R(V,\C)$. An exponential polynomial function on $V$ is a function of the form
    \[
       p(x) = \sum_{i=1}^r P_i(x)e^{\langle \lambda_i,x \rangle}, \quad x \in V,
    \]
where $\lambda_i \in V^*_\C$ and $P_i$ are non-zero polynomial functions on $V$. If $\lambda_i$ are distinct, then $P_i$ are uniquely determined. The term corresponding to $\lambda_i=0$ is called the \emph{pure polynomial term}. The complex numbers $\lambda_i$ are called exponents of $p$.

Let $\eta_{E/F}:\bA^\times \to \{ \pm 1 \}$ be the quadratic character associated to $E/F$. By an abuse of notation, we define a character $\eta$ on $\rH_2(
\bA)$ by
    \[
        \eta(h_{2,n},h_{2,n+1}) = \eta_{E/F}(h_{2,n})^{n+1} \eta_{E/F}(h_{2,n})^n.
    \]

\subsection{Coarse Jacquet-Rallis RTF}

For $f \in \cS(\rG'(\bA))$ and $P \in \cF_\mathrm{RS}$, the right translation action $\mathrm{R}(f)$ on $L^2([\rG']_P)$ is given by the kernel function
    \[
        K_{f,P}(x,y) = \sum_{m \in M_P(F)} \int_{N_P(\bA)} f(x^{-1}mny) \rd n,
    \]
where $x,y \in [\rG']_P$. For $\chi \in \fX(\rG')$, let $K_{f,P,\chi}(x,y)$ be the kernel function of the operator $p_\chi \circ R(f)$ on $L^2([\rG']_P)$, where $p_\chi$ denotes the projection to $\chi$-component in the Langlands spectral decomposition ~\eqref{eq:spectral decomposition}. For $(h_1,h_2) \in [\rH_1]_{P_{\rH_1}} \times [\rH_2]_{P_{\rH_2}}$ we put
    \[
        K_{f,P,b}(h_1,h_2) = \sum_{m \in M_P(F) \cap \rG'_b(F)} \int_{N_P(\bA)} f(h_1^{-1}mnh_2),
    \]
where we recall that $\rG'_b$ is the fiber of $b$ (see Subsection ~\ref{subsec:general_notations}).

For each $\bullet \in \fX(\rG') \sqcup \rB(F)$ and $T \in \fa_0$, the modified kernel is defined as
    \[
        K_{f,\bullet}^T(h_1,h_2) = \sum_{P \in \cF_{\mathrm{RS}}} \epsilon_P \sum_{\substack{\gamma \in P_{\rH_1}(F)\backslash \rH_1(F) \\ \delta \in P_{\rH_2}(F) \backslash \rH_2(F)}} \widehat{\tau}_{P_{n+1}}(H_{P_{n+1}}(\delta_n h_{2,n})-T_{P_{n+1}}) K_{f,P,\bullet}(\gamma h_1,\delta h_2),
    \]
where $(h_1,h_2) \in [\rH_1]  \times [\rH_2]$. More generally, for $Q \in \cF_{\mathrm{RS}}$ and $\bullet \in \fX(\rG') \sqcup \rB(F)$, we put
    \[
        K_{f,\bullet}^{Q,T}(h_1,h_2) = \sum_{P \in \cF_{\mathrm{RS}}} \epsilon_P^Q \sum_{\substack{\gamma \in P_{\rH_1}(F)\backslash Q_{\rH_1}(F) \\ \delta \in P_{\rH_2}(F) \backslash Q_{\rH_2}(F)}} \widehat{\tau}_{P_{n+1}}^{Q_{n+1}}(H_{P_{n+1}}(\delta_n h_{2,n})-T_{P_{n+1}}) K_{f,P,\bullet}(\gamma h_1,\delta h_2),
    \]
where $(h_1,h_2) \in [\rH_1]_{Q_{\rH_1}} \times [\rH_2]_{Q_{\rH_2}}$.

 We have the following asymptotic of modified kernel in ~\cite{BPCZ}*{Theorem 3.3.7.1}: for any $N>0$, there exists a continuous semi-norm $\| \cdot \|$ on $\cS(\rG'(\bA))$ such that
     \begin{equation} \label{eq:asymptotic_spectral}
           \sum_{\chi \in \fX(\rG')} \left| K_{f,\chi}^T(h_1,h_2) - F^{\GL_{n+1}}(h_{2,n},T)K_{f,\chi}(h_1,h_2) \right| \le e^{-N\|T\|} \|h_1\|_{H_1}^{-N} \|h_2\|_{H_2}^{-N} \|f\|
     \end{equation}
holds for all $f \in \rS(\rG'(\bA)), (h_1,h_2) \in [\rH_1] \times [\rH_2]$ and $T \in \fa_{n+1}$ sufficiently positive.

In Proposition ~\ref{prop:asympototic Q}, we generalize the result to any $K_{f,\chi}^{Q,T}$, we showed that for every $Q \in \cF_{\mathrm{RS}}$ and $N>0$, there exists a continuous semi-norm $\| \cdot \|$ on $\cS(\rG'(\bA))$ such that
        \begin{equation}
            \sum_{\chi \in \fX(\rG')} \lvert K_{f,\chi}^{Q,T}(h_1,h_2) - F^{Q_{n+1}}(h_{2,n},T_{Q_{n+1}}) K_{f,Q,\chi}(h_1,h_2) \rvert \le e^{-N\|T\|} \|h_1\|_{Q_{\rH_1}}^{-N} \|h_2\|_{Q_{\rH_2}}^{-N} \|f\|.
        \end{equation}
  holds for $f \in \cS(\rG'(\bA)),(h_1,h_2) \in [\rH_1]_{Q_{\rH_1}} \times [\rH_2]^\mathbbm{1}_{Q_{\rH_2}}$ and $T \in \fa_{n+1}$ sufficiently positive.

We also showed in Proposition ~\ref{prop:asymptotic Q geometric} that for any $f \in \cS(\rG'(\bA))$ and $N>0$, we have
    \begin{equation} \label{eq:asympototic Q geometric}
       \sum_{b \in \rB(F)} \lvert K_{f,b}^{Q,T}(h_1,h_2) - F^{Q_{n+1}}(h_{2,n},T_{Q_{n+1}}) K_{f,Q,b}(h_1,h_2) \rvert \ll e^{-N\|T\|} \|h_1\|_{Q_{\rH_1}}^{-N} \|h_2\|_{Q_{\rH_2}}^{-N} .
     \end{equation}

We say that an id\`{e}le class character $\xi:\bA_E^\times \to \C^\times$ is strictly unitary if it is trivial on the subgroup $A_{\mathbb{G}_m}^\infty \subset \bA_E^\times$. Note that any id\`{e}le class character can be uniquely written as $\xi \cdot | \cdot |^s$, where $\xi$ is strictly unitary and $s \in \C$.

The next theorem slightly generalizes the result in ~\cite{BPCZ}*{Section 3} by allowing a general character on $\rH_1(\bA)$. 

\begin{theorem} \label{thm:coarse Bessel}
    Let $\xi$ be a strictly unitary character of $\bA_E^\times$.
    \begin{enumerate}
        \item For any $f \in \cS(\rG'(\bA))$, $s \in \C$ and $T$ sufficiently positive, we have
            \begin{equation} \label{eq:spectral convergence}
                 \sum_{\chi \in \fX(\rG')} \int_{[\rH_1]} \int_{[\rH_2]}  \lvert K_{f,\chi}^T(h_1,h_2)  \rvert  \lvert \det h_1 \rvert^{\mathrm{Re}(s)}  \rd h_1 \rd h_2 < \infty
            \end{equation}
        and
            \begin{equation} \label{eq:geometric convergence}
                  \sum_{b \in \rB(F)} \int_{[\rH_1]} \int_{[\rH_2]}  \lvert K_{f,b}^T(h_1,h_2)  \rvert \lvert \det h_1 \rvert^{\mathrm{Re}(s)}  \rd h_1 \rd h_2 < \infty.
            \end{equation}
    \item For $\bullet \in \fX(\rG') \sqcup \rB(F)$, as a function of $T$, the integral
            \[
                I_\bullet^T(f,\xi,s) := \int_{[\rH_1]} \int_{[\rH_2]} K^T_{f,\bullet}(h_1,h_2) \xi(h_1) \lvert \det h_1 \rvert^s \eta(h_2) \rd h_1 \rd h_2
            \]
          is an exponential polynomial. If $s \not \in \{-1,1\}$, then the pure polynomial term is a constant, we denoted it by $I_\bullet(f,\xi,s)$. For a fixed $f$ and $\xi$, $I_\bullet(f,\xi,s)$ is meromorphic on $\C \setminus \{-1,1\}$. We write $I_\bullet(f,\xi):=I_\bullet(f,\xi,0)$.
    \item For each $\chi \in \fX(\rG')$ and $s \not \in \{-1,1\}$, the distribution $I_\chi(\cdot,\xi,s)$ is continuous and the sum
        \[
            I(f,\xi,s) := \sum_{\chi \in \fX(\rG')} I_\chi(f,\xi,s)
        \]
    is absolutely convergent and defines a continuous distribution $I(\cdot,\xi,s)$ on $\cS(\rG'(\bA))$.
    \item For each $b \in \rB(F)$ and $s \not \in \{-1,1\}$, the distribution $I_b(\cdot,\xi,s)$ is continuous on $\cS(\rG'(\bA))$ and
        \[
            I(f,\xi,s) = \sum_{b \in \rB(F)} I_b(f,\xi,s).
        \]
    Where the sum on the right-hand side is absolutely convergent.
    \end{enumerate}
\end{theorem}

The proof is similar to ~\cite{BPCZ}*{Section 3}, for later use, we briefly sketch the proof here.

\begin{proof}[Proof of Part (1)]
    Since $|h_1|^{\mathrm{Re}(s)} \ll \| h_1 \|_{\rH_1}^{ \left| \mathrm{Re}(s) \right|}$.  By ~\cite{BP21}*{Theorem A.1.1(vi)}, for $N$ large enough
        \[
            \int_{[\rH_1]} \int_{[\rH_2]} e^{-N\|T\|} \|h_1\|^{-N}_{H_1} \|h_2\|_{H_2}^{-N} |\det h_1|^s \rd h_1 \rd h_2
        \]
    is finite and defines continuous semi-norm on $\cS(\rG'(\bA))$.

    By ~\cite{BPCZ}*{(3.2.3.2)} for any $f \in \cS(\rG'(\bA))$ and any $(h_1,h_2) \in [\rH_1] \times [\rH_2]$, we have,
        \[
            \sum_{\chi \in \fX(\rG')}  \left| K_{f,\chi}(h_1,h_2) \right| |\det h_1|^{\mathrm{Re}(s)} \ll \|h_{2,n}\|_{\GL_n}^{N} \|h_{2,n+1}\|_{\GL_{n+1}}^{-N} \|h_1\|_{\rH_1}^{-N+|\mathrm{Re}(s)|}.
        \]
    Since $F^{\GL_{n+1}}(\cdot,T)$ is compactly supported as a function on $[\GL_n]$, we see that
        \[
           \int_{[\rH_1]} \int_{[\rH_2]} \sum_{\chi \in \fX(\rG')}  F^{\GL_{n+1}}(h_{2,n},T) \left| K_{f,\chi}(h_1,h_2) \right| \rd h_1 \rd h_2 
        \]
    is finite. Combining with ~\eqref{eq:asymptotic_spectral} ,we see that the integral ~\eqref{eq:spectral convergence} is finite. Using the estimate ~\eqref{eq:asympototic Q geometric}, the finiteness of ~\eqref{eq:geometric convergence} is proved in a similar way.
\end{proof}

For each $Q \in \cF_\mathrm{RS}$, let $\Gamma'_{Q_{n+1}}$ be the function on $\fa_{Q_{n+1}}^{\rG'_{n+1}} \times \fa_{Q_{n+1}}^{\rG'_{n+1}}$ defined in ~\cite{Arthur81}*{Section 2}. It is compactly supported in the first variable when the second variable stays in a compact subset.

By ~\cite{Zydor20}*{Lemme 3.5}, for any $Q \in \cF_{\mathrm{RS}}$, the function
    \[
        p_{Q,s}: X \mapsto \int_{\fa_{Q_{n+1}}^{\rG'_{n+1}}} e^{\langle \underline{\rho}_{Q,s},H \rangle} \Gamma'_{Q_{n+1}}(H,X)  \rd H, \quad X \in \fa_{n+1}
    \]
is an exponential polynomial function and when $s \not \in \{-1,1\}$, the pure polynomial term is constant.

\begin{proof}[Proof of Part(2)-(4)]
    For $Q \in \cF_{\mathrm{RS}}$ and $(h_1,h_2) \in [\rH_1]_{Q_{\rH_1}} \times [\rH_2]_{Q_{\rH_2}}$, we have the following equality
        \[
    K_{f,\bullet}^{T}(h_1,h_2) = \sum_{Q \in \cF_\mathrm{RS}} \sum_{\substack{\gamma \in Q_{\rH_1}(F) \backslash \rH_1(F) \\ \delta \in Q_{\rH_2}(F) \backslash \rH_2(F) }} \Gamma'_{Q_{n+1}}(H_{Q_{n+1}}(\delta_n h_{2,n})-T'_{Q_{n+1}},T_{Q_{n+1}}-T'_{Q_{n+1}}) K_{f,\bullet}^{Q,T}(\gamma h_1,\delta h_2).
        \]
    We hence see that $I^T_{\bullet}(f,\xi,s)$ is equal to
        \[
    \sum_{Q \in \cF_{\mathrm{RS}}} \int_{[\rH_1]_{Q_{\rH_1}}} \int_{[\rH_2]_{Q_{\rH_2}}} \Gamma'_{Q_{n+1}}\left(H_{Q_{n+1}}(h_{2,n})-T'_{Q_{n+1}},T_{Q_{n+1}}-T'_{Q_{n+1}}\right) K_{f,\bullet}^{Q,T}(h_1,h_2) \xi(h_1) \lvert \det h_1 \rvert^s \eta(h_2) \rd h_1 \rd h_2.
        \]
    Using Iwasawa decomposition, the summand corresponding to $Q$ is
        \begin{align*} 
          \int_{[M_{Q_{\rH_1}}]} \int_{[M_{Q_{\rH_2}}]} \int_{K_1 \times K_2} &e^{\langle -2\rho_{Q_{\rH_1}},H_{Q_{\rH_1}}(m_1) \rangle} e^{\langle -2\rho_{Q_{\rH_2}},H_{Q_{\rH_2}}(m_2) \rangle} 
          \Gamma'_{Q_{n+1}} \left( H_{Q_{n+1}}(m_{2,n})-T'_{Q_{n+1}},T_{Q_{n+1}}-T'_{Q_{n+1}} \right) \\
          &K_{f,\bullet}^{Q,T}(m_1k_1,m_2k_2) \xi(m_1k_1) \lvert  \det m_1 \rvert^s \eta(m_2 k_2) \rd m_1 \rd m_2 \rd k_1 \rd k_2.
        \end{align*}
    Define $f_Q \in \cS(M_Q(\bA)))$ by
        \[
            f_Q(m)=e^{\langle \rho_Q,H_Q(m) \rangle} \int_{K_1 \times K_2} \int_{N_Q(\bA)} f(k_1^{-1}mnk_2) \xi(k_1) \eta(k_2) \rd k_1 \rd k_2 \rd n.
        \]
    Then one readily check that for all $(m_1,m_2) \in [M_{Q_{\rH_1}}] \times [M_{Q_{\rH_2}}]$, we have
        \[
            \int_{K_1} \int_{K_2} K^{Q,T}_{f,\bullet}(m_1k_1,m_2k_2) \xi(k_1) \eta(k_2) \rd k_1 \rd k_2 = e^{\langle \rho_Q,H_Q(m_1) + H_Q(m_2) \rangle} K_{f_Q,\bullet}^{M_Q,T}(m_1,m_2).
        \]
    Let $M_1:=M_{Q_{\rH_1}}$ and $M_2:=M_{Q_{\rH_2}}$. We obtain
    \begin{equation} \label{eq:I^T_expression}
           \begin{aligned}
               I_\bullet^T(f,\xi,s) = \sum_Q \int_{[M_1]} \int_{[M_2]} &e^{\langle \underline{\rho}_{Q},H_{Q_{\rH_1}}(m_1) \rangle} \Gamma'_{Q_{n+1}} \left( H_{Q_{n+1}}(m_{2,n})-T'_{Q_{n+1}},T_{Q_{n+1}}-T'_{Q_{n+1}} \right) \\
               &K_{f_Q}^{M_Q,T'}(m_1,m_2) \xi(m_1) \lvert m_1 \rvert^s \eta(m_2) \rd m_1 \rd m_2 \\
               =\sum_{Q \in \cF_{\mathrm{RS}}} c_Q  p_{Q,s}(T_{Q_{n+1}}-T'_{Q_{n+1}})& \int_{A_{Q,\rH_1}^\infty \backslash [M_1] \times [M_2]}  e^{ \langle \underline{\rho}_{Q,s},-H_{Q_{n+1} }(m_{2,n})+T'_{Q_{n+1}} 
               \rangle}  \\
               &K_{f_Q}^{M_Q,T'}(m_1,m_2) \xi(m_1) \lvert m_1 \rvert^s \eta(m_2) \rd m_1 \rd m_2.
        \end{aligned}
    \end{equation}
    Using Proposition ~\ref{prop:asymptotic M}, the last integral in ~\eqref{eq:I^T_expression} is finite as the same argument of part (1), hence (2) is proved. Part (3) and (4) then follow from the expression of $I^T_\bullet(f,\xi,s)$ in ~\eqref{eq:I^T_expression}.
\end{proof}

\subsection{Coarse infinitesimal Jacquet-Rallis RTF}

There is also a Lie algebra version of the Jacquet-Rallis relative trace formula formulated by Zydor in ~\cite{Zydor18}. Let $P=MN \in \cF_{\mathrm{RS},F}$ , we put $\fp$ (resp. $\fm$, $\fn$) be the Lie algebra of $P_{n+1}$ (resp. $M_{P_{n+1}}$, $N_{P_{n+1}}$). They are Lie subalgebras of $\gl_{n+1}$. For  $\varphi \in \cS(\gl_{n+1}(\bA))$ and $P \in \cF_{\mathrm{RS}}$, we put a kernel function $K_{\varphi,P}$ on $[\GL_{n+1}]_{P_{n+1}}$ by
\begin{equation} \label{eq:defi_K_varphi,P}
  K_{\varphi,P}(g) = \sum_{M \in \fm(F)} \int_{\fn(\bA)} \varphi((M+N) \cdot g) \rd N, \quad g \in [\GL_{n+1}]_{P_{n+1}}
\end{equation}
For $a \in \cB(F)$, we put
\[
    K_{\varphi,P,a}(g) =  \sum_{M \in \fm(F) \cap \gl_{n+1,a}(F)} \int_{\fn(\bA)} \varphi((M+N) \cdot g) \rd N, \quad g \in [\GL_{n+1}]_{P_{n+1}}.
\]

For $T \in \fa_0$ and $a \in \cB(F)$, we define a modified kernel by
\[
    K_{\varphi,a}^T(g) = \sum_{P \in \cF_{\mathrm{RS}}} \varepsilon_P \sum_{\gamma \in P_n(F) \backslash \GL_{n}(F)} \widehat{\tau}_{P_{n+1}}(H_{P_{n+1}}(\gamma g)-T_{P_{n+1}}) K_{\varphi,P}(\gamma g), \quad g \in [\GL_n]
\]
More generally, for $Q \in \cF_{\mathrm{RS},F}$ we put
\[
    K_{\varphi,a}^{T,Q}(g) = \sum_{P \in \cF} \varepsilon_P^Q \sum_{\gamma \in P_n(F) \backslash Q_n(F)} \widetilde{\tau}_{P_{n+1}}^{Q_{n+1}}(H_{P_{n+1}}(\gamma g)-T_{P_{n+1}}) K_{\varphi,P,a}(\gamma g), \quad g \in [\GL_n]_{Q_n}.
\]
The following result is a slight generalization of the main theorem in ~\cite{Zydor18}.
\begin{theorem}[Zydor] \label{thm:coarse_Lie}
    Let $\xi$ be a strictly unitary character of $\bA_E^\times$.
    \begin{enumerate}
        \item For any $\varphi \in \cS(\gl_{n+1}(\bA))$, $s \in \C$ and $T \in \fa_0$ sufficiently positive, we have
        \[
            \sum_{a \in \cB(F)} \int_{[\GL_n]} \lvert K_{\varphi,a}^T(g) \rvert \lvert \det g \rvert^s \rd g < \infty.
        \]
        \item For any $a \in \cB(F)$, as a function of $T$, the integral
        \[
        I^T_a(f,\xi,s) := \int_{[\GL_n]} K^T_{\varphi,a}(g) \xi(g) \eta(g) \lvert \det g \rvert^s  \rd g
        \]
        is an exponential polynomial. If $s \not \in \{-1,1\}$, then the pure polynomial term is a constant, denoted by $I_a(\varphi,\xi,s)$. For a fixed $\varphi$ and $\xi$, $I_a(\varphi,\xi,s)$ is meromorphic on $\C \setminus \{-1,1\}$. We put $I_a(\varphi,\xi):=I(\varphi,\xi,0)$.
        \item For each $a \in \cB(F)$ and $s \not \in \{-1,1\}$ the distribution $I_a(\cdot,\xi,s)$ on $\cS(\gl_{n+1}(\bA))$ is continuous and the sum
        \[
            I(\varphi,\xi,s) := \sum_{a \in \cB(F)} I_a(\varphi,\xi,s)
        \]
        is absolutely convergent and defines a continuous distribution $I(\cdot,\xi,s)$ on $\cS(\gl_{n+1}(\bA))$.
    \end{enumerate}
\end{theorem}

\begin{proof}
    We proved in Appendix ~\ref{sec:Asymptotic of modified kernel} the asymptotic of the modified kernel in the Lie algebra case. (See Proposition ~\ref{prop:asympototic_Lie_algebra}). Using this, the proof is identical to the proof of Theorem ~\ref{thm:coarse Bessel}.
\end{proof}

\begin{remark} \label{rmk:trace_formula_for_tilde_gl_n}
    There is also a version of infinitesimal Jacquet-Rallis RTF on $\widetilde{\gl_n}$. Let $P=MN \in \cF_{\mathrm{RS},F}$, we let $\widetilde{\fm}$ and $\widetilde{\fn}$ be the intersection of $\fm$ and $\fn$ with $\widetilde{\gl_n}$ under the identification ~\eqref{eq:isomorphism gl and tilde gl}. Then for $\varphi \in \cS(\widetilde{\gl_n}(\bA))$, we can define $K_{\varphi,P}$ in the same way as ~\eqref{eq:defi_K_varphi,P} replacing $\fm,\fn$ by $\widetilde{\fm}$ and $\widetilde{\fn}$. We then define modified kernels $K_{\varphi,a}^T$ for any $a \in \cA(F)$ using the same formula; analogs of Theorem ~\ref{thm:coarse_Lie} hold in this setting.

    The theorem also directly generalizes to products of $\widetilde{\gl_n}$.  
\end{remark}
    
\section{Global Theory I: The case of $\rG'_{\pm}$-supported test functions} \label{sec:global}

In this section, we keep the notations in Section ~\ref{sec:RTF}.
\subsection{Explicit computation of exponents}
    For $P \in \cF_\mathrm{RS}$ or $\cF_{\mathrm{RS},F}$, we say that $P$ is \emph{standard}, if $P_{n+1}$ is standard, more concretely, if $P_n$ stabilizes the flag ~\eqref{eq:flag}, then $P_{n+1}$ stabilizes the flag ~\eqref{eq:flag 1} with $k=r-1$ or the flag ~\eqref{eq:flag 2} with $k=r$. We say that $P$ is \emph{antistandard}, if $P_{n+1}$ stabilizes the flag ~\eqref{eq:flag 1} with $k=0$ or the flag ~\eqref{eq:flag 2} with $k=0$. Let $\cF_\mathrm{RS}^{\mathrm{st}}$ be the set of standard Rankin-Selberg parabolic subgroups, and let $\cF_\mathrm{RS}^{\mathrm{ast}}$ be the set of antistandard Rankin-Selberg parabolic subgroups
    \begin{proposition} \label{prop:exponent computation}
        Let $P \in \cF_\mathrm{RS}$ or $\cF_{\mathrm{RS},F}$
        \begin{enumerate}
            \item If $P$ is standard and $\mathrm{Re}(s)<-1$, then for all $\varpi^\vee \in \widehat{\Delta}_{P_{n+1}}^\vee$, $\mathrm{Re} \langle \underline{\rho}_{P,s} , \varpi^\vee  \rangle < 0$.
             \item If $P$ is antistandard and $\mathrm{Re}(s)>1$, then for all $\varpi^\vee \in \widehat{\Delta}_{P_{n+1}}^\vee$, $\mathrm{Re} \langle \underline{\rho}_{P,s} , \varpi^\vee  \rangle < 0$.
        \end{enumerate}
    \end{proposition}

    \begin{proof}
        By the computation in \cite{Zydor18}*{Proof of Lemma 4.2}, if $P$ is standard, then for all $\varpi^\vee \in \widehat{\Delta}_{P_{n+1}}^\vee$, there exists positive integer $i>0$ such that $\langle \underline{\rho}_{P,s}, \varpi^\vee \rangle = i(1+s)$, hence (1) holds. If $P$ is antistandard, then for all $\varpi^\vee \in \widehat{\Delta}_{P_{n+1}}^\vee$, there exists positive integer $i>0$ such that $\langle \underline{\rho}_{P,s}, \varpi^\vee \rangle = i(1-s)$, hence (2) holds.
    \end{proof}

\subsection{Geometric distribution for the Jacquet-Rallis RTF}

Now we state a theorem that computes the geometric term $I_b$ under some assumptions on the support of the test function.
\begin{theorem} \label{thm:main}
    Let $f \in \cS(\rG'(\bA))$. If there exists a place $v$ of $F$ such that $f$ is of the form $f_v f^v$, with $f_v \in \cS(\rG'(F_v)), f^v \in \cS(\rG'(\bA^v))$ and $\supp(f_v) \subset \rG'_+(F_v)$ (resp. $\rG'_-(F_v)$), then
    \begin{enumerate}
        \item  for $s \in \cH_{<-1}$ (resp. $s \in \cH_{>1}$), we have
        \[
            \sum_{b \in \rB(F)} \int_{[\rH_1]} \int_{[\rH_2]} \lvert K_{f,\gamma}(h_1,h_2) \rvert \lvert\det h_1\rvert^{\mathrm{Re}(s)} \rd h_1 \rd h_2 < \infty.
        \]
    \item For $b \in \rB(F)$, and $s \in \cH_{<-1}$ (resp. $\cH_{>1}$) the integral
        \begin{equation} \label{eq:integral_of_K_b,gamma}
             \int_{[\rH_1]} \int_{[\rH_2]}  K_{f,b}(h_1,h_2)  \lvert \det h_1 \rvert^s \xi(h_1) \eta(h_2) \rd h_1 \rd h_2
        \end{equation} 
    (which is absolutely convergent by (1)) equals to $I_b(f,\xi,s)$. In particular, as a function of $s$, the integral ~\eqref{eq:integral_of_K_b,gamma} is holomorphic on $\cH_{<-1}$ (resp. $\cH_{>1}$) and has a meromorphic continuation to $\C$, which is holomorphic on $\C \setminus \{-1,1\}$, and equals to $I_b(f,\xi)$ at $s=0$.
    \item Pick any $\gamma \in O^+(b)$ (resp. $\gamma \in O^-(b)$), then the integral
    \begin{equation} \label{eq:global_regular_orbital_integral}
        \int_{\rH_1(\bA)} \int_{\rH_2(\bA)} f(h_1^{-1}\gamma h_2) \lvert \det h_1 \rvert^s \xi(h_1) \eta(h_2) \rd h_1 \rd h_2
    \end{equation}
    is absolutely convergent on $\cH_{<-1}$ (resp. $\cH_{>1}$) and equals to $I_b(f,\xi,s)$.
    \end{enumerate}   
\end{theorem}

\begin{proof}[Proof of (1)]
    We stick to the case where $\mathrm{supp}(f_v) \subset \rG'_+(F_v)$, the case $\mathrm{supp}(f_v) \subset \rG'_-(F_v)$ follows from the same argument. Let $P \in \cF_{\mathrm{RS}}$ and $P$ be not standard. Then, for any $m \in M_P(F)$ and $n \in N_P(\bA)$, we have
    \[\alpha(mn) = \begin{pmatrix}
        A & b \\ c & d
    \end{pmatrix} \in P_{n+1}(\bA) \cap \rS(\bA).\]
    Where the map $\alpha$ is defined in ~\eqref{eq:the_map_alpha}. Since $P$ is not standard, there is an integer $k>0$ such that the last $k$ coordinates of $b$ are $0$. Hence
        \[
            \det(b,Ab,\cdots,A^nb)=0.
        \]
    This implies $f(h_1^{-1}mnh_2)=0$. Therefore, for $b \in \rB(F)$ we have
        \[
        K_{f,P,b}(h_1,h_2) = \sum_{m \in M_P(F) \cap \rG'_{b}(F)} \int_{N_P(\bA)} f(h_1^{-1}mnh_2)=0
        \]
    Similar to the calculation as we have done in the proof of part (2) of Theorem ~\ref{thm:coarse Bessel}, we can write
        \begin{align*}
           K_{f,b}^T(h_1,h_2) &= \sum_{P \in \cF_\mathrm{RS}^{\mathrm{st}}} \epsilon_P \sum_{\substack{ \gamma \in P_{\rH_1}(F)\backslash \rH_1(F)\\ \delta \in P_{\rH_2}(F) \backslash \rH_2(F)}} \tau_{P_{n+1}}(H_{P_{n+1}}(\delta_n h_{2,n})-T_{P_{n+1}}) K_{f,P,b}(\gamma h_1,\delta h_2) \\
           &= \sum_{Q \in \cF_\mathrm{RS}^{\mathrm{st}}} \sum_{\substack{\gamma \in Q_{\rH_1}(F) \backslash \rH_1(F) \\ \delta \in Q_{\rH_2}(F) \backslash \rH_2(F) }} \Gamma'_{Q_{n+1}}(H_{Q_{n+1}}(\delta_n h_{2,n})-T'_{Q_{n+1}},T_{Q_{n+1}}-T'_{Q_{n+1}}) K_{f,b}^{Q,T}(\gamma h_1,\delta h_2).
        \end{align*}
    Pick $T' \in \fa_{n+1}$ such that $T-T' \in \fa_{n+1}^+$, so for all $Q \in \cF_{\mathrm{RS}}$, the values of 
        \[
            \Gamma'_{Q_{n+1}}(\cdot,T_{Q_{n+1}}-T'_{Q_{n+1}})
        \]
      are 0 or 1.  Then we have
        \begin{align*}
            &\int_{[\rH_1]} \int_{[\rH_2]} \sum_{b \in \rB(F)}  \lvert K_{f,b}^T(h_1,h_2) \rvert \lvert \det h_1 \rvert^{\mathrm{Re}(s)} \rd h_1 \rd h_2 \\
            \le& \sum_{Q \in \cF^{\mathrm{st}}_{\mathrm{RS}}} \int_{[\rH_1]_{Q_{\rH_1}}\times [\rH_2]_{Q_{\rH_2}}} \Gamma'_{Q_{n+1}}(H_{Q_{n+1}}(h_{2,n})-T'_{Q_{n+1}},T_{Q_{n+1}}-T'_{Q_{n+1}}) \times \\
            &\left( \sum_{b \in \rB(F)} \lvert K_{f,b}^{Q,T}(h_1,h_2) \rvert \right) \lvert \det h_1 \rvert^{\mathrm{Re}(s)} \rd h_1 \rd h_2 
        \end{align*}
    By the Iwasawa decomposition, and the fact that for $a \in A_{Q}^\infty$,
        \[
            K_{f,b}^{Q,T}(ax,ay)=e^{\langle 2\rho_Q,H_Q(a) \rangle} K_{f,b}^{Q,T}(x,y),
        \]
    the the summand corresponds to $Q$ in the above expression equals to
        \begin{align*}
            \int_{[M_1]} \int_{[M_2]} &\int_{K_1 \times K_2} e^{\langle -2\rho_{Q_{\rH_1}},H_{Q_{\rH_1}}(m_1) \rangle} e^{\langle -2\rho_{Q_{\rH_2}},H_{Q_{\rH_2}}(m_2) \rangle} 
          \Gamma'_{Q_{n+1}} \left( H_{Q_{n+1}}(m_{2,n})-T'_{Q_{n+1}},T_{Q_{n+1}}-T'_{Q_{n+1}} \right) \\
          & \left( \sum_{b \in \rB(F)} \lvert K_{f,b}^{Q,T}(m_1k_1,m_2k_2) \rvert \right)  \lvert  \det m_1 \rvert^{\mathrm{Re}(s)} \rd m_1 \rd m_2 \rd k_1 \rd k_2 \\
          =\int_{K_1 \times K_2} &\int_{A_{Q,\rH_1}^\infty \backslash ([M_1] \times [M_2])} c_Qp_{Q,\mathrm{Re}(s)}(T_{Q_{n+1}}-T'_{Q_{n+1}})e^{\langle \underline{\rho}_{Q,s},-H_{Q_{n+1}}(m_{2,n})+T'_{Q_{n+1}}\rangle} e^{-\langle 2 \rho_{Q_{\rH_1}}, H_{Q_{\rH_1}}(m_1) \rangle} \\
          &e^{-\langle 2 \rho_{Q_{\rH_2}}, H_{Q_{\rH_2}}(m_2) \rangle} \left( \sum_{b \in \rB(F)} \lvert K_{f,b}^{Q,T}(m_1k_1,m_2k_2) \rvert \right) \lvert \det m_1 \rvert^{\mathrm{Re}(s)} \rd m_1 \rd m_2 \rd k_1 \rd k_2.
        \end{align*}
    Note that the natural map $[M_1] \times [M_2]^{\mathbbm{1}} \to A^\infty_{Q,\rH_1} \backslash ([M_1] \times [M_2])$ is a bijection and is measure-preserving up to a constant $C_Q$, hence the expression above can be written as
    \begin{align} \label{eq:1}
    \begin{split}
        &c_QC_Qp_{Q,\mathrm{Re}(s)}(T_{Q_{n+1}}-T'_{Q_{n+1}})e^{\langle \underline{\rho}_{Q,\mathrm{Re}(s)},T'_{Q_{n+1}}\rangle} \int_{K_1 \times K_2} \int_{[M_1] \times [M_2]^\mathbbm{1}}  e^{-\langle 2 \rho_{Q_{\rH_1}}, H_{Q_{\rH_1}}(m_1) \rangle} \\
          & \left( \sum_{b \in \rB(F)} \lvert K_{f,b}^{Q,T}(m_1k_1,m_2k_2) \rvert \right) \lvert \det m_1 \rvert^{\mathrm{Re}(s)} \rd m_1 \rd m_2 \rd k_1 \rd k_2.
    \end{split}
    \end{align}
    By ~\eqref{eq:asympototic Q geometric}, for $(m_1,m_2) \in [M_1] \times [M_2]^\mathbbm{1}$ and $(k_1,k_2) \in K_1 \times K_2$, we have
        \[
         \sum_{b \in \rB(F)} \lvert K_{f,b}^{Q,T}(m_1k_1,m_2k_2) - F^{Q_{n+1}}(m_{2,n},T_{Q_{n+1}}) K_{f,Q,b}(m_1k_1,m_2k_2) \rvert \ll e^{-N\|T\|} \|m_1\|_{M_1}^{-N} \|m_2\|_{M_2}^{-N} .
         \]
    Using the same argument of the proof of Theorem ~\ref{thm:coarse Bessel} (1), we see that the integral in ~\eqref{eq:1} is finite. To conclude, the expression
      \begin{equation} \label{eq:integral of modified kernel}
              \int_{[\rH_1]} \int_{[\rH_2]} \sum_{b \in \rB(F)}  \lvert K_{f,b}^T(h_1,h_2) \rvert \lvert \det h_1 \rvert^{\mathrm{Re}(s)} \rd h_1 \rd h_2
    \end{equation}

    is bounded by a constant multiple of
        \[
        \sum_{Q \in \cF^\mathrm{st}_\mathrm{RS}} p_Q(T_{Q_{n+1}}-T'_{Q_{n+1}}),
        \]
    which is of the form
        \[
        \sum_{Q \in \cF^\mathrm{st}_\mathrm{RS}} a_Q P_Q(T) e^{\langle \underline{\rho}_{Q,\mathrm{Re}(s)},T_{Q_{n+1}} \rangle}
        \]
    with $a_Q \in \C,P_Q$ is a polynomial on $\fa_{n+1}$.

    When $T$ is sufficiently positive, $T_{Q_{n+1}}$ is of the form 
    \[
        T_{Q_{n+1}} = \sum_{\varpi^\vee \in \widehat{\Delta}_Q^\vee} T_{\varpi^\vee} \varpi^\vee
    \]
    where each $T_{\varpi^\vee}$ is a sufficiently positive real number. Therefore, by Proposition ~\ref{prop:exponent computation}, the expression ~\eqref{eq:integral of modified kernel} is bounded as $T$ varies (and is sufficiently positive).

    By ~\eqref{eq:asympototic Q geometric}, this implies
    \[
        \int_{[\rH_1]} \int_{[\rH_2]} F^{\rG_{n+1}'}(h_{2,n},T)  \left( \sum_{b \in \rB(F)} \lvert K_{f,b}(h_1,h_2)  \rvert \right) \rd h_1 \rd h_2
    \]
    is bounded as $T$ varies. As $T \to \infty$ (i.e. $\langle \alpha,T \rangle \to \infty$ for any $\alpha \in \Delta_0$), $F^{\rG'_{n+1}}(\cdot,T) \to 1$. Thus part (1) follows from the dominated convergence theorem.
\end{proof}

\begin{proof}[Proof of (2)]
    Let $f \in \cS(\rG'(\bA))$ and let $T \in \fa_{n+1}$ be sufficiently positive. For $N>0$ large enough, by applying  ~\eqref{eq:asympototic Q geometric} to $Q=\rG'$, we have
        \[
    \int_{[\rH_1]} \int_{[\rH_2]} \lvert F^{\rG'_{n+1}}(h_{2,n},T) K_{f,b}(h_1,h_2) - K_{f,b}^T(h_1,h_2) \rvert |\det h_1|^s \xi(h_1) \eta(h_2) \rd h_1 \rd h_2 \ll e^{-N\|T\|}.
        \]
    When $T \to \infty$, the integral
        \[
    \int_{[\rH_1]} \int_{[\rH_2]} F^{\rG'_{n+1}}(h_{2,n},T)  K_{f,b}(h_1,h_2)  |\det h_1|^s \xi(h_1) \eta(h_2) \rd h_1 \rd h_2 
        \]
    has a limit
        \[
        \int_{[\rH_1]} \int_{[\rH_2]} K_{f,b}(h_1,h_2)  |\det h_1|^s \xi(h_1) \eta(h_2) \rd h_1 \rd h_2.
        \]
    Similar to the computation in part (1) of the proof, the integral
        \[
         \int_{[\rH_1]} \int_{[\rH_2]}  K_{f,b}^T(h_1,h_2)|\det h_1|^s \xi(h_1) \eta(h_2) \rd h_1 \rd h_2
        \]
    is of the form
        \[
       I_b(f,\xi,s) +  \sum_{Q \ne G \in \cF^\mathrm{st}_{\mathrm{RS}}} P_Q(T) e^{\langle \underline {\rho}_{Q,s},T \rangle}
        \]
    where $P_Q$ is polynomial, hence as $T \to \infty$, it has the limit $I_b(f,\xi,s)$. Combining these, (2) is proved.
\end{proof}

\begin{proof}[Proof of (3)]
    Since $\mathrm{supp}(f_v) \subset \rG'_+(F_v)$, by the definition of $K_{f,b}(h_1,h_2)$ and Lemma ~\ref{lem:unique orbit}, we see that
    \[
        K_{f,b}(h_1,h_2) = \sum_{(\delta_1,\delta_2) \in \rH_1(F) \times \rH_2(F)} f(\delta_1^{-1}\gamma \delta_2).
    \]
    Therefore, the absolute convergence follows from (1), and the remaining part follows from (2).
\end{proof}

The integral ~\eqref{eq:global_regular_orbital_integral} is Eulerian, we will then study the local version of it in the next section.

\subsection{An infinitesimal variant}

The results in Theorem ~\ref{thm:main} have their infinitesimal analogues.

\begin{theorem} \label{thm:main_infinitesimal}
    Let $\fg = \fg_{n+1}$ or $\widetilde{\gl_n}$. Let $\varphi \in \cS(\fg(\bA))$. If there exists a place $v$ of $F$ such that $\varphi$ is of the form $\varphi_v \varphi^v$, with $\varphi_v \in \cS(\fg(F_v)), f^v \in \cS(\fg(\bA^v))$ and $\supp(\varphi_v) \subset \fg_{+}(F_v)$ (resp. $\fg_{-}(F_v)$), then
    \begin{enumerate}
        \item  for $s \in \cH_{<-1}$ (resp. $s \in \cH_{>1}$), we have
        \[
            \sum_{a \in \cB(F)} \int_{[\GL_n]} \lvert K_{\varphi,a}(g) \rvert \lvert\det g \rvert^{\mathrm{Re}(s)}  < \infty.
        \]
    \item For $a \in (\fg/\GL_n)(F)$, and $s \in \cH_{<-1}$ (resp. $\cH_{>1}$) the integral
        \begin{equation} \label{eq:integral_of_K_varphi,a}
             \int_{[\GL_n]}  K_{\varphi,a}(g)  \lvert \det g \rvert^s \xi(g) \eta(g) \rd g
        \end{equation} 
    equals to $I_a(\varphi,\xi,s)$. In particular, as a function of $s$, the integral ~\eqref{eq:integral_of_K_b,gamma} is holomorphic on $\cH_{<-1}$ (resp. $\cH_{>1}$) and has a meromorphic continuation to $\C$ which is holomorphic on $\C \setminus \{-1,1\}$, and equals to $I_a(\varphi,\xi)$ at $s=0$.
    \item Pick any $\gamma \in O^+(a)$ (resp. $\gamma \in O^-(a)$), then the integral
    \begin{equation} \label{eq:global_regular_orbital_integral_infinitesimal}
        \int_{\GL_n(\bA)} \varphi(\gamma \cdot g )  \lvert \det g \rvert^s \xi(g) \eta(g) \rd g
    \end{equation}
    is absolutely convergent on $\cH_{<-1}$ (resp. $\cH_{>1}$) and equals to ~\eqref{eq:integral_of_K_varphi,a}.
    \end{enumerate}   
\end{theorem}

\begin{proof}
    The proof is identical to the proof of Theorem ~\ref{thm:main}, where we use the asymptotic properties in Proposition ~\ref{prop:asympototic_Lie_algebra} instead.
\end{proof}

\section{Local Theory I: Normalized orbital integral} \label{sec:local}

Throughout this section, we let $F$ be a local field of characteristic zero. Let $\gamma \in \rG'_{\mathrm{reg}}(F)$, $f \in \cS(\rG'(F))$ and $\xi:F^\times \to \C^\times$ be a unitary character. We will define a meromorphic function $I_\gamma(f,\xi,s)$ in this section and study its properties in Subsection ~\ref{subsec:local_group}. The function $I_\gamma(f,\xi,s)$ is a regularization of the following integral
\[
    \int_{\rH_1(F) \times \rH_2(F)} f(h_1^{-1} \gamma h_2) \xi(h_1) \lvert h_1 \rvert^s \eta(h_2) \rd h_1 \rd h_2.
\]
This integral is divergent in general for any $s \in \C$, so how to regularize it will be the main part of this section. We will first study the infinitesimal version $I_X(\varphi,\xi,s)$ in Subsection ~\ref{subsec:central_and_rss} and \ref{subsec:local_I_general}, and the group version in Subsection ~\ref{subsec:local_group}.

\subsubsection{Notations and Measures} \label{subsubsec:measure_local}
We fix some notation and Haar measures throughout this section as follows: we fix an additive character $\psi$ of $F$. If $F'/F$ is a finite extension, we always choose $\psi' := \psi \circ \mathrm{Tr}_{F'/F}$ as a fixed additive character on $F'$. We endow $F'$ with the self-dual Haar measure with respect to $\psi'$. 

We fix the Haar measure on $\GL_n(F)$ defined by the differential form $\zeta_F(1) \cdots \zeta_F(n) (\det g_{ij})^{-n} \wedge dg_{ij}$. We denote by $K$ the standard maximal compact subgroup of $\GL_n(F)$. If $F$ is non-Archimedean, then $K = \GL_n(\cO_F)$ and $\mathrm{vol}(K) = \mathrm{vol}(\cO_F)^{n^2}$.

We also denote by $A_n$ (resp. $N_n$) the diagonal subgroup (resp. upper triangular unipotent subgroup) of $\GL_n$, and we endow $A_n(F)$ (resp. $N_n(F)$) with the Haar measure defined by differential form $\zeta_F(1)^n \prod da_i/a_i$ (resp. $\prod dx_{ij}$). If $F$ is non-Archimedean, one has $\vol(A_n(\cO_F)) = \vol(\cO_F)^n$ (resp. $\vol(N_n(\cO_F)) = \vol(\cO_F)^{\frac{n(n-1)}{2}}$). Iwasawa decomposition yields the integration formula
\begin{equation} \label{eq:Iwasawa_measure}
    \int_{\GL_n(F)} f(g) \rd g = C \int_{A_n(F)} \int_{N_n(F)} \int_K f(ank) \rd a \rd n \rd k,
\end{equation}
where $C \in \R_{>0}$ is a constant. When $F$ is non-Archimedean, then $C = \frac{1}{\vol(\cO_F)^{\frac{n(n+1)}{2}}}$.

Write $B_n$ for the upper triangular Borel subgroup, and let $\delta_{B_n}$ be the modulus character on $B_n(F)$, for $a = \mathrm{diag}(a_1,\cdots,a_n) \in A_n(F)$, we have $\delta_{B_n}(a) = \lvert a_1 \rvert^{n-1} \cdots \lvert a_n \rvert^{1-n}$. 

Let $\langle \cdot,\cdot \rangle$ be the $\GL_n(F)$-invariant bilinear pairing on $\widetilde{\gl_n}(F)$ defined by
    \begin{equation} \label{eq:pairing_on_tilde_gl}
         \langle (X_1,v_1,u_1),(X_2,v_2,u_2) \rangle = \mathrm{Trace}(X_1X_2) + u_1v_2+u_2v_1.
    \end{equation}
For $\varphi \in \cS(\widetilde{\gl_n}(F))$, we define its Fourier transform by
\[
    \cF \varphi(Y) = \int_{\widetilde{\gl_n}(F)} \varphi(X) \psi(\langle X,Y \rangle) \rd X.
\]
    
We then endow $\widetilde{\gl_n}(F)$ with the self-dual measure, one check directly that this measure coincides with the product of measure on $F^{n^2} \times F^n \times F^n$ when we use standard coordinates on $\widetilde{\gl_n}(F)$.

Let $\fn_n$ and $\fb_n$ be the Lie algebra of $N_n$ and $B_n$ respectively, and let $\fn_n'$ be the space of matrices $(a_{ij})$ such that $a_{ij} = 0$ unless $j-i \ge 2$. We define
        \[
           \widetilde{\fn_n} = \{ (A,v,u) \in \widetilde{\gl_n} \mid A \in \fn_n, v \in F^{n-1}, u =0 \} ,\quad \widetilde{\fn'_n} = \{ (A,v,u) \in \widetilde{\gl_n} \mid A \in \fn'_n, v \in F^{n-2}, u = 0 \},
        \]

where $F^{n-1}$ (resp. $F^{n-2}$) stands for the subspace of $F^n$ with the last (resp. last two) coordinate 0. 

We write $\widetilde{\fb_n} := \{(A,v,u) \in \widetilde{\gl_n} \mid A \in \fb_n, u = 0 \}$. All the vector spaces $\fn_n,\fb_n,\widetilde{\fn_n},\widetilde{\fn'_n},\widetilde{\fb_n}$ has a natural basis and can be identified with $F^m$ for some $m \ge 0$. We then transport the product measure on $F^m$ to the Haar measure on these vector spaces via this identification.

Recall the regular element $Z_\lambda^+$ defined in ~\eqref{eq:central orbit representative}. By ~\cite{Zhang14b}*{Lemma 6.8} (also see ~\cite{BP21b}*{Lemma 5.7.4}), the map
     \begin{equation} \label{eq:measure_preserving_map_n}
        N_n(F) \to Z_\lambda^+ + \widetilde{\fn_n'} , \quad n \mapsto Z_\lambda^+ \cdot n   
    \end{equation}
    is a bijection and is measure-preserving.
By ~\cite{BP21b}*{(5,7,8)}, we have the integration formula
\begin{equation} \label{eq:integration_formula_b}
    \int_{\widetilde{\gl_n}(F)} f(Y) \rd Y = \frac{1}{\zeta_n} \int_{N_n(F) \backslash \GL_n(F)} \int_{\widetilde{\fb_n}(F)} f((Z_0^- + Y) \cdot h ) \rd Y \rd h.
\end{equation}
where $f \in \cS(\widetilde{\gl_n}(F))$ and
\[
    \zeta_n = \zeta_F(1) \cdots \zeta_F(n).
\]

\subsection{Orbital integrals on the Lie algebra: central and regular semisimple elements} 
\label{subsec:central_and_rss}

Recall that we have a GIT quotient $\cA := \widetilde{\gl_n}/\GL_n$ and the corresponding quotient map $q: \widetilde{\gl_n} \to \cA$. Let $a \in \cA(F)$ be a central element, by Lemma ~\ref{lem:regular_nilpotent_orbit}, in the fiber $\widetilde{\gl_n}_{,a}$ of $a$, there are two regular orbits $O^+(a)$ and $O^-(a)$.

We define the $L$-factors associated to the orbits $O^+(a)$ and $O^-(a)$ to be  
\[
    L_a^+(s,\xi) = \prod_{i=1}^n L(-is-i+1,(\xi \cdot \eta)^{-i}), \quad  L_a^-(s,\xi) = \prod_{i=1}^n L(is-i+1,(\xi \cdot \eta)^{i}).
\]

For a regular element $X$ in $\widetilde{\gl_n}_{,a}(F)$. We put
\[
  L_X(s,\xi) =
    \begin{cases}
      L_a^+(s,\xi) & \text{ if }X \in O^+(a), \\
      L_a^-(s,\xi) & \text{ if }X \in O^-(a). 
    \end{cases}
\]

Let $\chi: F^\times \to \C^\times$ be a character. We say that $\varphi \in \cS(\widetilde{\gl_n}(F))$ is $\chi$\emph{-unstable}, if for all regular semisimple element $X \in \widetilde{\gl_n}(F)$, we have
\[
    \int_{\GL_n(F)} \varphi(X \cdot g) \chi(g) \rd g = 0.
\]
We call a continuous functional (i.e. a distribution) $I: \cS(\widetilde{\gl_n}(F)) \to \C$ is $\chi$\emph{-stable}, if for any $\chi$-unstable function $\varphi \in \cS(\widetilde{\gl_n}(F))$, we have $I(\varphi)=0$. 

For $g \in \GL_n(F)$ and $\varphi \in \cS(\widetilde{\gl_n}(F))$, we let $\rR(g)\varphi$ denote the right translation given by $\rR(g)\varphi(X) = \varphi(X \cdot g)$. Note that if $I$ is a $\chi$-stable distribution, then for any $\varphi \in \cS(\widetilde{\gl_n}(F))$ and $g \in \GL_n(F)$, we have
\[
    I(\rR(g)\varphi) = \chi^{-1}(g)I(\varphi).
\]

\begin{proposition}  \label{prop:local_central}
    Let $a \in \cA(F)$ be a central element and $X \in \widetilde{\gl_n}_{,a}(F)$ be a regular element. Let $\varphi \in \cS(\widetilde{\gl_n}(F))$ and $s \in \C$. Consider the integral
    \[
        I_X(\varphi,\xi,s) := \int_{\GL_n(F)} \varphi(X \cdot g) \xi(g) \eta(g) \lvert \det g \rvert^s \rd g.
    \]
    Then we have the following statements:
    \begin{enumerate}
        \item  If $X \in O^+(a)$ (resp. $X \in O^-(a)$), $I_X(\varphi,\xi,s)$ is absolutely convergent on $\cH_{<-1+\frac 1n}$ (resp. $\cH_{>1-\frac 1n}$) and has meromorphic continuation to $\C$, with poles contained in the poles of $L_X(s,\xi)$.
        \item For $s \in \C$ which is not a pole of $L_X(s,\xi)$, the resulting linear map
        \[
         \cS(\widetilde{\gl_n}(F)) \to \C, \quad \varphi \mapsto I_X(\varphi,\xi,s).
        \]
        is continuous.
        \item  As a function of $s$
        \[
           I_X^\natural(\varphi,\xi,s) := \frac{I_X(\varphi,\xi,s)}{L_X(s,\xi)}
        \]
    is entire for any $\varphi \in \cS(\widetilde{\gl_n}(F))$, and we can choose $\varphi$ such that it equals to 1.
        \item If $F$ is non-Archimedean with ring of integer $\cO$, suppose that $\varphi = 1_{\widetilde{\gl}(\cO)}$, $\xi$ and $\eta$ are unramified, $\vol(\cO_F)=1$, $X \in \widetilde{\gl_n}_{,\pm}(\cO)$, then
        \[
            I_X(\varphi,s,\xi) =  L_X(s,\xi).
        \]
        \item  For any $s$ which is not a pole of $L_X(s,\xi)$, the distribution $I_X(\cdot,\xi,s)$ satisfies the following two properties:
        \begin{itemize}
            \item  $I_X(\cdot,\xi,s)$ is a $\xi \eta \lvert \cdot \rvert^s$-stable distribution.
            \item $I_X(\cdot,\xi,s)$ is supported on the closure of $\GL_n(F)$-orbit of $X$.
        \end{itemize}
    \end{enumerate}
\end{proposition}

\begin{proof}
    We prove the case where $X \in O^+(a)$, and the case $X \in O^-(a)$ follows from the same argument. Assume $a$ is the image of $(\lambda \cdot \mathrm{id},0,0) \in \widetilde{\gl_n(F)}$.

    We first prove this proposition when $X=Z_\lambda^+$. Let $\chi = \xi \cdot \eta$. By Iwasawa decomposition ~\eqref{eq:Iwasawa_measure} and the fact that the map ~\eqref{eq:measure_preserving_map_n} is measure preserving, the integral defining $I_X(\varphi,\xi,s)$ can be written as
        \begin{align} \label{eq:an integral}
            \begin{split}
             &\int_{N_n(F) \backslash \GL_n(F)} \int_{\widetilde{\fn'_n}(F)} \varphi( (Z_\lambda^++N) \cdot g ) \chi(g) \lvert \det g \rvert^s \rd N \rd g \\
             = C &\int_{A_n(F)} \int_K \int_{\widetilde{\fn'_n}(F)} \delta_{B_n}(a)^{-1} \varphi( (Z_\lambda^++N) \cdot ak ) \chi(ak) \lvert \det a \rvert^s \rd N \rd a \rd k,
             \end{split}
        \end{align} 
    where $C$ is the constant appearing in ~\eqref{eq:Iwasawa_measure}. Let $f_{\varphi}$ be the Schwartz function on $F^n$ defined by
        \[
    f_{\varphi}(x_1,\cdots,x_n) =  C \int_{\widetilde{\fn'_n}(F)} \int_K \varphi  \left(( X_\lambda(x_1,\cdots,x_n)+N) \cdot k \right) \chi(k) \rd N \rd k,
        \]
    where
        \[
            X_\lambda(x_1,\cdots,x_n) = \left( \begin{pmatrix}
        \lambda & x_1 & 0 & \cdots & 0 \\
        0 & \lambda & x_2 & \cdots & 0 \\
        0 & 0 & \lambda & \cdots & 0 \\
        \cdots & \cdots & \cdots & \cdots & \cdots \\
        0 & 0 & 0 & \cdots & \lambda
    \end{pmatrix},
    \begin{pmatrix}
        0 \\ 0 \\ 0 \\ \cdots \\ x_n
    \end{pmatrix},
    0  \right)
        \]
    Then the integral ~\eqref{eq:an integral} reduces to
    \[
        \int_{A_n(F)} f_{\varphi} \left(\frac{a_2}{a_1},\cdots,\frac{a_n}{a_{n-1}},\frac{1}{a_n} \right) \lvert a_2 \cdots a_n \rvert \lvert \det a \rvert^s \rd a.
     \]
    Let $b_i=a_{i+1}/a_i$ for $1 \le i \le n-1$ and $b_n=1/a_n$, the integral above equals to
    \begin{equation} \label{eq:reduce_to_Tate_integral}
        \int_{(F^\times)^n} f_{\varphi}(b_1,b_2,\cdots,b_n) \prod_{i=1}^n  \chi(b_i)^{-i} \lvert b_i \rvert^{-is-i+1} \rd b_1 \cdots \rd b_n.
    \end{equation}
     
    Note that $\varphi \in \cS(\widetilde{{\gl_n}}(F)) \mapsto f_\varphi \in \cS(F^n)$ is continuous, part (1)--(3) also follows from Tate thesis. For part (4), we compute directly that $f_{\varphi} = 1_{\cO^n}$, by the unramified computation in Tate's thesis, $I_{Z_\lambda^+}(\varphi,s,\xi) = L_{Z_\lambda^+}(s,\xi)$. 
    
    We finally show part (5). The support of $I_X(\cdot,\xi,s)$ lies in the closure of orbit of $X$ follows from the definition, to show stability, following the calculation in ~\cite{BP21c}*{Subsection 5.7}, Let $Z^-_0$ be the element in ~\eqref{eq:central_orbit_representative_minus}.
    Consider the following iterated integral
    \begin{equation} \label{eq:I^1}
        I_\lambda^1(\varphi,\xi,s) := \int_{N_n(F)\backslash \GL_n(F)} \left( \int_{\widetilde{\fn_n}(F)} \varphi  ( Z^+_\lambda  \cdot g ) \psi(\langle Z_0^-, X \rangle) \rd X \right) \chi(g) \lvert \det g \rvert^s \rd g.
    \end{equation}    
    Using Iwasawa decomposition, the above integral is equal to
    \[
        \int_{A_n(F)} \int_{F^n} f_{\varphi}(x_1,\cdots,x_n) \psi \left( \frac{a_1}{a_2}x_1+\cdots+\frac{a_{n-1}}{a_n}x_{n-1}+a_n x_n \right) \lvert a_1 \cdots a_n \rvert^{s+1} \chi(a) \rd x \rd a.
    \]
    Let $b_i = a_i/a_{i+1}$ for $1 \le i \le n-1$ and $b_n=a_n$, this can be written as
    \[
        \int_{(F^\times)^n} \widehat{f_{\varphi}}(b_1,\cdots,b_n) \prod_{i=1}^n \lvert b_i \rvert^{i(s+1)} \chi(b_i)^i \rd b_1 \cdots \rd b_n.
    \]
    We thus see that the integral defining $I^1_\lambda(\varphi,\xi,s)$ is convergent when $\mathrm{Re}(s)>-1$ and by the functional equation for local Tate's zeta integral, we have
    \[
        I^1_\lambda(\varphi,\xi,s) = \gamma^+_\xi(s) I_{Z_\lambda^+}(\varphi,\xi,s),
    \]
    where
    \[
        \gamma^+_\xi(s) = \prod_{i=1}^n \gamma(-is-i+1,(\xi \cdot \eta)^{-i}, \psi).
    \]
    In other words, when $\mathrm{Re}(s)>-1$, the meromorphic continuation of  $I_X(\varphi,\xi,s)$ is given explicitly by
    \begin{equation} \label{eq:I_X^+_alternative}
        I_{Z_\lambda^+}(\varphi,\xi,s) = \gamma^+_\xi(s)^{-1} I^1(\varphi,\xi,s),
    \end{equation}

    Let $\omega^-_{\xi,s}$ be the function on $\widetilde{\gl_n}(F)$ defined by
    \[
        \omega^-_{\xi,s}(X) = \chi(\delta^-(X))  \lvert \delta^-(X) \rvert^{s}.
    \]
    Put $\varphi_\lambda(X) = \varphi(X+(\lambda,0,0))$.

     Direct computation shows that for $Y \in \widetilde{\fb_n}(F)$, we have
    \[
        \omega^-_{\xi,s}((Z_0^-+Y) \cdot g) = \chi(g) \lvert \det g \rvert^s \omega_{\xi,s}(Z_0^-).
    \]
    By Fourier inversion, we can write
    \[
        I_\lambda^1(\varphi,\xi,s) = \int_{N_n(F) \backslash \GL_n(F)} \int_{\widetilde{\fb_n}(F)} \cF \varphi_{\lambda}((Z_0^-+Y) \cdot g) \chi(g) \lvert \det g \rvert^s \rd Y \rd g.
    \]
    By ~\eqref{eq:integration_formula_b}, we then see that 
    \[
        I_\lambda^1(\varphi,\xi,s) = \zeta_n\omega^-_{\xi,s}(Z_0^-)^{-1}  \int_{\widetilde{\gl_n}(F)} \cF \varphi_{\lambda} (X) \omega^-_{\xi,s}(X) \rd X.
    \]
    By ~\cite{Chaudouard19}*{Corollary 3.1.8.2}, if $\varphi$ is $\xi \eta \lvert \cdot \rvert^s$-unstable, then so is $\cF \varphi$. (Indeed, in \emph{loc.cit} only proves the case when $\xi$ is trivial and $s=0$, but the same proof works in general). Therefore
    \[
    I_\lambda^1(\varphi,\xi,s) = \zeta_n \omega^-_{\xi,s}(Z_0^-)^{-1} \int_{\cA_{\mathrm{rs}}(F)} \int_{\GL_n(F)} \cF \varphi_\lambda(X \cdot g) \chi(g) \lvert \det g \rvert^s \omega_{\xi,s}^-(X) \rd g \rd X = 0. 
    \]
    Together with ~\eqref{eq:I_X^+_alternative}, we see that the distribution $I_X(\cdot,\xi,s)$ is $\xi \eta \lvert \cdot \rvert^s$-stable. This finishes the proof when $X=Z_\lambda^+$.

     Now let $X \in O^+(a)$ be a general element, then it can be written of the form $Z_\lambda^+ \cdot g$ in ~\eqref{eq:central orbit representative}. By a change of variable, we see that when convergent, we have
    \begin{equation} \label{eq:relation_I_Xg}
         I_{Z_\lambda^+}(\varphi,\xi,s) = \xi(g) \eta(g) \lvert \det g \rvert^s \cdot I_{X}(\varphi,\xi,s). 
    \end{equation}
    
    Part (1),(2), and (5) of the proposition directly follows from this. In the Tate integral ~\eqref{eq:reduce_to_Tate_integral}, for any $a \in F$ we can choose $f_\varphi$ such that the integral gives $\lvert a \rvert^s \cdot L_{Z_\lambda^+}(s)$, therefore part (3) holds. Part (4) follows from the fact that if $X \in \widetilde{\gl_n}_{,+}(\cO)$ then $g  \in \GL_n(\cO_F)$ (see ~\cite{Zhang14b}*{Proposition 6.3}).
 \end{proof}

According to the equation ~\eqref{eq:I_X^+_alternative} in the proof above, for $\mathrm{Re}(s)>-1$, we have the following expression
\begin{equation} \label{eq:I_X^+_Fourier}
    I_{Z^+_\lambda}(\varphi,\xi,s) = \zeta_n \omega^-_{\xi,s}(Z_0^-)^{-1} \gamma^+_\xi(s)^{-1} \int_{\widetilde{\gl_n}(F)} \cF \varphi_{\lambda} (X) \omega^-_{\xi,s}(X) \rd X.
\end{equation}
By the same computation, if we denote by $\omega^+_{\xi,s}$ the function on $\widetilde{\gl_n}(F)$ defined by
\[
    \omega_{\xi,s}^+(X) = \xi^{-1}(\delta^+(X)) \eta(\delta^+(X)) \lvert \delta^+(X) \rvert^{-s}.
\]
Then for $\mathrm{Re}(s)<1$, we then have 
\begin{equation} \label{eq:I_X^-_Fourier}
   I_{Z_\lambda^-}(\varphi,\xi,s) = \zeta_n \omega^+_{\xi,s}(Z_0^+)^{-1} \gamma_\xi^-(s)^{-1} \int_{\widetilde{\gl_n}(F)} \cF \varphi_{\lambda}(X) \omega_{\xi,s}^+(X) \rd X.
\end{equation}
where
   \[
        \gamma^-_\xi(s) = \prod_{i=1}^n \gamma(is-i+1,(\xi \cdot \eta)^{i}, \psi).
    \]
  
Now we switch to the case when $a \in \cA(F)$ is a regular semisimple element. In this case, the fiber $\widetilde{\gl_n}_{,a}(F)$ forms a single $\GL_n(F)$ orbit, any element in this orbit is regular. For $X \in \widetilde{\gl_n}_{,a}(F)$, we put
    \[
        L_X(s,\xi) = 1.
    \]

The following proposition is analogous to (and easier than) Proposition ~\ref{prop:local_central}

\begin{proposition}  \label{prop:local_rss}
    Let $a \in \cA(F)$ be a regular semisimple element and $X \in \widetilde{\gl_n}_{,a}(F)$. Let $\varphi \in \cS(\widetilde{\gl_n}(F))$. Consider the integral
    \[
        I_X(\varphi,\xi,s) := \int_{\GL_n(F)} \varphi(X \cdot g) \xi(g) \eta(g) \lvert \det g \rvert^s \rd g.
    \]
    Then we have the following statements:
    \begin{enumerate}
        \item  The integral $I_X(\varphi,\xi,s)$ is absolutely convergent for any $s \in \C$ and defines an entire function on $\C$.
        \item For any $s \in \C$, the resulting map
        \[
         \cS(\widetilde{\gl_n}(F)) \to \C, \quad \varphi \mapsto I_X(\varphi,\xi,\cdot)
        \]
        is continuous.
        
        \item  As a function of $s$
        \[
            I_X^\natural(\varphi,\xi,s) :=  \frac{I_X(\varphi,\xi,s)}{L_X(s,\xi)} = I_X(\varphi,\xi,s)
        \]
    is entire for any $\varphi \in \cS(\widetilde{\gl_n}(F))$,and we can choose $\varphi$ such that it equals to 1.
    
        \item If $F$ is non-archimedean with ring of integer $\cO$, suppose that $\varphi = 1_{\widetilde{\gl}(\cO)}$, $X \in \widetilde{\gl_n}_{,\mathrm{rs}}(\cO)$, $\xi$ and $\eta$ are unramified and $\vol(\cO_F)=1$, then
        \[
            I_X(\varphi,s,\xi) = L_X(s,\xi).
        \]

        \item For any $s \in \C$, the distribution $I_X(\cdot,\xi,s)$ satisfies the following two properties:
        \begin{itemize}
            \item  $I_X(\cdot,\xi,s)$ is a $\xi \eta \lvert \cdot \rvert^s$-stable distribution.
            \item $I_X(\cdot,\xi,s)$ is supported on the orbit of $X$.
        \end{itemize}
    \end{enumerate}
\end{proposition}

\begin{proof}
    Since $X$ is regular semisimple, its orbit is closed, and as a function of $g$, $g \mapsto \varphi_s(X \cdot g)$ is then compactly supported. The integral defining $I_X(\varphi_s,\xi,s)$ is therefore absolutely convergent and entire for $s \in \C$. This proves (1) and (2), for (3), the orbit of $X$ is a closed subset in the non-archimedean case and a closed Nash submanifold in the archimedean case. Take a Schwartz function $\varphi'$ on the orbit of $X$ such that 
    \[
        \int_{\GL_n(F)} \varphi'(X \cdot g) \xi(g) \eta(g) \lvert \det g \rvert^s \rd g = 1.
    \]
    we can then extend $\varphi'$ to a Schwartz function $\varphi$ on $\widetilde{\gl_n}(F)$. (For Archimedean case, see ~\cite{AG}, Theorem 4.6.1). The function $\varphi$ satisfies the requirement of (3).

    For part (4), We show that $\varphi(X \cdot g)=1$ if and only if $g \in \GL_n(\cO_F)$. In fact, $X=(A,v,u) \in \widetilde{\gl_n}_{,\mathrm{rs}}(\cO)$ implies that $v,Av,\cdots,A^{n-1}v$ is a $\cO$-basis of $\cO^n$, and $u,uA,\cdots,uA^{n-1}$ is a $\cO$-basis of $\cO_n$. Therefore, if $X \cdot g \in \gl_n(\cO)$, then $g^{-1}(v,Av,\cdots,A^{n-1}v) \in \gl_n(\cO)$, hence $g^{-1} \in \gl_n(\cO)$, similarly $(u,uA,\cdots,uA^{n-1})g \in \gl_n(\cO)$ implies $g \in \gl_n(\cO)$, hence $g \in \GL_n(\cO)$. Therefore, the integral $I^+_a(\varphi,s,\xi)$ is just the volume of $\GL_n(\cO_F)$, which is $1$. Finally, since $a$ is regular semisimple, part (5) is trivial in this case.
\end{proof}

\subsection{Orbital Integral on the Lie algebra: general element} \label{subsec:local_I_general}

Let $a \in \cA(F)$ be a general element, we use the notation in descent construction in Subsection ~\ref{subsec:descent}. From the element $a$, we have an associated embedding $\iota_\fh: \widetilde{\fh}' \to \widetilde{\gl_n}$, where the induced map $\iota_\cA: \cA_H' \to \cA$ sending $a_H$ to $a$ and each component $a_i (0 \le i \le k)$ of $a_H$ is either regular semisimple (for $i=0$) or central (for $i>0$).

By Lemma ~\ref{lem:regular_nilpotent_orbit}, there are $2^k$ regular orbits for the action of $H(F)$ on $\widetilde{\fh}_{a_H}(F)$. Let $\cE$ be the set of functions from $\{1,\cdots,k\}$ to $\{ +,- \}$. For each $\varepsilon \in \cE$, there is a regular orbit $O^\varepsilon(a)$ in $\widetilde{\fh}_{a_H}$ defined by
    \[
        O^\varepsilon(a_H) := \widetilde{\fh_0}_{,a_0}(F) \times \prod_{1 \le k \le i} O^{\varepsilon_i}(a_i),
    \]

For each $1 \le i \le k$, we put
\[
    \xi_i = \xi \circ \mathrm{Nm}_{F_i/F}, \quad \eta_i = \eta_{(F_i \otimes_F E)/F_i} = \eta \circ \mathrm{Nm}_{F_i/F},
\]
both are unitary characters of $F_i^\times$. By an abuse of notation, we also use $\xi$ and $\eta$ to denote a character on $H(F)$ defined by
\[
    \xi(h) := \prod_{i=0}^k \xi_i(\det h_i), \quad \eta(h) := \prod_{i=0}^k \eta_i(\det h_i).
\]
We then define
\[
    L_{a_i}^+(\xi,s) = \prod_{j=1}^{n_i} L(-js-j+1, (\xi_i \cdot \eta_i)^{-j}), \quad L_{a_i}^-(\xi,s) = \prod_{j=1}^{n_i} L(js-j+1, (\xi_i \cdot \eta_i)^{j}).
\]

For $\varepsilon \in \cE$, we define an $L$-factor by
\begin{equation} \label{eq:L_central_rss}
    L_{a_H}^\varepsilon(\xi,s) = \prod_{i=1}^k L_{a_i}^{\varepsilon_i}(s,\xi).
\end{equation}

Let $X \in \widetilde{\fh}_{a_H}(F)$ be a regular element, suppose that $X \in O^\varepsilon(a_H)$, we then put
\begin{equation} \label{eq:L_X_definition}
      L_X(s,\xi) = L_{a_H}^{\varepsilon}(s,\xi),
\end{equation}
  
 By Proposition ~\ref{prop:local_central} and Proposition ~\ref{prop:local_rss}, for each $0 \le i \le k$ and any $s \in \C$ which is not a pole of $L^\varepsilon_{a_H}(s,\xi)$, there exists a continuous linear map
    \[
      \cS(\widetilde{\fh_i}(F)) \to \C, \quad   \varphi \mapsto I_{X_i}(\varphi,\xi,s).
    \]
Identifying $\cS(\widetilde{\fh}(F))$ with $\widehat{\bigotimes}_{i=0}^k \cS(\widetilde{\fh_i}(F))$, for any regular element $X \in \widetilde{\fh}_{a_H}(F)$ and $s \in \C$ which is not a pole of $L_{X}(s,\xi)$, we define the distribution $I_X(\cdot,\xi,s)$ as the tensor product of those $I_{X_i}(,\xi,s)$. If $\varphi$ is factorizable in the sense that

    \[
        \varphi(X_0,\cdots,X_k) = \prod_{i=0}^k \varphi_{i}(X_i)
    \]
    where $\varphi_{i} \in \cS(\widetilde{\fh_i}(F))$, then we have
    \[
        I_X(\varphi,\xi,s) = \prod_{i=0}^k I_{X_i}(\varphi_{i},\xi_i,s). 
    \]

Combining Proposition ~\ref{prop:local_central} and ~\ref{prop:local_rss}, we have following result
\begin{corollary} \label{cor:local_central_and_rss}
    Let $a_H \in \cA_H(F)$ as above and let $\varphi \in \cS(\widetilde{\fh}(F))$, and $X$ be a regular element in $\widetilde{\fh}_{a_H}(F)$. We then have the following assertions
    \begin{enumerate}
        \item As a function of $s$, $I_X(\varphi,\xi,s)$ is meromorphic on $\C$, and the pole is contained in the pole of $L_X(s,\xi)$.

\item For $s \in \C$ which is not a pole of $L_X(s,\xi)$, the resulting linear map
        \[
         \cS(\widetilde{\fh}(F)) \to \C, \quad \varphi \mapsto I_X(\varphi,\xi,s).
        \]
        is continuous.
        \item  As a function of $s$
        \[
           I_X^\natural(\varphi,\xi,s)  := \frac{I_X(\varphi,\xi,s)}{L_X(s,\xi)}
        \]
    is entire for any $\varphi \in \cS(\widetilde{\gl_n}(F))$, and we can choose $\varphi$ such that it equals to 1.
        \item If $F$ is non-archimedean with ring of integer $\cO$, suppose that $\varphi = 1_{\widetilde{\fh}(\cO)}$, $\xi$ and $\eta$ are unramified, each $F_i$ is unramified over $F$, $\vol(\cO_F) =1$ and $X \in \fh_{0,\mathrm{rs}}(\cO) \times \widetilde{\fh^\varepsilon}(\cO)$, then
        \[
            I_X(\varphi,s,\xi) =  L_X(s,\xi).
        \]
        \item  For any $s$ which is not a pole of $L_X(s,\xi)$, the distribution $I_X(\cdot,\xi,s)$ satisfies the following two properties:
        \begin{itemize}
            \item  $I_X(\cdot,\xi,s)$ is a $\xi \eta \lvert \cdot \rvert^s$-stable distribution.
            \item $I_X(\cdot,\xi,s)$ is supported on the closure of the $\GL_n(F)$ orbit of $X$.
        \end{itemize}
    \end{enumerate}
\end{corollary}

The notion of stable distribution on $\widetilde{\fh}(F)$ is defined in the same way as in the beginning of Subsection ~\ref{subsec:central_and_rss}.

Now let $X \in \widetilde{\gl_n}_{,a}(F)$ be a regular element. Assume that $X$ is of type $\varepsilon$ (see Proposition ~\ref{prop:regular_orbit_tilde_gl_n}), thus there exists $X_H \in O^\varepsilon(a_H)$ and $g \in \GL_n(F)$ such that $\iota_\fh(X_H)=X \cdot g$, we set
\begin{equation} \label{eq:general_L}
     L_X(s,\xi) := L_{X_H}(s,\xi).
\end{equation}
   
For any $s \in \C$ which is not a pole of $L_X(x,\xi)$, we will now construct a map 
\[
    \cS(\widetilde{\gl_n}(F)) \to \C, \quad \varphi \mapsto I_X(\varphi,\xi,\cdot),
\]
the construction will consist of several steps.

First of all, since $\iota_\cA:\cA'_H \to \cA$ is \'{e}tale and sends $a_H$ to $a$,  we can choose an open neighbourhood $\omega_H$ of $a_H$ in $\cA'_H(F)$ and an open neighbourhood $\omega$ of $a$ in $\cA(F)$ such that
    \begin{itemize}
        \item If $F$ is archimedean, both $\omega$ and $\omega_H$ are semi-algebraic ~\cite{BCR98}*{Proposition 8.1.2}.
        \item $\iota_\cA$ induces a bijection $\omega_H \to \omega$, which is an isomorphism of Nash manifolds if $F$ is archimedean, and isomorphism of analytic manifolds if $F$ is non-archimedean,
    \end{itemize}
Let $\Omega := q^{-1}(\omega)$ and $\Omega_H:=q_H^{-1}(\omega_H)$. Then the top horizontal map in diagram ~\eqref{diag:Luna slice} induces an isomorphism
        \begin{equation} \label{eq:analytic_isomorphism}
            \Omega_H \times^{H(F)} \GL_n(F) \to \Omega, \quad (Y,g) \mapsto \iota_{\fh}(Y) \cdot g
        \end{equation}
of Nash manifolds when $F$ is archimedean and analytic manifolds if $F$ is non-archimedean.

Given $\varphi \in \cS(\Omega)$. Let $\varphi' \in \cS(\Omega_H \times \GL_n(F))$ such that for any $(Y,g) \in \Omega_H \times \GL_n(F)$ we have
        \begin{equation} \label{eq:varphi'}
               \varphi(\iota_\fh(Y) \cdot g) = \int_{H(F)} \varphi' (Y \cdot h,h^{-1}g) \rd h.
        \end{equation}
The existence of $\varphi'$ follows from the fact that $\Omega_H \times \GL_n(F) \to \Omega_H \times^{H(F)} \GL_n(F) \cong \Omega$ is a submersion. For any $s \in \C$, we put
        \begin{equation} \label{eq:varphi_H}
                \varphi_{H,s}(Y) = \int_{\GL_n(F)} \varphi'(Y,g) \xi(g) \eta(g) \lvert \det g \rvert^s \rd g,
        \end{equation}
then $\varphi_{H,s} \in \cS(\Omega_H)$ for any $s \in \C$ and the map $\C \to \cS(\Omega_H), s \mapsto \varphi_{H,s}$ is holomorphic.

\begin{lemma} \label{lem:well-defined}
    For $s \in \C$, we can regard $\varphi_{H,s}$ as a Schwartz function on $\widetilde{\fh}(F)$, if $s$ is not a pole of $L_X(s,\xi)$, then the complex number
        \[
            I_{X_H}(\varphi_{H,s},\xi,s)
        \]
    is independent of the choice of $\varphi'$ defining $\varphi_{H,s}$.
\end{lemma}

\begin{proof}
    Let $\varphi''$ be another Schwartz function on $\cS(\Omega_H \times \GL_n(F))$ such that for any $(Y,g) \in \cS(\Omega_H \times \GL_n(F)$, the equation
        \[
        \varphi(\iota_\fh(Y) \cdot g) = \int_{H(F)} \varphi'' (Y \cdot h,h^{-1}g) \rd h.
        \]
    holds. We put
    \[
         \psi_{H,s}(Y) = \int_{\GL_n(F)} \varphi''_s (Y,g) \xi(g) \eta(g) \lvert \det g \rvert^s \rd g.
    \]
    Switching the order of integral shows that the function $\varphi_{H,s}-\psi_{H,s}$ is $\xi \eta \lvert \cdot \rvert^s$-unstable. By Corollary ~\ref{cor:local_central_and_rss} (5), $I_{X_H}(\varphi_{H,s},\xi,s) = I_{X_H}(\psi_{H,s},\xi,s)$.
\end{proof}

Recall that there is a unique $g \in \GL_n(F)$ such that $X = \iota_\fh(X_H) \cdot g$. For $\varphi \in \cS(\Omega)$, we define 
    \begin{equation} \label{eq:definition_I_X_on_Omega}
            I_X(\varphi,\xi,s) := \xi(g)^{-1} \eta(g) \lvert \det g \rvert^{-s} I_{X_H}(\varphi_{H,s},\xi,s).
    \end{equation}
    
By Lemma ~\ref{lem:well-defined}, $I_X(\varphi,\xi,s)$ is a well-defined meromorphic function on $\C$.

\begin{lemma} \label{lem:independence_of_orbit}
    For $\varphi \in \cS(\Omega)$, the meromorphic function $I_X(\varphi,\xi,s)$ only depends on the value of $\varphi$ on the orbit of $X$ (i.e. it is supported on the closure of the orbit of $X$).
\end{lemma}

\begin{proof}
    We are reduced to show that if $\varphi$ is zero on the orbit of $X$, then $I_X(\varphi,\xi,s)=0$. For this $\varphi$, we can pick the $\varphi' \in \cS(\Omega_H \times \GL_n(F))$ such that $\varphi'$ vanishes on $\Omega_H \times \{ \text{orbit of }X \}$. Then $\varphi_{H,s}$ vanishes on orbit of $X_H$, the result hence follows from ~\ref{cor:local_central_and_rss} (5).
\end{proof}

\begin{remark}
    As we mentioned in Subsection ~\ref{subsec:descent}, the map $\iota_\fh$ depends on the choices of the isomorphisms $i$ and $i^0$, but we can check directly that the definition of $I_X(\varphi,\xi,s)$ is independent of the choice of these isomorphisms, since different $i$ and $i^0$ differ by conjugation by an element in $\GL_n(F)$.
\end{remark}

Now we finally give the definition of $I_X(\varphi,\xi,s)$.

\begin{definition} \label{defi:I_X_in_general}
    Pick $u \in C_c^\infty(\omega)$ with $u(a)=1$. Let $\varphi \in \cS(\widetilde{\gl_n}(F))$, and let $s \in \C$ which is not a pole of $L_{X}(s,\xi)$ we define
    \[
        I_X(\varphi,\xi,s) := I_X(\varphi \cdot (u \circ q),\xi,s),
    \]
    where the right hand side is defined as ~\eqref{eq:definition_I_X_on_Omega}.
\end{definition}

By Lemma ~\ref{lem:independence_of_orbit}, the definition of $I_X(\varphi,\xi,s)$ is independent of choice of $u$. We now study the properties of $I_{X}(\varphi,\xi,s)$ 
\begin{proposition} \label{prop:local_general}
    We have the following assertions 
    \begin{enumerate}
        \item For $\varphi \in \cS(\widetilde{\gl_n}(F))$, as a function of $s$, $I_X(\varphi,\xi,s)$ is meromorphic on $\C$, and its pole is contained in the pole of $L_X(s,\xi)$.
        \item For $s \in \C$ which is not a pole of $L_X(s,\xi)$, the resulting linear map
        \[
         \cS(\widetilde{\gl_n}(F)) \to \C, \quad \varphi \mapsto I_X(\varphi,\xi,s).
        \]
        is continuous.
        \item  As a function of $s$
        \[
            \frac{I_X(\varphi,\xi,s)}{L_X(s,\xi)}
        \]
    is entire for any $\varphi \in \cS(\widetilde{\gl_n}(F))$.
        \item If $F$ is non-archimedean with ring of integer $\cO$, $F_i$ are unramified over $F$, $\xi$ and $\eta$ are unramified and $X \in \widetilde{\gl_n}_{,\pm}(\cO)$ and $a_H \in \cA_H'(\cO)$, $\vol(\cO_F) = \vol(\cO_{F_i})=1$, $X$ can be written of the form $\iota_\fh(X_H) \cdot g$, with $X_H \in \fh_{0,\mathrm{rs}}(\cO) \times \fh^\varepsilon(\cO)$ and $g \in \GL_n(\cO)$, then
        \[
            I_X(1_{\widetilde{\gl_n}(\cO)},s,\xi) =  L_X(s,\xi).
        \]
        \item  For any $s$ which is not a pole of $L_X(s,\xi)$, the distribution $I_X(\cdot,s,\xi)$ satisfies the following two properties:
        \begin{itemize}
            \item  $I_X(\cdot,\xi,s)$ is a $\xi \eta \lvert \cdot \rvert^s$-stable distribution.
            \item $I_X(\cdot,\xi,s)$ is supported on the closure of $\GL_n(F)$ orbit of $X$.
        \end{itemize}
    \end{enumerate}
\end{proposition}

\begin{proof}
    Without loss of generality, we can assume $\iota_\fh(X_H)=X$ (otherwise, we can choose a different $i$ and $i^0$), part (1) and (2) then follows from Corollary ~\ref{cor:local_central_and_rss} (1), (2) and the fact that map $\cS(\widetilde{\gl_n}(F)) \to \cS(\widetilde{\fh}(F)),  \varphi \mapsto \varphi_{H,s}$ . For part (3), note that for a fixed $s=s_0$, the analytic property of  $I_X(\varphi,\xi,s)/L_X(s,\xi)$ only depends on $\varphi_{H,s_0}$, therefore (3) follows from Corollary ~\ref{cor:local_central_and_rss} (3). Under the assumption of (4),  by ~\cite{CZ21}*{Lemme 3.4.5.1}, for all $Y \in \widetilde{\fh}(F)$ and $g \in G(F)$, we have
        \[
        1_{\widetilde{\gl_n}(\cO)}(\iota_\fh(Y) \cdot g) = \int_{H(F)} 1_{\widetilde{\fh}(\cO)}(Y \cdot h) \cdot 1_{\GL_n(\cO)}(h^{-1}g) \rd h.
        \]
    Therefore, if we choose any $u \in C_c^\infty(\omega)$ with $u(a)=1$ and let $\varphi = 1_{\widetilde{\gl_n}} \cdot (u \circ q)$. Then $\varphi' : (Y,g) \mapsto  1_{\widetilde{\fh}(\cO)}(Y) u(q(\iota_\fh(Y))) 1_{\GL_n(\cO)}(g)$ satisfies the equation ~\eqref{eq:varphi'}. Then 
    \[
        \varphi_{H,s}(Y) = \int_{\GL_n(F)} 1_{\widetilde{\fh}(\cO)}(Y) u(\iota_\fh(Y)) 1_{\GL_n(\cO)}(g) \chi(g) \lvert g \rvert^s \rd g = 1_{\widetilde{\fh}(\cO)}(Y) u(q(\iota_\fh(Y))).
    \]
    By Lemma ~\ref{lem:independence_of_orbit} and our assumptions, $I_X(1_{\widetilde{\gl_n}(\cO)},\xi,s) = I_X(\varphi,\xi,s) =  I_{X_H}(\varphi_H,\xi,s) = I_{X_H}(1_{\widetilde{\fh}(\cO)},\xi,s)$. Therefore (4) follows from Corollary ~\ref{cor:local_central_and_rss} part (4). For (5), to show $I_X(\cdot,\xi,s)$ is a $\xi \eta \lvert \cdot \rvert^s$-stable distribution, let $\varphi \in \cS(\widetilde{\gl_n}(F))$ be a $\xi \eta \lvert \cdot \rvert^s$-unstable function, replacing $\varphi$ by $\varphi \cdot (u \circ q)$, we can assume $\varphi \in \cS(\Omega)$, then direct computation shows $\varphi_{H,s}$ is also $\xi \eta \lvert \cdot \rvert^s$-unstable, therefore, the result follows from the fact that $I_{X_H}$ is a $\xi \eta \lvert \cdot \rvert^s$-stable distribution. The fact that it supports in the closure of the orbit of $X$ follows from Lemma ~\ref{lem:independence_of_orbit}. 

\end{proof}

\begin{remark}
    When $X$ is central or regular semisimple, we then have two definitions of $I_X$ in Subsection ~\ref{subsec:central_and_rss} and ~\ref{subsec:local_I_general} respectively, these two definitions coincide because, when $a \in \cA_{\mathrm{rs}}(F)$, the descent construction associated with $a$ is given by $H=H_0=\GL_n$ and $\iota_H = \iota_\fh = \mathrm{id}$. And when $a \in \cA(F)$ is central, the descent construction associated to $a$ is given by $H_0 = \{1\}$ and $k=1, F_1 = F,n_1=n$, therefore we also have $H=\GL_n$ and $\iota_H = \iota_\fh = \mathrm{id}$.
\end{remark}

When $X \in \widetilde{\gl_n}_{,+}(F)$ or $X \in \widetilde{\gl_n}_{,-}(F)$, the distribution can be defined in a more direct way.

\begin{proposition} \label{prop:plus_or_minus_regular_direct_def}
    Let $X \in \widetilde{\gl_n}_{,+}(F)$ (resp. $X \in \widetilde{\gl_n}_{,-}(F)$), then for $\varphi \in \widetilde{\gl_n}(F)$ the integral
    \begin{equation} \label{eq:plus_minus_reg_direct_def}
         \int_{\GL_n(F)} \varphi(X \cdot g) \xi(g) \eta(g) \lvert g \rvert^s \rd g
    \end{equation}  
    is absolutely convergent when $\mathrm{Re}(s)>-1+\frac 1n$ (resp. $\mathrm{Re}(s)<1-\frac 1n$), and coincides with $I_X(\varphi,\xi,s)$ in the convergence domain.
\end{proposition}

\begin{proof}
    We choose a non-negative function $\psi' \in C_c(\Omega_H \times \GL_n(F))$ such that for any $(Y,g) \in \Omega_H \times \GL_n(F)$, we have
    \[
        \lvert \varphi(\iota_\fh(Y) \cdot g) \rvert = \int_{H(F)} \psi'(Y \cdot h,h^{-1} g) \rd h.
    \]
    For $s \in \C$, let $\psi_{H,s} \in C_c(\widetilde{\fh}(F))$ be the non-negative function defined by
    \[
        \psi_{H,s}(Y) = \int_{\GL_n(F)} \psi'(Y,g) \lvert \det g \rvert^{\mathrm{Re}(s)} \rd g.
    \]
    Then the absolute convergence of the integral ~\eqref{eq:plus_minus_reg_direct_def} is equivalent to the absolute convergence of 
    \[
        \int_{H(F)} \psi_{H,s}(Y \cdot h) \lvert \det h \rvert^s \rd h
    \]
    which follows from and Proposition ~\ref{prop:regular_orbit_tilde_gl_n} (2) and Proposition ~\ref{prop:local_central} (although in Proposition ~\ref{prop:local_central}, we require the functions are Schwartz, but the proof of part (1) only needs the condition that the functions are in $C_c(\widetilde{\fh}(F))$). After the knowledge of absolute convergence, Fubini theorem implies that the integral ~\eqref{eq:plus_minus_reg_direct_def} coincides with $I_X(\varphi,\xi,s)$ in the domain of convergence.
\end{proof}

\subsubsection{A variant}

We provide a variant of the construction above, replacing $\widetilde{\gl_n}$ by $\gl_{n+1}$. Recall that $\cB = \gl_{n+1}/\GL_n$ and we have a canonical identification $\cB \cong \cA \times F$ induced by the map ~\eqref{eq:isomorphism gl and tilde gl}.

For a regular element $X = \begin{pmatrix}
        A & v \\ u & d
    \end{pmatrix} \in \gl_{n+1}(F) $, we define
    \[
        L_X(s,\xi) := L_{(A,v,u)}(s,\xi),
    \]
    where the right-hand side is defined in ~\eqref{eq:general_L}. For $\varphi \in \cS(\gl_{n+1}(F))$, and $s \in \C$ which is not a pole of $L_X(s,\xi)$, we put
    \[
        I_X(\varphi,\xi,s) := I_{(A,v,u)}(\varphi_{d},\xi,s),
    \]
    where $\varphi_{d} \in \cS(\widetilde{\gl_n}(F))$ is defined by
    \[
        \varphi_{d}(A',v',u') = \varphi \begin{pmatrix}
            A' & v' \\ u' & d
        \end{pmatrix}.
    \]

\begin{remark} \label{rmk:variant_similar_property}
By definition, it is easy to see that $I_X$ admits similar properties as those listed in Proposition ~\ref{prop:local_general}.
\end{remark}

\subsection{Orbital integrals on the group} \label{subsec:local_group}

Let $f \in \cS(\rG'(F))$ and $s \in \C$, we define $f^\rS_s \in \cS(\rS(F))$ by
\begin{equation} \label{eq:defi_f^S}
    f^\rS_s(x)= \begin{cases}  \int_{\rH_1(F)} \int_{\rH_{2,n+1}(F)} f(h_1^{-1},h_1^{-1}\nu^{-1}(x)h_{2,n+1}) \xi(h_1) \lvert \det h_1 \rvert^s \mu(\nu^{-1}(x)h_{2,n+1})  ,& n \text{ odd}, \\
               \int_{\rH_1(F)} \int_{\rH_{2,n+1}(F)} f(h_1^{-1},h_1^{-1}\nu^{-1}(x)h_{2,n+1}) \xi(h_1) \lvert \det h_1 \rvert^s  ,& n \text{ even}.  \end{cases}
\end{equation}
Fix $\gamma=(\gamma_n,\gamma_{n+1}) \in \rG_{\mathrm{reg}}'(F)$, let $x=\alpha(\gamma) \in \rS(F)$, and let $b = q(x) \in \rB(F)$. Choose $\sigma \in E^\times$ with $\sigma \sigma^\mathtt{c}=1$ such that $x \in \rS^\sigma(F)$, so the Cayley transform $\fc_\sigma$ is defined. We use $\fc_\sigma$ to transport $f_s^\rS$ to a Schwartz function $f_s^{\gl}$ on $\gl_{n+1}(F)$: pick $u \in C_c^\infty(\rB^\sigma(F))$ such that $u(b)=1$, and define $f^{\gl}_s \in \cS(\gl_{n+1}^\tau(F))$ by
\begin{equation} \label{eq:defi_f^gl}
     f^{\gl}_s(X) = (u \cdot f^\rS_s)(\fc_\sigma(X)),
\end{equation}
and extends $f^{\gl}_s$ to an element of $\cS(\gl_{n+1}(F))$ using extension by zero, we still denote it by $f^{\gl}_s$.

We put an $L$-factor by
\[
    L_\gamma(s,\xi) := L_{\fc_\sigma^{-1}(x)}(s,\xi),
\]
and for $s \in \C$ which is not a pole of $L_\gamma(s,\xi)$, we put
\[
    I^\sigma_\gamma(f,\xi,s) := \begin{cases}
        I_{\fc_\sigma(x)}(f^{\gl}_s,\xi,s) \mu(\gamma_n \gamma_{n+1}^{-1}) & n\text{ odd}  \\ I_{\fc_\sigma(x)}(f^{\gl}_s,\xi,s) & n\text{ even}
    \end{cases}.
\]
Note that, by Proposition ~\ref{prop:local_general} (5) (see also Remark ~\ref{rmk:variant_similar_property}), the definition of $I^\sigma_\gamma$ does not depend on the choice of $u \in C_c^\infty(\rB^\sigma(F))$. 

We say $f \in \cS(\rG'(F))$ is $(\xi \lvert \cdot \rvert^s, \eta)$-unstable, if for any $\gamma \in \rG'_{\mathrm{rs}}(F)$, the orbital integral
\[
    \int_{\rH_1(F)} \int_{\rH_2(F)} f(h_1^{-1} \gamma h_2) \lvert \det h_1 \rvert^s \xi(h_1) \eta(h_2) \rd h_1 \rd h_2
\]
vanishes. A distribution $I:\cS(\rG'(F)) \to \C$ is called $(\xi \lvert \cdot \rvert^s, \eta)$-stable if $I(f)=0$ for any $f$ that is $(\xi \lvert \cdot \rvert^s, \eta)$-unstable.

We now list the properties of $I^\sigma_\gamma$. 
\begin{proposition} \label{prop:local_group}
    We have the following assertions
    \begin{enumerate}
        \item The function $s \mapsto L_\gamma(f,\xi,s)$ is meromorphic, and the pole is contained in the pole of $L_\gamma(s,\xi)$.
        \item If $s$ is not a pole of $L_\gamma(s,\xi)$, then $I^\sigma(\cdot,\xi,s)$ is continuous.
        \item As a function of $s$, for any $f \in \cS(\rG'(F))$
        \[
           I^{\sigma,\natural}(f,\xi,s):= \frac{I^\sigma_\gamma(f,\xi,s)}{L_\gamma(s,\xi)}
        \]
        is entire.
        \item  If $F$ is non-Archimedean with ring of integer $\cO$, suppose that we can choose $\sigma$ such that
        \begin{itemize}
                \item $\xi,\eta$ and $\mu$ are unramified,
                \item $2 \tau \sigma \in \cO^\times$,
                \item $\gamma \in \rG'(\cO)$,
                \item $a_H \in \cA_H'(\cO)$,
                \item $X = \fc_\sigma^{-1}(x)$ satisfies the condition of (4) in Proposition ~\ref{prop:local_general}.
            \end{itemize}
            Then $I^\sigma_\gamma(1_{\rG'(\cO)},\xi,s) = L_\gamma(s,\xi)$
        \item For any $s \in \C$ which is not a pole of $L_\gamma(s,\xi)$, $I^\sigma_\gamma(\cdot,\xi,s)$ is a $\xi \eta \lvert \cdot \rvert^s$-stable distribution and supported on the closure of the $\rH_1(F) \times \rH_2(F)$ orbit of $\gamma$. 
    \end{enumerate}
\end{proposition}

\begin{proof}
    Part (1) follows from Proposition ~\ref{prop:local_general} (1) and the fact that $s \mapsto f^{\gl}_s$ is holomorphic and valued in $\cS(\gl_{n+1}(F))$. (2) and (3) also follow from their corresponding part of Proposition ~\ref{prop:local_general}. Put $f = 1_{\rG'(\cO)}$ Under the condition of (4), one directly checks that $f^\rS_s = 1_{\rS(\cO)}$ for any $s \in \C$ and $I^\sigma_\gamma(f,\xi,s) = I_{\fc_\sigma(x)}(1_{\widetilde{\gl_n}(\cO)},\xi,s)$. Therefore, part (4) follows from Proposition ~\ref{prop:local_general} (4), and part (5) also follows from Proposition ~\ref{prop:local_general} (5). 
\end{proof}

\begin{remark}
    It is expected that $I_\gamma^\sigma$ depends only on $\gamma$ (but does not depend on $\sigma$), for example, the following lemma shows that this holds when $\gamma \in \rG'_+(F)$ or $\rG'_-(F)$, but the author does not have a proof now.

    However, we can show that $L_\gamma$ is independent of the choice of $\sigma$, for the following reason: for any $s_0 \in \C$ the order of the pole of $L_\gamma(s,\xi)$ at $s=s_0$ is the maximum of that of $I^\sigma_\gamma(f,\xi,s)$ as $f$ runs through elements in $\rS(\rG'(F))$. This follows from Corollary ~\ref{cor:local_central_and_rss} (3), and this quantity is independent of the choice of $\sigma$.
\end{remark}

When $\gamma \in \rG'_+(F)$ or $\gamma \in \rG'_-(F)$, $I_\gamma$ can be described more explicitly. The following lemma can be proved using the same strategy as that in the proof of Proposition ~\ref{prop:plus_or_minus_regular_direct_def}.
\begin{lemma} \label{lem:local_group_central_rss_description}
    We have the following assertions:
    \begin{itemize}
        \item If $\gamma$ is regular semisimple, then
        \[
            I^\sigma_\gamma(f,\xi,s) = \int_{\rH_1(F)} \int_{\rH_2(F)} f(h_1^{-1} \gamma h_2) \xi(h_1) \lvert \det h_1 \rvert^s \eta(h_2) \rd h_1 \rd h_2,
        \]
        where the integral converges absolutely for any $s \in \C$.
        \item If $\gamma \in \rG'_+(F)$ (resp. $\rG'_-(F)$), then the integral
        \[
             \int_{\rH_1(F)} \int_{\rH_2(F)} f(h_1^{-1} \gamma h_2) \xi(h_1) \lvert \det h_1 \rvert^s \eta(h_2) \rd h_1 \rd h_2,
        \]
        is absolutely convergent for $s \in \cH_{<-1+\frac 1n}$ (resp. $s \in \cH_{>1-\frac 1n}$) and equals $I^\sigma_\gamma(f,\xi,s)$ there.  
    \end{itemize}
    In particular, in these cases, the definition of $I^\sigma_\gamma(f,\xi,s)$ does not depend on the choice of $\sigma$
\end{lemma}

\section{Global Theory II: Regular supported functions} \label{sec:global_II}

In this section, we set back to our notation in Section ~\ref{sec:global}. So $F$ will be a number field. The main result in this section is Theorem ~\ref{thm:global_II}, which computes the geometric side of Jacquet-Rallis RTF when the test function is regular supported at a place.

\subsection{Global orbital integral} \label{subsec:global_orbital_integral}

Fix an additive character $\psi: F \backslash \bA_F \to \C^\times$, for each place $v$ of $F$, the local component $\psi_v:F_v \to \C^\times$ of $\psi$ determines measures on various algebraic groups over $F_v$ as discussed in ~\ref{subsubsec:measure_local}. If $F'/F$ is a finite extension, we put $\psi \circ \mathrm{Tr}_{F'/F}$ as the additive character on $\bA_{F'}$ and it can be used to define measures on various local groups over $F'_{w}$ where $w$ is a place of $F'$. For any reductive group $G$ over $F$, we write $\Delta_{G}^*$ for the leading Laurant coefficient of the Artin-Tate $L$-function $L_G(s)$ associated to $G$ (See ~\cite{BPCZ}*{2.3.2}). For $m>0$, the Tamagawa measure on $\GL_m(\bA_{F'})$ is given by
\[
    \rd g = \Delta_{\GL_{m,F'}}^{*,-1}  \prod_w \rd g_w.
\]
In this case, $\Delta_{\GL_{m,F'}} = \zeta^*_{F'}(1)\zeta_{F'}(2)\cdots \zeta_{F'}(m)$ where $\zeta^*_{F'}(1)$ denotes the residue of $\zeta_F(s)$ (the completed Dedekind zeta function) at $s=1$, and the product runs through the places of $F'$. 

We consider the three cases for a group $G$ acting on $X$ as in ~\eqref{eq:four_cases}, excluding the case $(G,X)=(\GL_n,\rS)$. So the notation $G$ and $X$ will mean any pair $(G,X)$ in these three cases. Let $\xi: \bA_E^\times \to \C^\times$ be a strictly unitary character. For a place $v$ of $F$, an element $\gamma \in X_{\mathrm{reg}}(F)$ and $f \in \cS(X(F_v))$, in Section ~\ref{sec:local}, we have defined the local factor $L_\gamma(s,\xi_v)$ and regularized orbital integral $I_\gamma(f,\xi_v,s)$ (If $G=\rH_1 \times \rH_2$, it depends on a choice of a norm 1 element $\sigma$ in $E_v$, which we assume it comes from a global element, and it should be denoted by $I_\gamma^\sigma$, similar for the notations later). We will denote them by $L_{\gamma,v}$ and $I_{\gamma,v}$ here. Using the compatibility of descent and base change (see the last sentences of Subsection ~\ref{subsec:descent}), the local factor $L_{\gamma,v}$ is the local component of a global $L$-factor $L_\gamma(s,\xi)$ 
 (which is defined by the global analogue of ~\eqref{eq:L_central_rss}). For $f \in \cS(X(F_v))$, we put
\[
    I_{\gamma,v}^\natural(f,\xi_v,s) := \frac{I_{\gamma,v}(f,\xi_v,s)}{L_{\gamma,v}(s,\xi_v)}.
\]
For $f \in \cS(X(\bA))$, if $f = \otimes f_v$ is factorizable, by results in Section ~\ref{sec:local}, $I_{\gamma,v}^\natural(f_v,\xi_v,s)=1$ for almost all places $v$, hence we can define
\[
    I^\natural_\gamma(f,\xi,s) :=  \prod_{v} I^\natural_{\gamma,v}(f_v,\xi_v,s).
\]
The definition of $I_\gamma^\natural$ extends by continuity to any $f \in \cS(X(\bA))$. We then define
\begin{equation} \label{eq:I_gamma_global}
     I_\gamma(f,\xi,s) := \Delta_G^{*,-1} L_{\gamma}(s,\xi) I_\gamma^\natural(f,\xi,s).
\end{equation}

By definition, $I_\gamma(f,\xi,s)$ depends only on the value of $f$ on the $G(\bA)$-orbit of $\gamma$ and for any $g \in G(F), I_{\gamma \cdot g}(f) = I_{\gamma}(f)$, and when $X=\rG', G = \rH_1 \times \rH_2$, it depends on the choice of norm one element $\sigma$.

The construction directly generalizes to the product of cases we consider. Fix $a \in \cA(F)$ we, therefore, have $\widetilde{\fh}$ as discussed in Subsection ~\ref{subsec:descent}. Then for any $\varphi \in \cS(\widetilde{\fh}(\bA))$ and regular element $Y \in \widetilde{\fh}(F)$, we can define $I_Y(\varphi,\xi,s)$ in the same way.

\subsection{Main result}

We now state our main result. 
\begin{theorem} \label{thm:global_II}
    Let $f = f_v  f^v \in \cS(\rG'(\bA))$ with $f_v \in \cS(\rG'(F_v)), f^v \in \cS(\rG'(\bA^v))$ and $\mathrm{supp}(f_v) \subset \rG'_{\mathrm{reg}}(F_v)$, then for $b \in \rB(F)$, we have
    \begin{equation} \label{eq:global_II_main}
      I_b(f,\xi,s) = \sum_{\gamma} I^\sigma_\gamma(f,\xi,s).
    \end{equation}    
    Where $I_b$ is the distribution from relative trace formula in Theorem ~\ref{thm:coarse Bessel} (2), and $I^\sigma_\gamma$ is defined in $\eqref{eq:I_gamma_global}$, the sum runs through all the representative of $\rH_1(F) \times \rH_2(F)$ orbits in $\rG'_b(F) \cap \rG'_{\mathrm{reg}}(F)$.
\end{theorem}

We will derive Theorem ~\ref{thm:global_II} from a Lie algebra version of it.

\begin{proposition} \label{prop:global_II_Lie}
    Let $\fg$ denote either $\gl_{n+1}$ or $\widetilde{\gl_n}$. Let $\varphi = \varphi_v  \varphi^v \in \cS(\fg(\bA))$ with $\varphi_v \in \cS(\fg(F_v)), \varphi^v \in \cS(\fg(\bA^v))$ and $\mathrm{supp}(\varphi_v) \subset \fg_{\mathrm{reg}}(F_v)$, then for $a \in (\fg/\GL_n)(F)$, we have
    \begin{equation} \label{eq:global_II_main_Lie}
      I_a(\varphi,\xi,s) = \sum_{X} I_X(\varphi,\xi,s).
    \end{equation}    
    Where $I_a$ is the distribution from relative trace formula (see Theorem ~\ref{thm:coarse_Lie} and Remark ~\ref{rmk:trace_formula_for_tilde_gl_n}) and $I_X$ is defined in $\eqref{eq:I_gamma_global}$, the sum runs through all the representative of $\GL_n(F)$ orbits in $\fg_{a}(F) \cap \fg_{\mathrm{reg}}(F)$.
\end{proposition}

Fix $a \in \cA(F)$. By the descent construction in Subsection ~\ref{subsec:descent}, there is $\widetilde{\fh}$ which is a product of $\widetilde{\gl_{n_i,F_i}}$ of smaller size such that there is $a_H \in \cA_H'$ maps to $a$ under the \'{e}tale map $\iota_\cA:\cA_H' \to \cA$. We first show a version of Proposition ~\ref{prop:global_II_Lie} for $a_H$.

\begin{lemma} \label{lem:global_II_rss_central}
    Let $\varphi = \varphi_v  \varphi^v \in \cS(\widetilde{\fh}(\bA))$ with $\varphi_v \in \cS(\widetilde{\fh}(F_v)), \varphi^v \in \cS(\widetilde{\fh}(\bA^v))$ and $\mathrm{supp}(\varphi_v) \subset \widetilde{\fh}_{\mathrm{reg}}(F_v)$, then we have
    \begin{equation} \label{eq:global_II_main_Lie}
      I_{a_H}(\varphi,\xi,s) = \sum_{X} I_X(\varphi,\xi,s).
    \end{equation}    
    Where $I_{a_H}$ is the distribution from relative trace formula (see Remark ~\ref{rmk:trace_formula_for_tilde_gl_n}), the sum runs through all the representative of $H(F)$ orbits in $\widetilde{\fh}_{a_H}(F) \cap \widetilde{\fh}_{\mathrm{reg}}(F)$.
\end{lemma}

\begin{proof}
 We can assume $\varphi$ is a pure tensor $\otimes \varphi_v$ and each $\varphi_v$ is of the form $\varphi_v=\otimes_{i=0}^k \varphi_{v,i}$ by continuity, where $\varphi_{v,i} \in \cS(\gl_{n_i}(F_i))$. Put $\varphi_i = \otimes \varphi_{v,i} \in \cS(\widetilde{\gl_{n_i}}(\bA_{F_i}))$, so that $\varphi = \otimes \varphi_i$. Therefore, 
 \[
    I_{a_H}(\varphi,\xi,s) = \prod_{i=0}^k I_{a_i}(\varphi_i,\xi_i,s),
 \]
 and
 \[
    \sum_{X} I_X(\varphi,\xi,s) = \prod_{i=0}^k \left( \sum_{X_i} I_{X_i}(\varphi_i,\xi_i,s)  \right),
 \]
 where $\xi_i := \xi \circ \mathrm{Nm}_{(E \otimes_F F_i)/E}$, and in the sum on the right-hand side, for each $0 \le i \le k$, $X_i$ runs through the regular orbits in $\widetilde{\gl_{n_i,F_i}}_{,a_i}(F_i)$. Therefore up to possibly replacing $F$ by $F_i$, we are reduced to prove Proposition ~\ref{prop:global_II_Lie}, for $\fg=\widetilde{\gl_n}$ and $a$ is either central or regular semisimple. 
 
 If $a$ is regular semisimple, then both sides of ~\eqref{eq:global_II_main_Lie} equals to
 \[
    \int_{\GL_n(\bA)} \varphi(X \cdot g)  \xi(g) \eta(g) \lvert \det g \rvert^s \rd g,
 \]
therefore ~\eqref{eq:global_II_main_Lie} holds. We now assume $a$ is central. By Lemma ~\ref{lem:regular_nilpotent_orbit}, all regular orbits in the fiber of $a$ are $O^+(a)$ and $O^-(a)$. Note that we have the following open cover of $\widetilde{\gl_n}_{,\mathrm{reg}}(F_v)$:
    \[
        \widetilde{\gl_n}_{,\mathrm{reg}}(F_v) = \widetilde{\gl_n}_{,+}(F_v) \cup \widetilde{\gl_n}_{,-}(F_v) \cup (\widetilde{\gl_n}_{,\mathrm{reg}} \setminus \widetilde{\gl_n}_{,a} )(F_v).
    \]
     We choose functions $\alpha_0,\alpha_+,\alpha_- \in C^\infty(\widetilde{\gl_n}_{,\mathrm{reg}}(F_v))$ (a partition of unity), such that
    \begin{itemize}
        \item $\alpha_+ + \alpha_- + \alpha_0 = 1$,
        \item $\mathrm{supp}(\alpha_0) \subset  (\widetilde{\gl_{n}}_{,\mathrm{reg}} \setminus \widetilde{\gl_n}_{,a})(F_v)$, $\mathrm{supp}(\alpha_+) \subset \widetilde{\gl}_{n,+}(F_v), \mathrm{supp}(\alpha_-) \subset \widetilde{\gl}_{n,-}(F_v)$, 
        \item for any $f \in \cS(\widetilde{\gl_n}_{,\mathrm{reg}}(F_v))$ and any $i \in \{0,+,-\}$, $\alpha_i f \in \cS(\widetilde{\gl_n}_{,\mathrm{reg}}(F_v))$.
    \end{itemize}
    The existence of these functions is easy for non-Archimedean $v$ and for $v$ Archimedean, see ~\cite{AG}*{Theorem 4.4.1}.

    For $i \in \{0,+,-\}$, let $\varphi_i \in \cS(\widetilde{\gl_n}(\bA))$ be the function $(\alpha_i \varphi_v) \cdot \varphi^v$. Then since the distribution $I_{a}(\cdot,\xi,s)$ is supported on $\widetilde{\gl_n}_{,a}(\bA)$, we see that
    \begin{equation} \label{eq:proof_of_global_II_0}
        I_{a}(\varphi_0,\xi,s) = 0.
    \end{equation}
    Since $\varphi_+$ satisfies the assumption of Theorem ~\ref{thm:main_infinitesimal}, by part (3) of this theorem, we see that for $s \in \cH_{<-1}$
    \begin{equation} \label{eq:proof_of_global_II_+}
        I_a(\varphi_+,\xi,s) = \int_{\GL_n(\bA)} \varphi_+(X \cdot g) \xi(g) \eta(g) \lvert \det g \rvert^s \rd g = \int_{\GL_n(\bA)} \varphi_+(X \cdot g) \xi(g) \eta(g) \lvert \det g \rvert^s \rd g,
    \end{equation}
    where the integral in ~\eqref{eq:proof_of_global_II_+} is absolutely convergent. That is, the infinite product
    \[
       \Delta_{\GL_n}^{*,-1} \prod_v I_{X_+,v}(\varphi_v,\xi_v,s)
    \]
    is convergent when $s \in \cH_{<-1}$, and is therefore equals to $I_{X_+}(\varphi,\xi,s)$. Since both $I_a(\varphi_+,\xi,s)$ and $I_{X_+}(\varphi,\xi,s)$ are meromorphic in $s$, we see that whenever they are holomorphic, we have $I_a(\varphi_+,\xi,s)=I_{X_+}(\varphi,\xi,s)$. By a similar argument, $I_a(\varphi_-,\xi,s)=I_{X_-}(\varphi_-,\xi,s)$. Together with ~\eqref{eq:proof_of_global_II_0}, we see that
    \[
        I_a(\varphi,\xi,s) = I_a(\varphi_+,\xi,s)+I_a(\varphi_-,\xi,s)+I_a(\varphi_0,\xi,s) = I_{X_+}(\varphi,\xi,s)+I_{X_-}(\varphi,\xi,s),
    \]
    which is exactly the equality we want for the central element $a$.
\end{proof}

Now we give the proof of Proposition ~\ref{prop:global_II_Lie}.

\begin{proof}[Proof of Proposition ~\ref{prop:global_II_Lie}]
    We first show the case when $\fg=\widetilde{\gl_n}$. We will use the notation introduced in Subsection ~\ref{subsec:descent}. 

    Let $\tS$ be a sufficiently large finite set of places of $F$, containing all archimedean places. For any $v \in \tS$, pick an open neighbourhood $\omega_v$ of $a$ in $\cA(F_v)$ and a neighbourhood $\omega_{H,v}$ of $a_H$ in $\cA_H(F_v)$ such that $\omega_v$ and $\omega_{H,v}$ are semi-algebraic if $v$ is Archimedean and $\iota_\cA: \cA_H' \to \cA$ induces an isomorphism $\omega_{H,v} \to \omega_{H}$. Let $\omega_\tS := \prod_{v \in \tS} \omega_v$ and $\omega_{H,\tS} := \prod_{v \in \tS} \omega_{H,v}$. We choose $u \in C_c^\infty(\omega_\tS)$ such that $u(a)=1$. Denote by $\Omega_\tS$ (resp. $\Omega_{H,\tS}$) the preimage of $\omega$ (resp. $\omega_{H,v}$) under the quotient map $q$ (resp. $q_H$).

    Since both sides of ~\eqref{eq:global_II_main_Lie} only depend on the value of $\varphi$ on $\widetilde{\gl_n}_{,a}(\bA)$. After replacing $\varphi$ by $\varphi \cdot (u \circ q)$ and enlarging $\tS$, we can assume $\varphi = \varphi_\tS \otimes 1_{\widetilde{\gl_n}(\cO_F^\tS)}$, where $\varphi_\tS \in \cS(\widetilde{\gl_n}(F_\tS))$. We can choose $\varphi' \in \cS(\Omega_{H,\tS} \times \GL_n(F_\tS))$ (see Subsection ~\ref{subsec:local_I_general}) such that for any $(X,g) \in \Omega_{H,\tS} \times \GL_n(F_\tS)$, we have
    \[
        \varphi(\iota_{\fh}(X) \cdot g) = \int_{H_\tS(F)} \varphi'(Xh,\iota_H(h)^{-1}g) \rd h.
    \]
    For any $X_H \in \Omega_{H,\tS}$ and $s \in \C$, we put
    \[
        \varphi_{H,\tS,s}(X_H) = \int_{\GL_n(F_\tS)} \varphi'(X_H,g) \xi_\tS(g) \eta_\tS(g) \lvert \det g \rvert^s \rd g.
    \]
    Define $\varphi_{H,s} \in \cS(\widetilde{\fh}(\bA))$ by $\varphi_{H,s} = \varphi_{H,\tS,s} \otimes 1_{\widetilde{\fh}(\cO_F^\tS)}$. Up to enlarging $\tS$, by ~\cite{CZ21}*{Th\'eor\`eme 6.4.6.1}, we have
    \[
        I_a(\varphi,\xi,s) = \Delta_{\GL_n}^{*,-1} \Delta_{H}^{*} I_{a_H}(\varphi_{H,s},\xi,s).
    \]
    (In \emph{loc. cit}, this result is only proved for the case when $\xi$ is trivial and $s=0$, but the same proof works for general case). By our construction of regularized orbital integral $I_X$ in Subsection ~\ref{subsec:local_I_general}, for any regular element $X_H \in \widetilde{\fh}(F)$, we have
    \[
        I_{X_H}(\varphi_{H,s},\xi,s) = \Delta_{\GL_n}^{*,-1} \Delta_{H}^{*}I_X(\varphi,\xi,s).
    \]
    Therefore, Proposition ~\ref{prop:global_II_Lie} follows from Lemma ~\ref{lem:global_II_rss_central} and Proposition ~\ref{prop:regular_orbit_tilde_gl_n}.
\end{proof}

Now we can prove Theorem ~\ref{thm:global_II}.
\begin{proof}[Proof of Theorem ~\ref{thm:global_II}]
    Choose $\sigma \in E$ such that $\mathrm{Nm}_{E/F}(\sigma)=1$ and $b \in \rB^\sigma(F)$. So Cayley transform $\fc_\sigma(b) \in \cB(F)$ is defined.

    Choose a sufficiently large set of finite places $\tS$ of $F$. We assume that $f$ is of the form $f_\tS \otimes 1_{\rG'(\cO_F^{\tS})}$, where $f_\tS \in \cS(F_\tS)$, and we define $f^{\gl}_{\tS,s}$ by a similar formula as in ~\eqref{eq:defi_f^gl}, replacing $F$ there by $F_\tS$, and we put $f^{\gl}_s := f^{\gl}_{\tS,s} \otimes 1_{\widetilde{\gl_n}(\cO^{\tS})} \in \cS(\widetilde{\gl_n}(\bA))$. When $\tS$ is sufficiently large, by ~\cite{CZ21}*{Lemme 14.4.4.2}, we have
    \[
        I_b(f,\xi,s) = \Delta_{\rH_1}^{*,-1} \Delta_{\rH_2}^{*,-1} \Delta_{\GL_n}^* I_{\fc_\sigma(b)}(f^{\gl}_s).
    \]
    Let $\gamma \in \rG'(F)$ be regular element, by our construction of $I^\sigma_\gamma(f,\xi,s)$, we have
    \[
        I^\sigma_\gamma(f,\xi,s) = \Delta_{\rH_1}^{*,-1} \Delta_{\rH_2}^{*,-1} \Delta_{\GL_n}^* I_{\fc_\sigma(\alpha(\gamma))}(f^{\gl}_s,\xi,s).
    \]
    Therefore, the result follows from Proposition ~\ref{prop:global_II_Lie}.
\end{proof} 

\section{Local Theory II: Local singular transfer} \label{sec:localII}

  We retain the notation in Section ~\ref{sec:local} so that $F$ is a local field of characteristic 0 and $E/F$ is a quadratic \'{e}tale algebra. Recall that in Section \ref{sec:local}, we have defined, for any $f \in \cS(\rG'(F))$, and any $\gamma \in \rG'_{\mathrm{reg}}(F)$ and character $\xi:E^\times \to \C^\times$, a regularized orbital integral $I_\gamma^\sigma(f,\xi,s)$ and a normalized version $I_\gamma^{\sigma,\natural}(f,\xi,s)$ which depends on a choice of norm one element $\sigma$ in $E$.

In this section, we will study a transfer relation between $I^\sigma_\gamma(f)$ and semisimple orbital integrals on the unitary groups. We will first introduce orbital integral unitary groups and matching in Subsection ~\ref{subsec:semisimple_orbital_integral} and  ~\ref{subsec:transfer}, we then review the descent on $\widetilde{\fu}^V$ in ~\ref{subsec:descent_on_u}, which is a unitary analogue of what we have done in Subsection ~\ref{subsec:descent}. Then we study the Lie algebra version of local singular transfer in Subsection ~\ref{subsec:local_II_lie_algebra} and the group version in Subsection ~\ref{subsec:local_II_group}.

\subsubsection{Notations and measure} \label{subsubsec:measure_u}

 Let $\cH$ denote the isometric classes of $n$-dimensional $E/F$ Hermitian space. Let $(V,h) \in \cH$, denote the \emph{discriminant} of $V$ by $\mathrm{disc}(V)$, which is the determinant of the gram matrix of $V$ in any basis, regarded as an element of $F^\times/\mathrm{Nm}_{E/F}(F^\times)$. Let $\U(V)$ be the unitary group associated to $V$, and let $V+Ee_0$ be the Hermitian space formed by orthogonal direct sum of $V$ and 1-dimensional $E/F$-Hermitian space $(Ee_0,h_0)$ spanned by $e_0$ with $h_0(e_0,e_0)=1$. Define
\[
    \rG^V = \U(V) \times \U(V \oplus Ee_0), \quad \rH^V = (h,\mathrm{diag}(h,1)), h \in \U(V).
\]
So $\rH^V$ is a subgroup of $\rG^V$ isomorphic to $\U(V)$. The group $\rH^V \times \rH^V$ acts on the right on $\rG^V$ by
\[
    g \cdot (h_1,h_2) := h_1^{-1}gh_2.
\]
For any $\gamma \in \rG^V(F)$, let $(\rH^V \times \rH^V)_\gamma$ denote the stabilizer of $\gamma$ under this action. We call $\gamma \in \rG^V(F)$ semisimple (resp. regular semisimple), if its $\rH^V \times \rH^V$ orbit is Zariski closed (resp. is semisimple and has trivial stabilizer). We call any $\U(V)(F)$ orbits of a semisimple element a \emph{semisimple orbit}. We write $\rG^V_{\mathrm{rs}}$ for the open subset of $\rG^V$ consisting of regular semisimple elements. 

 Let $\fu^V$ be the $F$-subspace of $\mathrm{End}_E(V)$ consisting of self-adjoint operators, so it is $\tau$ times the Lie algebra of $\U(V)$, where $\tau$ is any purely imaginary element (i.e. $\mathrm{Tr}_{E/F} \tau=0$) in $E$. We put $\widetilde{\fu}^V := \fu^V \times \mathrm{Res}_{E/F} V$, which is an $F$-vector space with a right $\U(V)$ action by
\[
    (A,v) \cdot g = (g^{-1}Ag,g^{-1}v).
\]
When we want to emphasize the rule of the Hermitian form $h$ on $V$, we will also write $\widetilde{\fu}^h := \widetilde{\fu}^V$.

For $X \in \widetilde{\fu}^V(F)$, we denote its stabilizer under this action by $\U(V)_X$. An element $X \in \widetilde{\fu}^V(F)$ is called semisimple (resp. regular semisimple) if its orbit under the $\U(V)$ action is Zariski closed (resp. is semisimple and has stabilizer). Let $\widetilde{\fu}^V_{\mathrm{rs}}$ denote the open subset of $\widetilde{\fu}^V$ consisting of regular semisimple elements. 

As in ~\ref{subsubsec:measure_local}, we fix an additive character $\psi$ on $F$. For any $m$-dimensional Hermitian space $V'$ over $F$, choose a finite extension $F'/F$ such that $\U(V')$ is split over $F$. (e.g. $F'=F$ if $E$ is split, and $F'=E$ if $E$ is a field). We have an embedding $\U(V') \hookrightarrow \GL(V')$. We put the measure on $\U(V')(F)$ defined by the $F'$-valued differential form $\prod_{i=1}^m L(i,\eta^i) \det(g_{ij})^{-m} \bigwedge dg_{ij}$ on $\GL(V')(F')$ pulled back to $\U(V')(F)$. We put the product measure on $\rG^V(F)$ and put the Haar measure on $\rH^V(F)$ via the natural identification $\rH(V) \cong \U(V)$. Note that when $E=F \times F$, the measure on $\U(V) \cong \GL_n$ coincides with the measure we gave in ~\ref{subsubsec:measure_local}.

Let $F'/F$ be any finite field extension, we put an additive character $\psi'$ on $F'$ by $\psi \circ \mathrm{Tr}_{F'/F}$. So that we can also put a measure on $\U(V')(F')$ for any $(E \otimes_F F')/F'$ Hermitian space $V'$. For $V \in \cH$, let $\gamma \in \rG^V(F)$ or $X \in \widetilde{\fu}^V(F)$ be a semisimple element, we will see in Subsection ~\ref{subsec:descent_on_u} that the stabilizer of $\gamma$ or $X$ is a reductive group which is isomorphic to finitely many product of $\mathrm{Res}_{F'/F} \U(V')$, we choose the measure on $(\rH^V \times \rH^V)_\gamma(F)$ or $\U(V)_X(F)$ to be the product of the measures we have fixed.

We put a non-degenerate bilinear form on $\widetilde{\fu}^V(F)$ by $\langle (A,v), (B,w) \rangle = \mathrm{Tr}(AB) + \mathrm{Tr}_{E/F} h(v,w)$.  For $\varphi^V \in \cS(\widetilde{\fu}^V(F))$, we then define its Fourier transform $\cF \varphi^V \in \cS(\widetilde{\fu}^V(F))$ to be
\[
    \cF \varphi(Y) = \int_{\widetilde{\fu}(F)} \varphi(X) \psi(\langle X,Y \rangle) \rd Y,
\]
where the measure on $\widetilde{\fu}^V(F)$ is chosen to be self-dual.

\subsection{Orbital integral on the unitary groups}
\label{subsec:semisimple_orbital_integral}

For a semisimple element $\gamma \in \rG^V(F)$ and $f \in \cS(\rG^V(F))$, we define its orbital integral by
\[
    J^V_\gamma(f) = J_\gamma(f) := \int_{(\rH^V \times \rH^V)_\gamma(F) \backslash (\rH^V \times \rH^V)(F) } f(\gamma \cdot h) \rd h.
\]
The semisimplicity of $\gamma$ guarantees $h \mapsto \gamma \cdot h$ is a Schwartz function on the orbit of $\gamma$, so the integral is absolutely convergent. $J^V_\gamma$ only depends on the semisimple orbit $\fo$ of $\gamma$. Let $\fo$ be a semisimple orbit, we write
\[
    J_\fo(f) = J_\gamma(f),
\]
where $\gamma$ is any element in $\fo$. When $\gamma$ is regular semisimple, the orbital integral is then given by
\[
    J_\gamma(f) = \int_{\rH^V(F) \times \rH^V(F)} f(\gamma \cdot h) \rd h.
\]

For a semisimple element $X \in \widetilde{\fu}^V(F)$ and for $\varphi \in \cS(\widetilde{\fu}(F))$, we define its orbital integral by
\[
    J_X(\varphi) = J_X^V(\varphi) := \int_{\U(V)_X(F)\backslash \U(V)(F)} \varphi(X \cdot h) \rd h.
\]
The integral is absolutely convergent for the same reason and depends only on the orbit $\U(V)(F)$ of $X$. Let $\fo$ be a semisimple orbit in $\widetilde{\fu}^V(F)$, we define $J_\fo(\varphi)$ to be $J_X(\varphi)$, for any $X \in \fo$. When $X$ is regular semisimple, this reduces to
\[
    J_X(\varphi) = \int_{\U(V)(F)} \varphi(X \cdot h) \rd h.
\]

For $\lambda \in F$, we write $Z_\lambda$ for the element 
$(\lambda \cdot \mathrm{id},0,0) \in \widetilde{\fu}^V$. The subspace of $\widetilde{\fu}^V$ formed by $Z_\lambda$ is the center of $\widetilde{\fu}^V$ under the $\U(V)$ action. Its image in $\cA$ coincides with the center we considered in Subsection ~\ref{subsec:descent}. By definition, $Z_\lambda$ is semisimple and for $\varphi \in \cS(\widetilde{\fu}^V(F))$ we have
\[
    J_{Z_\lambda}(\varphi) = \varphi(Z_\lambda).
\]

\subsection{Transfer of functions}
\label{subsec:transfer}
The GIT quotient $\rG^V/\rH^V \times \rH^V$ (resp. $\widetilde{\fu}^V(F)/\U(V)$) can be canonically identified with $\rB$ (resp. $\cA$), (see ~\cite{CZ21}*{Lemme 15.1.4.1} and Subsection 8.1). For $\gamma^V = (x,y) \in \rG^V(F)$ and $\gamma = (g_n,g_{n+1}) \in \rG'(F)$, we put $s=(g_n^{-1} g_{n+1})(g_n^{-1} g_{n+1})^{\mathtt{c},-1}$ and $g=x^{-1}y$. Then $\gamma$ and $\gamma^V$ have the same image in $\rB(F)$ if and only if $s$ and $g$ have the same characteristic polynomial and $e_{n+1}^t s^i e_{n+1} = h(g^ie_0,e_0)$ for $1 \le i \le n$. 

If we use the identification ~\eqref{eq:gl_GIT_quotient}, then for $(A,v) \in \widetilde{\fu}^V$, the quotient map $q_V:\widetilde{\fu}^V \to \cA$ can be identified with
\begin{equation} \label{eq:GIT_quotient_unitary_lie_algebra}
    (A,v) \mapsto (\mathrm{Tr}(\wedge^i A)_{1 \le i \le n},h(A^iv,v)_{0 \le i \le i-1}).
\end{equation}
    
So that for $X=(A,v,u) \in \widetilde{\gl_n}(F)$ and $X^V=(A',v') \in \widetilde{\fu}^V$, $X$ and $X^V$ have the same image in $\cA(F)$ if and only if $A'$ and $A$ have the same characteristic polynomial and $uA^iv=h(A‘^{i}v’,v‘)$ for $1 \le i \le n$.

These identifications induce bijections
\[
    \rG_{\mathrm{rs}}'(F)/\rH_1(F) \times \rH_2(F) \longleftrightarrow \bigsqcup_{V \in \cH} \rG_{\mathrm{rs}}^V(F)/\rH^V(F) \times \rH^V(F),
\]
and
\begin{equation} \label{eq:matching_of_rss_lie_algebra}
      \widetilde{\gl_n}_{,\mathrm{rs}}(F)/\GL_n(F) \longleftrightarrow \bigsqcup_{V \in \cH} \widetilde{\fu}^V_{\mathrm{rs}}(F)/\U(V)  
\end{equation}

We say $\gamma \in 
\rG_{\mathrm{rs}}'(F)$ (resp. $X \in \widetilde{\gl_n}(F)$) matches with $\gamma^V \in \rG^V_{\mathrm{rs}}(F)$ (resp. $X^V \in \widetilde{\fu}^V_{\mathrm{rs}}(V)$) if they correspond to each other under the bijection above, equivalently, their image in $\rB(F)$ (resp. $\cA(F)$) are the same.

We define the transfer factors
\[
    \Omega^+,\Omega^-: \rG'_{\mathrm{rs}}(F) \to \C^\times, \text{and }\omega^+,\omega^-: \widetilde{\gl_n}_{,\mathrm{rs}}(F) \to \C^\times
\]
by
\[
    \Omega^\varepsilon(\gamma_n,\gamma_{n+1}) = \begin{cases} \mu(\gamma_n^{-1} \gamma_{n+1}) \mu(x)^{-\frac{n+1}{2}}    \mu(\Delta^\varepsilon(x)) & n \text{ odd} \\ \mu(x)^{-\frac n2}
    \mu(\Delta^\varepsilon(x)) & n \text{ even}
    \end{cases},
\]
where $x = \nu(\gamma_n^{-1} \gamma_{n+1})$ and
\[
    \omega^\varepsilon(X) = \eta(\delta^\varepsilon(X)).
\]
where $\varepsilon \in \{+,-\}$.

We say that $f \in \cS(\rG'(F))$ and $(f^V \in \cS(\rG^V(F)))_{V \in \cH}$ are plus-transfer (resp. minus-transfer) of each other, if whenever the element $\gamma \in \rG'_{\mathrm{rs}}(F)$ and $\gamma^V \in \rG^V_{\mathrm{rs}}(F)$ match, we have
\[
    J^V_{\gamma^V}(f^V) = \Omega^+(\gamma) \cdot I_\gamma(f) \text{ (resp. }J^V_{\gamma^V}(f^V) = \Omega^-(\gamma) \cdot I_\gamma(f)\text{)}
\]
We write $f \xleftrightarrow{+} (f^V)_V$ (resp. $f \xleftrightarrow{-} (f^V)_V$) if they are plus-transfer (resp. minus transfer).

Similarly, we define $\varphi \in \cS(\widetilde{\gl_n}(F)) \xleftrightarrow{+} (\varphi^V \in \cS(\widetilde{\fu}^V(F))_{V \in \cH}$ (resp. $\varphi \xleftrightarrow{-} (\varphi^V)_V$) if whenever $X \in \widetilde{\gl_n}_{,\mathrm{rs}}(F)$ and $X^V \in \widetilde{\fu}_{\mathrm{rs}}^V(F)$ match, we have
\[
    J^V_{X^V}(\varphi^V) = \omega^+(X) \cdot I_X(\varphi) \text{ (resp. }J^V_{X^V}(\varphi^V) = \omega^-(X) \cdot I_X(\varphi)\text{)}
\]

\begin{lemma} \label{lem:relation_between_plus_minus_transfer}
    Given $\varphi \in \cS(\widetilde{\gl_n}(F))$ and $\varphi^V \in \cS(\widetilde{\fu}^V(F))$ for each $V \in \cH$. Then
    \[
        \varphi \xleftrightarrow{+} (\varphi^V) \iff \varphi \xleftrightarrow{-} (\eta(\mathrm{disc}(V)) \cdot \varphi^V)  
    \]
\end{lemma}

\begin{proof}
    For $(V,h) \in \cH$ and $X^V=(A,v) \in \widetilde{\fu}^V(F)$, define $d_n(X) = \det(h(A^{i+j-2}v,v)_{1 \le i,j \le n})$. For $X \in \widetilde{\gl_n}(F)$, we have defined $d_n(X)$ in ~\eqref{eq:d_r_on_gl}.

    Now $\varphi \xleftrightarrow{+} (\varphi^V)$ if and only if for any matching of regular semisimple elements $X$ and $X^V$, we have
    \[
        J^V_{X^V}(\varphi^V) = \omega^+(X) I_X(\varphi).
    \]
    Note that $d_n(X)= \det(\delta^+(X) \delta^-(X))$, so the condition above is the same as
    \[
        J^V_{X^V}(\varphi^V) \eta(d_n(X))  = \omega^-(X) I_X(\varphi).
    \]
     The matching of $X$ and $X^V$ implies $d_n(X)=d_n(X^V)$. Moreover, $\eta(d_n(X^V))= \eta(\mathrm{disc}(V))$, since $(v,Av,\cdots,A^{n-1}v)$ is a basis for $V$. So the last equation is equivalent to
    \[
        J^V_{X^V}(\varphi^V) \eta(\mathrm{disc}(V)) = \omega^-(X) I_X(\varphi),
    \]
    which is equivalent to $ \varphi \xleftrightarrow{-} (\eta(\mathrm{disc}(V)) \cdot \varphi^V)$.
\end{proof}

We also recall the deep results of Zhang and Xue: 

\begin{theorem}[Zhang \cite{Zhang14},Xue \cite{Xue19}] \label{thm:transfer_Fourier_transform} 
    Given $\varphi \in \cS(\widetilde{\gl_n}(F))$ and $\varphi^V \in \cS(\widetilde{\fu}^V(F))$ for each $V \in \cH$. Then
    \[
        \varphi \xleftrightarrow{+} (\varphi^V) \iff \cF \varphi \xleftrightarrow{+} \left( \eta(\mathrm{disc}(V))^n \varepsilon(\eta,\frac{1}{2},\psi)^{\frac{n(n+1)}{2}} \cF \varphi^V \right)
    \]
\end{theorem}

\begin{remark} \label{rmk:generalized_transfer}
    Suppose that for each $0 \le i \le k$, we have a field extension $F_i/F$. Write $E_i = F_i \otimes_F E$, and suppose that for each $i$, we have an $n_i$-dimensional $E_i/F_i$ Hermitian space $V_i$. Let
    \[
        \widetilde{\fg} = \prod_{i=0}^k \mathrm{Res}_{F_i/F} \widetilde{\gl_{n_i,F_i}}, \quad \widetilde{\fu} = \prod_{i=0}^k \mathrm{Res}_{F_i/F} \widetilde{\fu}^{V_i}.
    \]
    We write $\eta_i := \eta_{E \otimes F_i/F_i} = \eta \circ \mathrm{Nm}_{F_i/F}$ for the quadratic character on $F_i^\times$. For $X = (X_i) \in \widetilde{\fg}(F)$, let $\omega^\pm(X):= \prod_{i=0}^k \eta_i(\det \delta^\pm(X_i))$.
    
    The orbital integral and transfers then directly generalize to $\varphi \in \cS(\widetilde{\fg}(F))$ and $\varphi^V \in \cS(\widetilde{\fu}(F))$.
\end{remark}

\subsection{Descent on $\tilde{\fu}$}
\label{subsec:descent_on_u}

\subsubsection{Descent construction}
\label{subsubsec:descent_construction_u}

There is a descent construction on $\widetilde{\fu}^V$ which is similar to the descent for $\widetilde{\gl_n}$ as we discussed in Subsection ~\ref{subsec:descent}. Recall that we have a quotient map $q= q^V: \widetilde{\fu}^V \to \cA$. For $X=(A,b) \in \widetilde{\fu}^V(F)$, let $d_r(X) = \det(h(A^{i+j-2}b,b)_{1 \le i,j \le r}$. Then $d_r$ descends to a function on $\cA$ and coincides with the function in ~\eqref{eq:d_r_on_gl}. The functions $d_r$ induce a stratification $\widetilde{\fu}^{V,(r)}$ on $\widetilde{\fu}^V$ in the same way as we discussed after ~\eqref{eq:d_r_on_gl}. 

Fix $a \in \cA^{(r)}(F)$. Let $\cA_0 := \widetilde{\gl_r}/\GL_r, \cA^0 := \widetilde{\gl_{n-r}}/\GL_{n-r}$. As we discussed in Subsection ~\ref{subsec:descent}, $a$ can be written uniquely $\iota(a_0,a^0)$ with $a_0 \in \cA_0(F), a^0 \in \cA^0(F)$, and $a^0$ is determined by a polynomial $P \in F[x]$, which factorize as $P=P_1^{n_1}\cdots P_k^{n_k}$. Let $F_i=F[x]/(P_i)$ and $\alpha_i$ be the image of $x$ in $F_i$. Put $E_i = F_i \otimes_F E$, and we fix an $r$-dimensional $E$ vector space $V_0$ and $n_i$-dimensional $E_i$ vector spaces for each $1 \le i \le k$. Let $\cH^\flat$ be the isometric class of the family $(h_i)_{0 \le i \le n}$, where each $h_i$ is a non-degenerate Hermitian form on $V_i$. Take $h^\flat=(h_i) \in \cH^\flat$, put
\begin{equation} \label{eq:isometry_unitary_descent}
     (V^0,h^{\flat,0}) := \bigoplus_{i=1}^k \mathrm{Res}_{E_i/E} (V_i,h_i), 
\end{equation}
where $\mathrm{Res}_{E_i/E} (V_i,h_i)$ is the $E/F$ Hermitian space given by $(\mathrm{Res}_{E_i/E}V, \mathrm{Tr}_{E_i/E} \circ h_i)$, and $\oplus$ denotes the orthogonal direct sum. We also write $h^\flat_\oplus$ for the $E$ vector space $V_0 \oplus V^0$ equipped with the Hermitian form $h_0 \oplus h^{\flat,0}$.
 
 Note that $\cA_0$ and $\cA^0$ can be canonically identified with $\widetilde{\fu_0}/\U(V_0)$ and $\widetilde{\fu}^{V^0}/\U(V^0)$ respectively.

Put
\[
   \U_{h_0} := \U(V_0,h_0), \quad  \U_{h_i} := \mathrm{Res}_{F_i/F} \U(V_i,h_i), \quad  \U_{h^\flat}^0 := \prod_{i=1}^k \U_{h_i},  \quad \U_{h^\flat} := \prod_{i=0}^k \U_{h_i}.
\]
For each $1 \le i \le k$, we regard $\widetilde{\fu}^{h_i}$ as an $F$-vector space and define
\[
    \quad  \widetilde{\fu}^{h^\flat,0} := \prod_{i=1}^k \widetilde{\fu}^{h_i}, \quad \widetilde{\fu}^{h^\flat} := \prod_{i=0}^k \widetilde{\fu}^{h_i}.
\]
We then have GIT quotients:
\[
    \cA_i := \widetilde{\fu}^{h_i}/\U_{h_i},\quad \cA_{h^\flat}^0 := \widetilde{\fu}^{h^\flat,0}/\U_{h^\flat}^0 \cong \prod_{i=1}^k \cA_i, \quad \cA_{h^\flat} := \widetilde{\fu}^{h^\flat}/\U_{h^\flat} \cong \prod_{i=0}^k \cA_i
\]

Using ~\eqref{eq:isometry_unitary_descent}, we have embeddings:
\[
    \iota_{\U^{0}}: \U^{0}_{h^\flat} \to \U(V^0), \quad \iota_{\fu^{0}}: \widetilde{\fu}^{h^\flat,0} \to \widetilde{\fu}^{V^0}.
\]

 Let $V \in \cH$ such that there exists an isometry
 \[
    i: h^\flat_\oplus \cong V.
 \]
 The map $i$ then determines an embedding
\begin{equation} \label{eq:iota_unitary_Lie_algebra}
       \iota: \widetilde{\fu}^{h_0,(r)} \times \widetilde{\fu}^{V^0} \to \widetilde{\fu}^V 
\end{equation}
similar to what we have considered in ~\eqref{eq:descent_1}, see ~\cite{CZ21}*{(8.2.2.4)}. It descends to a map on the GIT quotient which coincides with $\iota$ in ~\eqref{eq:descent step 1 isomorphism}. 

Let $\widetilde{\fu}^{h^\flat,0,\prime}$ be the open subset of $\widetilde{\fu}^{h^\flat,0}$ consists of $(A_i,v_i)$ with $\det Q_i(A_j) \ne 0$ for all $i \ne j$, where $Q_i$ is the characteristic polynomial of $A_i$, it descends to an open subset $\cA^{0,\prime}_{h^\flat}$ of $\cA^{0}_{h^\flat}$. Denote by $\widetilde{\fu}^{h^\flat,\prime} = \widetilde{\fu}_0^{(r)} \times \widetilde{\fu}^{h^\flat,0,\prime}$, it descends to a subset $\cA^{\prime}_{h^\flat}$ of $\cA_{h^\flat}$. Composing $\iota$ in ~\eqref{eq:iota_unitary_Lie_algebra} and $\iota_{\fu^{0}}$, we get a map
\[
    \iota_{\fu^\flat}: \widetilde{\fu}^{h^\flat,\prime} \to \widetilde{\fu},
\]
it descends to
\[
    \iota_{\cA}: \cA^{\prime}_{h^\flat} \to \cA.
\]
Note that although the embedding $\iota_{\fu^\flat}$ depends on the choice of the isometry $i$, $\iota_\cA$ does not, and it coincides with the map $\iota_\cA$ introduced in Subsection ~\ref{subsec:descent}. The isometry $i$ also induces an embedding
\[
    \iota_{\U}: \U_{h^\flat} \hookrightarrow \U(V).
\]
By ~\cite{Zhang14}*{Appendix B}, $\iota_\cA$ is \'{e}tale, and we have a pullback diagram
\begin{equation} \label{diag:Luna slice_u}
           \begin{tikzcd}
\widetilde{\fu}^{h^\flat,\prime} \times^{\U_{h^\flat}} \U(V) \arrow[r] \arrow[d] & \widetilde{\fu}^V \arrow[d] \\
\cA^{\prime}_{h^\flat} \arrow[r]       & \cA                        
\end{tikzcd}
\end{equation}
where the top horizontal map sends $(X,g)$ to $\iota_{\fu^\flat}(X) \cdot g$.
    \begin{itemize}
        \item The right vertical map is the quotient map, the left vertical map is induced by the quotient map $\widetilde{\fu}^{\flat,\prime} \to \cA^{\flat,\prime}$, trivial on the second component,
        \item The bottom horizontal map is $\iota_\cA$, the top horizontal map sends $(X,g)$ to $\iota_{\fu^{\flat}}(X) \cdot g$.
    \end{itemize}

Since $a_0 \in \cA_0(F)$ is regular semisimple, the bijection ~\eqref{eq:matching_of_rss_lie_algebra} implies that there exists a unique Hermitian form $h_0=h_0^{a_0}$ on $V_0$ up to isometry such that there exists $X_0 \in \widetilde{\fu}_{\mathrm{rs}}^{h^{a_0}_0}(F)$ mapping to $a_0 \in \cA_0(F)$. Let $a_i$ be the image of $Z_{\alpha_i}$ in $\cA_i(F)$ and $a_{h^\flat} = (a_0,a_1,\cdots,a_k) \in \cA_{h^\flat}(F)$.

Let $\cO_a$ be the set 
    \begin{equation} \label{eq:O_a_definition}
        \{ (V,\fo) \mid V \in \cH, \fo \subset \widetilde{\fu}^V_a(F) \text{ is a semisimple orbit} \}.
    \end{equation}

By ~\cite{CZ21}*{Corollaire 8.4.6.2}, there is a bijection
\begin{equation} \label{eq:classiciation_semisimple_orbit}
\{ h^\flat = (h_i) \in \cH^\flat | h_0=h_0^{a_0} \}   \leftrightarrow  \cO_a,
\end{equation}
 sending $h^\flat$ to $h^\flat_\oplus$ together with the orbit of $\iota_{\fu^\flat}(X_0,Z_{\alpha_1},\cdots,Z_{\alpha_k})$, note that this orbit is independent of the choice of the isometry $i$. 

\subsubsection{Descent and transfer} \label{subsubsec:Descent_and_Transfer}

Fix $a \in \cA^{(r)}(F)$. In Subsection ~\ref{subsec:descent} and in ~\ref{subsubsec:descent_construction_u}, we have recalled the descent construction with respect to $a$ in the general linear setting and the unitary setting respectively. In the setting of Subsection ~\ref{subsec:descent}, we fix a basis of $V_i$ so that $\delta^+$ and $\delta^-$ are defined. By construction, $\cA_H$ and $\cA_{h^\flat}$ can be canonically identified for any $h^\flat \in \cH^\flat$. Under this identification $\cA_H'$ corresponds to $\cA^{\prime}_{h^\flat}$, $a_H$ corresponds to $a_{h^\flat}$.  

 Since $\iota_\cA$ is \'{e}tale and sending $a_H$ to $a$, we can choose an open neighborhood $\omega_H$ of $a_H$ in $\cA'_H(F)$ and an open neighbourhood $\omega$ of $a$ in $\cA(F)$ such that
    \begin{itemize}
        \item If $F$ is Archimedean, both $\omega$ and $\omega_H$ are semi-algebraic.
        \item $\omega_H \to \omega$ is an isomorphism of Nash manifolds if $F$ is Archimedean, and isomorphism of analytic manifolds if $F$ is non-Archimedean. (See ~\cite{BCR98}*{Proposition 8.1.2} for the Archimedean case)
        \item There are $c^+,c^- \in \{ +1,-1\}$, such that for any $a' \in \omega_H \cap \cA_{H,\mathrm{rs}}(F)$ and $X \in \widetilde{\fh}_a(F)$, we have
        \begin{equation} \label{eq:c^pm}
            \omega^+(\iota(X)) = c^+ \omega^+(X), \quad \omega^-(\iota(X))=c^- \omega^-(X).
        \end{equation}
        where $\omega^+(X)$ and $\omega^-(X)$ are defined in Remark ~\ref{rmk:generalized_transfer}. (See \cite{Zhang14}*{Lemma 3.15}). 
    \end{itemize}
Let $\Omega$ and $\Omega_H$ be the preimage of $\omega$ and $\omega_H$ under the quotient map $\widetilde{\gl_n} \to \cA$ and $\widetilde{\fh} \to \cA_H$ respectively. Then the top horizontal map in diagram ~\eqref{diag:Luna slice} induces an isomorphism
        \begin{equation} \label{eq:analytic_isomorphism_gl}
            \Omega_H \times^{H(F)} \GL_n(F) \to \Omega, \quad (Y,g) \mapsto \iota_{\fh}(Y) \cdot g
        \end{equation}
of Nash manifolds when $F$ is archimedean and analytic manifolds if $F$ is non-archimedean.

Let $h^\flat \in \cH^\flat$. Under the identification between $\cA_H'$ and $\cA^{\prime}_{h^\flat}$, $\omega_H$ corresponds to $\omega_{h^\flat} \subset \cA'_{h^\flat}(F)$. Let $\Omega_V$ and $\Omega_{h^\flat}$ be the preimage of $\omega$ and $\omega_{h^\flat}$ under the quotient map $\widetilde{\fu}^V \to \cA$ and $\widetilde{\fu}^{h^\flat} \to \cA_{h^\flat}$ respectively. By the general fact of group action of variety over local field, $\widetilde{\fu}^{h^\flat,\prime}(F) \times^{\U_{h^\flat}(F)} \U(V)(F)$ is naturally an open and closed subset of $(\widetilde{\fu}^{h^\flat,\prime} \times^{\U_{h^\flat}} \U(V) )(F)$. Let $\Omega_{V}^{h^\flat}$ be the image of $\Omega_{h^\flat} \times^{\U^\flat(F)} \U^V(F)$ under the top horizontal map of ~\eqref{diag:Luna slice_u}. Then $\Omega_{V}^{h^\flat}$ is independent of choice of $i$ (thus only depends on $V$ and $h^\flat$), and is an open and closed subset of $\Omega_V$ and we have an isomorphism
\begin{equation} \label{eq:analytic_isomorphism_u}
            \Omega_{h^\flat} \times^{\U^\flat(F)} \U^V(F) \to \Omega_{V}^{h^\flat}, \quad (Y,g) \mapsto \iota^{\flat}(Y) \cdot g
        \end{equation}
of Nash manifolds when $F$ is archimedean and analytic manifolds if $F$ is non-archimedean.

Let $\cH^\flat_V$ be the subset of $\cH^\flat$ consisting of $h^\flat=(h_i) \in \cH^\flat$ such that $h_0=h_0^{a_0}$ and $V \cong h^\flat_\oplus$.  By ~\cite{CZ21}*{Lemme 13.4.7.1}, shrinking $\omega$ if necessary, we have a disjoint union decomposition
\begin{equation} \label{eq:descent_disjoint_decomposition}
     \Omega_V = \bigsqcup_{h^\flat \in \cH^\flat_V} \Omega_V^{h^\flat}.
\end{equation}

For $\varphi \in \cS(\Omega)$, from $\eqref{eq:analytic_isomorphism_gl}$, we can choose $\varphi' \in \cS(\Omega_H \times \GL_n(F))$ such that for any $X \in \Omega_H$ and $g \in \GL_n(F)$
\[
    \varphi(\iota_\fh(X) \cdot g) = \int_{H(F)} \varphi'(Xh,\iota_H(h)^{-1}g) \rd h.  
\]
We define $\varphi_H \in \cS(\Omega_H)$ by
\[
    \varphi_H(X_H) = \int_{\GL_n(F)} \varphi'(X_H,g) \eta(g) \rd g,
\]
where $X_H \in \Omega_H$.

Suppose that for each $V \in \cH$, we have a function $\varphi^V \in \cS(\Omega_V)$. For each $h^\flat=(h_i) \in \cH^\flat$, we define $\varphi^{\flat} \in \cS(\Omega_{h^\flat})$ as follows. If $h_0 \ne h_0^{a_0}$, we put $\varphi^{h^\flat}=0$. Otherwise, choose an isometry $h^\flat_\oplus \cong V$ for a unique $V \in \cH$, and from ~\eqref{eq:analytic_isomorphism_u} and ~\eqref{eq:descent_disjoint_decomposition}, we can choose a function $\varphi^{h^\flat,\prime} \in \cS(\Omega_{h^\flat} \times \U(V)(F))$ such that for any $X \in \Omega_{h^\flat}$ and $g \in \U(V)(F)$
\[
    \varphi^V(\iota_{\fu^\flat}(X) \cdot g) = \int_{\U_{h^\flat}(F)} \varphi^{h^\flat,\prime}(Xh,\iota_{\U}(h)^{-1}g) \rd h.
\]
Then we put
\[
    \varphi^{h^\flat}(X) = \int_{\U(V)(F)} \varphi^{h^\flat,\prime}(X,g) \rd g,
\]
where $X \in \Omega_{h^\flat}$. We have $\varphi^{h^\flat} \in \cS(\Omega_{h^\flat})$.

We have the following lemma.
\begin{lemma} \label{lem:descent_and_transfer}
    Let $\varphi \in \cS(\Omega)$ and for each $V \in \cH$, let $\varphi^V \in \cS(\Omega_V)$. We have the following assertions:
    \begin{enumerate}
        \item If $\varphi \xleftrightarrow{+} (\varphi^V)_{V \in \cH}$, then
        \[
           c^+ \varphi_H \xleftrightarrow{+} (\varphi^{h^\flat})_{h^\flat \in \cH^\flat}.
        \]
        \item If $\varphi \xleftrightarrow{-} (\varphi^V)_{V \in \cH}$, then
        \[
        c^- \varphi_H \xleftrightarrow{-} (\varphi^{h^\flat})_{h^\flat \in \cH^\flat}.
        \]
    \end{enumerate} 
    Where we recall $c^\pm$ is defined in ~\eqref{eq:c^pm}.
\end{lemma}

\begin{proof}
    We only prove (1), the proof of (2) is the same as (1). Take any $a' \in \cA_{H,\mathrm{rs}}(F) = \cA_{h^\flat,\mathrm{rs}}(F)$, assume that $a'$ is the image of some elements in $\widetilde{\fu}^{h^\flat}(F)$. 
    
    Take any $X_H \in \widetilde{\fh}_{a'}(F), X_\flat \in \widetilde{\fu}^{h^\flat}_{a'}(F)$, we need to check that
    \[
       \omega^+(X_H)  I_{X_H}(\varphi_H) = J_{X_\flat}(\varphi^{h^\flat}).
    \]
    If $a' \not \in \omega_H=\omega_{h^\flat}$, then by definition
    \[
        I_{X_H}(\varphi_H) = J_{X_\flat}(\varphi^{h^\flat})=0.
    \]
    Thus we only need to consider the case when $a' \in \omega_H=\omega_{h^\flat}$. Since $\omega_H \subset \cA_H'$, by ~\cite{CZ21}*{Lemme 3.4.1.1}, we have $\iota_{\cA}(a') \in \cA_{\mathrm{rs}}(F)$. Using the definition of $\varphi_H$ and $\varphi^{h^\flat}$, we check directly that
    \begin{equation} \label{eq:descent_and_rss_orbital_integral}
        I_{X_H}(\varphi_H) = I_{\iota_{\fh}(X)}(\varphi), \quad J_{X^\flat}(\varphi^{h^\flat}) = J_{\iota_{\fu^{\flat}}(X_\flat)}(\varphi^V).
    \end{equation}
    The image of $\iota_\fh(X)$ and $\iota_{\fu^{\flat}}(X_\flat)$ in $\cA(F)$ are both $\iota_\cA(a')$, hence their orbits match. The lemma then follows from the definition of transfer and $c^\pm$.
\end{proof}

We also record a lemma whose proof is elementary.

\begin{lemma} \label{lem:localization_and_transfer}
     Let $u \in C_c^\infty(\omega)$, $\varphi \in \cS(\widetilde{\gl_n}(F))$ and for each $V \in \cH$, let $\varphi^V \in \cS(\widetilde{\fu}^V(F))$. For $\varepsilon \in \{+,-\}$, if $\varphi \xleftrightarrow{\varepsilon} (\varphi^V)$,
    then
    \[
        \varphi \cdot (q \circ u) \in \cS(\Omega) \xleftrightarrow{\varepsilon} (\varphi^V \cdot (q_V \circ u) \in \cS(\Omega_V)).
    \]
\end{lemma}

\begin{proof}
    This is proved by a direct computation.
\end{proof}

\subsection{Singular transfer on the Lie algebra} \label{subsec:local_II_lie_algebra}

Let $X \in \widetilde{\gl_n}(F)$ be a regular element and $\varphi \in \cS(\widetilde{\gl_n}(F))$. Recall that we have defined $I_X(\varphi,\xi,s)$ in Definition ~\ref{defi:I_X_in_general}.

The function $I_X(\varphi,1_{E^\times},s)/L_X(s,1_{E}^\times)$ is holomorphic at $s=0$. We write $I_X^\natural(\varphi)$ for $I_X^\natural(\varphi,1_{E^\times},0)$.

We now define some constants related to Proposition ~\ref{prop:local_singular_transfer_lie_algebra} below. Let $F'$ be a finite extension of $F$, and we write $E' := F' \otimes_F E$. Let $\eta':=\eta_{E/F} \circ \mathrm{Nm}_{F'/F} = \eta_{E'/F'}$ and $\psi'=\psi \circ \mathrm{Tr}_{F'/F}$. Let $(V',h')$ be an $n'$-dimensional $E'/F'$ Hermitian space. We put
\[
    c_{V'}^+ = \prod_{i=1}^{n'} \varepsilon(1-i,\eta^{\prime,i},\psi')^{-1} \eta'(\mathrm{disc}(V'))^{n'+1} \varepsilon(\frac 12,\eta',\psi')^{\frac{n'(n'+1)}{2}} \eta'(-1)^{\frac{n'(n'-1)}{2}}.
\]
and
\[
    c_{V'}^- = c_{V'}^+ \cdot \eta'(\mathrm{disc}(V')) =  \prod_{i=1}^{n'} \varepsilon(1-i,\eta^{\prime,i},\psi')^{-1} \eta'(\mathrm{disc}(V'))^{n'} \varepsilon(\frac 12,\eta',\psi')^{\frac{n'(n'+1)}{2}} \eta'(-1)^{\frac{n'(n'-1)}{2}}.
\]

Let $h^\flat=(h_i) \in \cH^\flat$ and let $\varepsilon:\{1,\cdots,k\} \to \{+,-\}$ be a map, we put $c^\varepsilon_{h^\flat} = \prod_{i=1}^k c^{\varepsilon_i}_{(V_i,h_i)}$. Let $\fo \subset \widetilde{\fu}_a^V(F)$ be a semisimple orbit, by the bijection ~\eqref{eq:classiciation_semisimple_orbit}, $\fo$ is the orbit of $\iota_{h^\flat}(X_0,Z_{\alpha_1},\cdots,Z_{\alpha_k})$ for some $h^\flat=(h_i) \in \cH^\flat_V$ with $X_0 \in \widetilde{\fu}^{V_0}_{a_0}(F)$. We then put 
\[
c_\fo^\varepsilon := c_{h^\flat}^{\varepsilon}.
\]

Let $X \in \widetilde{\gl_n}(F)$ be a regular element with $q(X)=a$. We use the descent construction as discussed in ~\ref{subsubsec:Descent_and_Transfer}. $X$ can be written of the form $\iota_\fh(X_H) \cdot g$, where $X_H = (X_{0,H},\cdots) \in \Omega_H$ is a regular element, and $g \in \GL_n(F)$. we define a constant $c_X \in \{\pm 1\}$ by 
\[
    c_X = c^+ \eta(g) \omega^+(X_{0,H}).
\] 
where $c^+$ is defined in ~\eqref{eq:c^pm}. Finally, if $X \in \widetilde{\gl_n}(F)$ is a regular element of type $\varepsilon$, we put
\[
    c_{X,\fo} = c_X c_\fo^\varepsilon.
\]

 Now we state our main proposition
\begin{proposition} \label{prop:local_singular_transfer_lie_algebra}
    Let $X \in \widetilde{\gl_n}(F)$ be a regular element with $q(X)=a$. Suppose $\varphi \in \cS(\widetilde{\gl_n}(F)) \xleftrightarrow{+} (\varphi^V \in \cS(\widetilde{\fu}^V(F)))_{V \in \cH}$, then
    \[
        I^\natural_X(\varphi) =  \sum_{(V,\fo) \in \cO_a} c_{X,\fo} J^V_{\fo}(\varphi^V),
    \]
    where we recall that the set $\cO_a$ is defined in ~\eqref{eq:O_a_definition}.
\end{proposition}

\begin{proof}
    By definition, for any semisimple orbit $\fo$. $J^V_\fo(\varphi^V)$ only depends on the value of $\varphi^V$ on $\fo$. By Proposition ~\ref{prop:local_general} (2), $I_X(\varphi)$ depends only on the value of $\varphi$ at the orbit of $X$. Let $\omega$ be the open subset of $\cA(F)$ as introduced in ~\ref{subsubsec:Descent_and_Transfer}. Take $u \in C_c^\infty(\omega)$ with $u(a)=1$. Thus $I_X(\varphi \cdot (q \circ u)) = I_X(\varphi)$, and for any $\fo \in \cO_a$, $J_\fo(\varphi^V) = J_\fo(\varphi^{V} \cdot (q_V \circ u))$. By Lemma ~\ref{lem:localization_and_transfer}, after replacing $\varphi$ and $(\varphi^V)$ by $\varphi \cdot (q \circ u)$ and $\varphi^V \cdot (q_V \circ u)$, we can assume $\varphi \in \cS(\Omega)$ and $\varphi^V \in \cS(\Omega_V)$.

    Recall that we have defined $\varphi_H \in \cS(\Omega_H)$ and $\varphi^{h^\flat} \in \cS(\Omega_{h^\flat})$ for each $h^\flat \in \cH^\flat$ in \ref{subsubsec:Descent_and_Transfer}. Choose $X_H = (X_{0,H},Z_{\alpha_1}^{\varepsilon_1},\cdots,Z_{\alpha_k}^{\varepsilon_k}) \in \Omega_H$ and $g \in \GL_n(F)$ such that $X = \iota_\fh(X_H) \cdot g$. From the definition of $I_X(\varphi)$ in ~\eqref{eq:definition_I_X_on_Omega}, we have $I_X(\varphi) = \eta(g) I_{X_H}(\varphi_H)$. For $\fo \in \cO_a$, thanks to the bijection ~\eqref{eq:classiciation_semisimple_orbit}, we can write $\fo$ as the orbit of $\iota_{\fu^{h^\flat}}(X_0,Z_{\alpha_1},\cdots,Z_{\alpha_k})$, where $h^\flat=(h_i) \in \cH^\flat$ with $h_0=h_0^{a_0}$ and $X_0 \in \widetilde{\fu}^{h_0}_{a_0}(F)$. Write $X_{h^\flat} = (X_{0},Z_{\alpha_1},\cdots,Z_{\alpha_k})$, then we see that
    \begin{align*}
        J_X(\varphi^V) &= \int_{\U(V)_X(F) \backslash \U(V)(F)} \int_{\U_{h^\flat}(F)} \varphi^{h^\flat,\prime}(X_{h^\flat} \cdot h, \iota_{\U}(h)^{-1} g) \rd h \rd g  \\
        &= \int_{\U(V)_X(F) \backslash \U(V)(F)} \int_{\U_{h^\flat}^0(F) \backslash \U_{h^\flat}(F)}\int_{\U_{h^\flat}^0(F)} \varphi^{h^\flat,\prime}(X_{h^\flat} \cdot h , \iota_{\U}(h'h)^{-1} g) \rd h' \rd h \rd g  \\
        &= \int_{\U_{h_0}(F)} \int_{\U(V)(F)} \varphi^{h^\flat}(X_{h^\flat} \cdot h) \rd h =  J_{X^\flat}(\varphi^{h^\flat}).
    \end{align*}

    Recall if $h_0 \ne h_0^{a_0}$, then $\varphi^{h^\flat}=0$. By Lemma ~\ref{lem:descent_and_transfer}, $c^+ \varphi_H \xleftrightarrow{+} \varphi^{h^\flat}$.

    Thus we are left show that for any $\varphi_1 \in \widetilde{\fh}(F)$ matches with $(\varphi^{h^\flat}_1 \in \cS(\widetilde{\fu}^{h^\flat}(F)))_{h^\flat \in \cH^\flat}$, we have
    \[
       \omega^+(X_{0,H}) I^\natural_{X_H}(\varphi_1) = \sum_{\substack{h^\flat=(h_i) \in \cH^\flat \\ h_0 = h_0^{a_0}}} c_{h^\flat}^{\varepsilon_i} J_{X^\flat}(\varphi_1^{h^\flat}).
    \]
    We henceforth reduce to the case when both $\varphi$ and $\varphi^{h^\flat}$ are pure tensors. The 0-th component is regular semisimple, hence the equality follows from the definition of transfer. For other components, we are reduced to the Lemma ~\ref{lem:nilpotent_singular_transfer} below, which itself is a special case of this proposition. This finishes the proof.
\end{proof}

\begin{lemma} \label{lem:nilpotent_singular_transfer}
    Let $\varphi \in \cS(\widetilde{\gl_n}(F)) \xleftrightarrow{+} (\varphi^V \in \cS(\widetilde{\fu}^V(F)))_{V \in \cH}$ and let $\lambda \in F$, we have
    \[
        I^{\natural}_{Z_\lambda^+}(\varphi) =  \sum_{V \in \cH}  c_V^+  J_{Z_\lambda} (\varphi^V) \text{ and } I^{\natural}_{Z_\lambda^-}(\varphi) = \sum_{V \in \cH} c_V^-  J_{Z_\lambda} (\varphi^V).
    \]
\end{lemma}

\begin{proof}
    We only prove the assertion for $I^\natural_{Z_\lambda^+}$. The case for $I^\natural_{Z_\lambda^-}$ is parallel.
     It is clear that $\varphi_\lambda \xleftrightarrow{+} (\varphi^V_\lambda)$ By Theorem ~\ref{thm:transfer_Fourier_transform} and Lemma ~\ref{lem:relation_between_plus_minus_transfer} we have
    \[
    \cF \varphi_\lambda \xleftrightarrow{-} \left(\eta(\mathrm{disc}(V))^{n+1}\varepsilon(\eta,\frac 12,\psi)^{\frac{n(n+1)}{2}} \cF \varphi_{\lambda}^V \right).
    \]
    If we put
    \[
        \gamma^+(s) = \prod_{i=1}^n \gamma(-is-i+1,\eta^i,\psi).
    \]
    Then
    \[
        \gamma^+(s) L_{Z_\lambda^+}(s) = \prod_{i=1}^n L(is+i,\eta^i) \varepsilon(-is-i+1,\eta^i,\psi)
    \]
    is non-vanishing at $s=0$ with value $\prod_{i=1}^n L(i,\eta^i) \varepsilon(1-i,\eta^i,\psi)$. Also note that $\omega^{-}(Z_0^-) = \eta(-1)^{\frac{n(n-1)}{2}}$.
       
    Combining the discussion above and  ~\eqref{eq:I_X^+_Fourier}, we see that
     \begin{align*}
         I^\natural_{Z_\lambda^+}(\varphi) &= \zeta_n  \eta(-1)^{\frac{n(n-1)}{2}} \frac{1}{\prod_{i=1}^n L(i,\eta^i) \varepsilon(1-i,\eta^i,\psi)} \int_{\widetilde{\gl_n}(F)} \cF \varphi_\lambda (X) \omega^-(X) \rd X \\
         &= \zeta_n  \eta(-1)^{\frac{n(n-1)}{2}} \frac{1}{\prod_{i=1}^n L(i,\eta^i) \varepsilon(1-i,\eta^i,\psi)} \int_{\cA_{\mathrm{rs}}(F)} \int_{\GL_n(F)} \cF \varphi_\lambda(X \cdot g) \omega^-(X \cdot g) \rd g \rd X.
     \end{align*}
By ~\cite{BP21b}*{Lemma 5.2.1}, the bijection ~\eqref{eq:matching_of_rss_lie_algebra} is measure preserving up to constant. More precisely, suppose a measurable subset $M$ corresponds to $\sqcup M_V$, then 
\[
    \vol(M) = \frac{\prod_{i=1}^n L(i,\eta^i)}{\zeta_n} \vol(M_V).
\]
Therefore, the expression for $I_{Z_\lambda^+}^\natural(\varphi)$ can be written as
\begin{align*}
         &  \frac{ \eta(-1)^{\frac{n(n-1)}{2}}}{\prod_{i=1}^n \varepsilon(1-i,\eta^i,\psi)}  \sum_{V \in \cH} \int_{\widetilde{\fu}^V_{\mathrm{rs}}(F)/\U(V)(F)} \eta(\mathrm{disc}(V))^{n+1}\varepsilon(\eta,\frac 12,\psi)^{\frac{n(n+1)}{2}} \int_{\U(V)(F)} \cF \varphi^V_\lambda(X \cdot g) \rd g \rd X \\
         =&  \sum_{V \in \cH} c_V^+ \int_{\widetilde{\fu}^V(F)} \cF \varphi_\lambda^V(X) \rd X = \sum_{V \in \cH} c_V^+ J^V_{Z_\lambda}(\varphi^V).
\end{align*}
Thus the lemma is proved.
\end{proof}

When $X$ lies in $\widetilde{\gl_n}_{,+}(F)$, then the complex number $\omega^+(X) I_X^\natural(\varphi)$ only depends on the image $a$ of $X$ in $\cA(F)$. In this case, the Proposition ~\ref{prop:local_singular_transfer_lie_algebra} can be stated as follows: if $\varphi \xleftrightarrow{+} (\varphi^V)_{V \in \cH}$, then 
\begin{equation} \label{eq:singular_transfer_with_transfer_factor}
    \omega^+(X) I_X^\natural(\varphi) = \sum_{V,\fo} c'_{a,\fo} J_\fo^V(\varphi^V),
\end{equation}
where the constant $c'_{a,\fo} := c^+ c_{\fo}^{+,\prime}$ and $c_{\fo}^{+,\prime}$ is defined similarly as $c_\fo^+$, but we omit the constant $\eta'(-1)^{\frac{n(n-1)}{2}}$ in the definition of $c_{V'}^+$.

We also record a lemma related to the proof above.

\begin{lemma} \label{lem:semisimple_orbital_integral_stable}
    Let $V \in \cH$ and let $\fo \subset \widetilde{\fu}$ be a semisimple orbit. Then the distribution $J_\fo$ is a stable distribution, in the sense that if $\varphi \in \cS(\widetilde{\fu}(F))$ such that all the regular semisimple orbital integral of $\varphi$ is 0, then $J_\fo(\varphi)=0$.
\end{lemma}

\begin{proof}
    Let $\varphi \in \cS(\widetilde{\fu}^V(F))$ such that all regular semisimple orbital integrals of $\varphi$ is 0. Take $X \in \fo$, assume that $q_V(X)=a \in \cA(F)$. Replacing $\varphi$ by $\varphi^V \cdot (q_V \circ u)$, we can assume $\varphi \in \cS(\Omega_V)$. Taking $h^\flat \in \cH^\flat$ such that there is a semisimple orbit $\fo^\flat$ corresponds to $\fo$ under the bijection ~\eqref{eq:classiciation_semisimple_orbit}. Then as the computation in the proof of Proposition ~\ref{prop:local_general}, $J_\fo(\varphi) = J_{\fo^\flat}(\varphi^\flat)$ and $\varphi^\flat$ also has vanishing regular semisimple orbital integrals, we therefore reduced to the case when $\fo$ is a regular semisimple orbit or a central orbit. If $\fo$ is regular semisimple, $J_\fo(\varphi)$ vanishes by definition. If $\fo$ is the orbit of $Z_\lambda$, after a translation, we can assume $\lambda=0$. Then
    \[
        J_\fo(\varphi) = \varphi(0) = \int_{\widetilde{\fu}^V(F)} \cF \varphi(X) \rd X = \int_{\widetilde{\fu}^V_{\mathrm{rs}}(F)/\U(V)(F)} \int_{\U(V)} \varphi(X \cdot g) \rd g \rd X.
    \]
    By ~\cite{Chaudouard19}*{3.2.4}, $\cF \varphi$ also has vanishing regular semisimple orbital integrals, therefore the above expression vanishes.
\end{proof}

We introduce a variant of the results in this subsection. Recall from Subsection ~\ref{subsec:symmetric space S} that, $\cB = \gl_{n+1}/\GL_n$. For $(V,h) \in \cH$, we put $(V',h')=(V \oplus Ee_0,h \oplus h_0)$, where $\oplus$ denotes the orthogonal direct sum and $h_0(e_0,e_0)=1$. We have an isomorphism of $\U(V)$ representations:
\begin{equation} \label{eq:iso_uV'_and_tildeu}
        \fu^{V'} \cong \widetilde{\fu}^V \times F, \quad A \mapsto ((A|_V,Ae_0),h'(Ae_0,e_0)),
\end{equation}
where $\U(V)$ acts on $F$ trivially. We thus have a canonical identification $\fu^{V'}/\U(V) \cong \cB$. Semisimple orbital integral $J_\fo(\varphi^V)$ directly generalize to $\varphi^V \in \cS(\fu^{V'}(F))$. For $X \in \gl_{n+1}(F)$, let $\omega^\pm(X)=\eta(\delta^\pm(X))$. Then there is an obvious notion of plus/minus transfer between $\varphi \in \cS(\gl_{n+1}(F))$ and $(\varphi^V) \in \cS(\fu^{V'}(F))_{V \in \cH}$. For $b \in \cB(F)$, let $\cO_b$ be the set
\[
    \{ (V,\fo) \mid V \in \cH, \fo \subset \fu^{V'}_a(F) \text{ is a semisimple orbit} \}.
\]
For $\fo \in \cO_b$, it projects to a semisimple orbit $\widetilde{\fo} \in \cO_a$ under the isomorphism ~\eqref{eq:iso_uV'_and_tildeu}, where $a$ is the image of $b$ in $\cA$. 

For $X \in \gl_{n+1}(F)$, let $Y$ is the image of $X$ in $\widetilde{\gl_n}(F)$ under the bijection $\gl_{n+1} \cong \widetilde{\gl_n} \times F$. Let $c_{X,\fo} := c_{Y,\widetilde{\fo}}$. Then Proposition ~\ref{prop:local_singular_transfer_lie_algebra} implies that, if $X \in \gl_{n+1}(F)$ is a regular element with $q(X)=a$ of type $\varepsilon$, and $\varphi \in \cS(\gl_{n+1}(F)) \xleftrightarrow{+} (\varphi^V \in \cS(\fu^{V'}(F)))$, then
\begin{equation} \label{eq:local_II_variant}
     I^\natural_X(\varphi) = \sum_{(V,\fo) \in \cO_a} c_{X,\fo} J_\fo^V(\varphi^V).
\end{equation}

\subsection{Singular transfer on the group}
\label{subsec:local_II_group}

We now deduce the singular transfer on the group from the Lie algebra as we discussed in Subsection ~\ref{subsec:local_II_lie_algebra}. Fix $b \in \rB(F)$, pick $\sigma \in E^\times$, with $\sigma \sigma^\mathtt{c}=1$, so that $b \in \rB^\sigma(F)$. For $f \in \cS(\rG'(F))$, let $I_\gamma^{\sigma,\natural}(f) := I_\gamma^{\sigma,\natural}(f,1_{E^\times},0)$. 

We use the Cayley map to relate the group $\U(V')$ the the Lie algebra $\fu^{V'}$. Let $\tau,\sigma \in E$ such that $\tau^\mathtt{c}=-\tau$ and $\sigma \sigma^\mathtt{c}=1$. Let $\fu^{V',\tau}$ be the open subscheme of $\fu^{V'}$ consists of $Y \in \fu^{V'}$ such that $Y - \tau \cdot 1$ is invertible. Let $\U(V')^{\sigma}$ be the open subscheme of $\U(V')$ consists of $g \in \U(V')$ such that $g - \sigma \cdot \mathrm{id}$ is invertible and $\rG^{V,\sigma}$ be the open subscheme of $\rG^V$ consists of $(x,y) \in \rG^V$ such that $x^{-1}y \in \U(V')^\sigma$. The Cayley map
\[
    \fc_\sigma^V: \fu^{V',\tau} \to \U(V')^\sigma, \quad Y \mapsto \sigma \frac{1+\tau^{-1}Y}{1-\tau^{-1}Y}
\]
defines a $\U(V)$ equivariant isomorphism between $\fu^{V',\tau}$ and $\U(V')^{\sigma}$, and descends to a map $\cB^{\tau} \to \rB^\xi$ which coincides with the map induced by $\fc_\sigma$ as in ~\eqref{eq:Cayley_gl_n}.

We define
\[
    \cO_b = \{ (V,\cO) \mid V \in \cH, \cO \subset \rG^V_b(F) \text{ is a } \rH^V(F) \times \rH^V(F) \text{ semisimple orbit}  \}.
\]

Let $a = \fc_\sigma^{-1}(b)$. Note that there is a natural $H^V=\U(V)$ equivariant isomorphism $\rG^V/\rH^V \cong \U(V')$. Since $\fc_\sigma^V$ is an isomorphism, the Cayley map induces a bijection $\cO_a \to \cO_b$, which we also denote by $\fc_\sigma$. For $\cO \in \cO_b$ and $\gamma \in \rG^{\prime,\sigma}(F)$, let $x=\nu(\gamma_n^{-1} \gamma_{n+1})$ and define a constant $c_{\gamma,\cO}$ by
\begin{equation} \label{eq:c_gamma_o}
    c_{\gamma,\cO} = c_{\fc_\sigma^{-1}(x),\fc^{-1}_\sigma(\cO)} \mu(x)^{\frac{n+1}{2}} \mu(\gamma_n^{-1} \gamma_{n+1})^{-1} \mu(-2 \sigma \tau^{-1})^{-\frac{n(n+1)}{2}} \mu(1-\tau^{-1} \fc_\sigma(x))^{-n}.
\end{equation}

Now we state the main theorem of this section
\begin{theorem} \label{thm:local_singular_transfer}
    Let $f \in \cS(\rG'(F))$ and for each $V \in \cH$ let $f^V \in \cS(\rG^V(F))$ such that $f \xleftrightarrow{+} (f^V)$. Then for any regular element $\gamma \in \rG'_b(F)$, we have
    \[
        I^{\sigma,\natural}_\gamma(f) =\sum_{(V,\cO) \in \cO_b} c_{\gamma,\cO} J_\cO^V(f^V).
    \]
\end{theorem}

\begin{remark}
    When $\gamma \in \rG'_+(F)$, recall from Lemma ~\ref{lem:local_group_central_rss_description} that $I_\gamma^\sigma$ is independent of the choice of $\sigma$, we denote it by $I_\gamma$ here.

    Using the equation ~\eqref{eq:singular_transfer_with_transfer_factor}, we easily see that the above theorem can be written as
    \begin{equation} \label{eq:plus_regular_reinterpretation}
          \Omega^+(\gamma)I_\gamma^\natural(f) = \sum_{(V,\cO) \in \cO_b} c_{X,\cO} \omega^+(X) J^V_\cO(f^V),
    \end{equation}    
    where $\fc_\sigma(X) = \nu(\gamma_n^{-1} \gamma_{n+1})$, and $c_{X,\cO} := c_{X,\fc_\sigma^{-1}(\cO)}$.

    Note that the left-hand side of the equation ~\eqref{eq:plus_regular_reinterpretation} only depends on $b$. This description removes some factors appearing in ~\eqref{eq:c_gamma_o}.
\end{remark}

\begin{proof}
    Let us denote by $q$ the quotient map $\rG' \to \rB$, for $V \in \cH$, we denote by $q^V$ the quotient $\rG^V \to \rB$. Since both $I_\gamma(f)$ (resp. $J_\cO$) only depends on the value of $f$ on $\rG'_b(F)$ (resp. $\rG^V_b(F)$). Take $u \in C_c^\infty(\rB^\sigma(F))$ such that $u(b)=1$. After replacing $f$ (resp. $f^V$) by $f \cdot (u \circ Q)$ (resp. $f^V \cdot (u \circ Q_V)$),
    we can $f \in \cS(\rG^{\prime,\sigma}(F))$ and $f^V \in \cS(\rG^{V,\sigma}(F))$.

    Let $f^\rS := f^{\rS}_0$ defined in ~\eqref{eq:defi_f^S}. Then $f^\rS \in \cS(\rS^\sigma(F))$. We define $f^{\U(V')} \in \cS(\U(V')^\sigma(F))$ by
    \[
        f^{\U(V')}(x) = \int_{\rH^V(F)} f(hx) \rd h,
    \]
    where $x \in \U(V')(F)$. 
    
    Define $\varphi_f \in \cS(\gl_{n+1}^\tau(F))$ by
    \[
        \varphi_f(X) = f^\rS(\fc_\sigma(X)) \mu(\fc_\sigma(X))^{-\frac{n+1}{2}} \mu(-2 \sigma \tau^{-1})^{\frac{n(n+1)}{2}} \mu(1-\tau^{-1}X)^{-n},
    \]
    where $X \in \gl_{n+1}^\tau(F)$. For $V \in \cH$ and $X \in \fu^{V',\tau}(F)$ we put $\varphi_{f^V}(X) = f^{\U(V')}(\fc_\sigma(X))$, then $\varphi_f \in \cS(\fu^{V',\tau}(F))$. Direct computation shows if $f \xleftrightarrow{+} f^V$ then $\varphi_f \xleftrightarrow{+ }\varphi_{f^V}$.

    Let $x=\nu(\gamma_n^{-1} \gamma_{n+1})$. By the definition of $I_\gamma$ and $J_\cO$, we have
    \[
        I^\natural_{\fc_\sigma^{-1}(x)}(\varphi_f) =\mu(\gamma_n^{-1}\gamma_{n+1})  \mu(x)^{-\frac{n+1}{2}}  \mu(-2 \sigma \tau^{-1})^{\frac{n(n+1)}{2}} \mu(1-\tau^{-1} \fc_\sigma^{-1}(x)) I^{\sigma,\natural}_\gamma(f),
    \]
    and
    \[
        J_\cO^V(f) = J_{\fc(\cO)}^V(\varphi_f).
    \]
    Then the theorem follows from the equality ~\eqref{eq:local_II_variant}.
\end{proof}

\section{Fourier-Jacobi case} \label{sec:Fourier-Jacobi}

All the results in this article have their counterparts for the relative trace formula developed by Liu ~\cite{Liu14}. The proofs of the results in this section are parallel to their Jacquet-Rallis counterpart, so we will omit their proofs. 

\subsection{Notations}

Let $F$ be a field of characteristic 0 and let $E$ be a quadratic \`{e}tale algebra over $F$. We redefine $\rG' := \mathrm{Res}_{E/F}(\GL_{n,E} \times \GL_{n,E}) \times F^n \times F_n$. We put
\[
    \rS_n = \{ x \in \GL_{n,E} \mid x x^\mathtt{c}=1 \},
\]
the map $\nu: \mathrm{Res}_{E/F} \GL_{n,E} \to \rS_n, g \mapsto g g^{-1,\mathtt{c}}$ identifies $\rS_n$ with $\GL_{n,E}/\GL_{n,F}$. We redefine $\rH_1$ and $\rH_2$ to be the following subgroups of $\mathrm{Res}_{E/F}(\GL_{n,E} \times \GL_{n,E})$:
\[
    \rH_1 := \{ (g,g) \mid g \in \mathrm{Res}_{E/F} \GL_{n,E} \}, \quad \rH_2 = \GL_{n,F} \times \GL_{n,F}. 
\]
The group $\rH_1 \times \rH_2$ acts on the right on $\rG'$ by:
\[
    (g,v,u) \cdot (h_1,(h_{2,1},h_{2,2})) = (h_1^{-1}gh_2,h_{2,2}^{-1}v,uh_{2,2}). 
\]
We redefine $\rB$ to be the GIT quotient $\rG'/\rH_1 \times \rH_2$, it can be identified with the quotient $(\rS_n \times F^n \times F_n)/\GL_n$, where $\GL_n$ acts by
\[
    (s,v,u) \cdot h = (h^{-1}sh,h^{-1}v,uh),
\]
and the identification is made through the map:
\[
    \alpha: \mathrm{Res}_{E/F}(\GL_{n,E} \times \GL_{n,E}) \to \rS_n, \quad (g_1,g_2) \mapsto \nu(g_1^{-1}g_2).
\]

Let $\sigma \in E, \sigma \sigma^\mathtt{c}=1$. $\rS_n$ has an open subset $\rS_n^\sigma$ consists of $x$ such that $x-\sigma \cdot \mathrm{id}$ is invertible, it descends to an open subset $\rB^\sigma$ of $\rB$. Let $\tau \in E$ with $\tau + \tau^{\mathtt{c}}=0$, let $\widetilde{\gl_n}^\tau$ be the open subset of $\widetilde{\gl_n}$ consisting of $(A,v,u)$ such that $A - \tau$ is invertible. We define the Cayley transform:
\[ 
    \fc_\sigma:\widetilde{\gl_n}^\tau \to \rS_n^\sigma \times F^n \times F_n, \quad (A,v,u) \mapsto \left( -\sigma \frac{1+\tau^{-1}A}{1-\tau^{-1}A},v,u \right)
\]
Let $(\rS_n \times F^n \times F_n)_\mathrm{reg}$ be the open subset of the regular elements under this action of $\GL_n$, it has open subsets $(\rS_n \times F^n \times F_n)_+$ (resp. $(\rS_n \times F^n \times F_n)_-$) consists of $(x,v,u)$ such that $(v,xv,\cdots,x^{n-1}v)$ forms a basis of $E^n$ (resp. $(u,ux,\cdots,ux^{n-1})$ form a basis of $E_n$), its preimage in $\rG'$ will be denoted by $\rG'_+$ (resp. $\rG'_-$). For $X \in \widetilde{\gl_n}^\tau$, and $\bullet \in \{+,-\}$
\[
    X \in \widetilde{\gl_n}_{,\bullet} \iff \fc_\sigma(X) \in (\rS_n^\sigma \times F_n \times F^n)_\bullet
\]

Let $\cH$ be the isometric class of $n$-dimensional non-degenerate $E/F$-Hermitian space. For $V \in \cH$, we redefine
\[
    \rG^V := \U(V) \times \U(V) \times V,
\]
let $\rH^V$ be the diagonal subgroup of $\U(V) \times \U(V)$. The group $\rH^V \times \rH^V$ acts on the right on $\rG^V$ by
\[
    (g,v) \cdot (h_1,h_2) = (h_1^{-1}gh_2,h_2^{-1}v), \quad g \in \U(V) \times \U(V), v \in V.
\]
The GIT quotient $\rG^V/\rH^V \times \rH^V$ can be identified with $\rB$. Let
\[
    \cO_b = \{ (V,\cO) \mid V \in \cH, \cO \subset \rG_b^V \text{ is a }\rH^V(F) \times \rH^V(F) \text{ semisimple orbit} \}
\]

\subsection{Local theory}

Now let $F$ be a local field. For $f \in \cS(\rG'(F)),(x,v,u) \in \rS_n(F) \times F^n \times F_n$ and $s \in \C$, we put
\begin{equation} \label{eq:defi_f^S_FJ}
    f^\rS_s(x,v,u) := \begin{cases}  \int_{\rH_1(F)} \int_{\GL_n(F)} f(h_1^{-1},h_1^{-1}\nu^{-1}(x)(1,h_2),v,u) \xi(h_1) \lvert \det h_1 \rvert^s \mu(\nu^{-1}(x)h_{2,3})  ,& n \text{ odd}, \\
    \int_{\rH_1(F)} \int_{\GL_n(F)} f(h_1^{-1},h_1^{-1}\nu^{-1}(x)(1,h_2),v,u) \xi(h_1) \lvert \det h_1 \rvert^s  ,& n \text{ even}.  \end{cases}
\end{equation}
Fix $\gamma \in \rG'(F)$, let $b$ be the image of $\gamma$ in $\rB(F)$. Choose $\sigma \in E, \sigma \sigma^\mathtt{c}=1$ such that $b \in \rB^\sigma(F)$, let $u \in C_c^\infty(\rB^\sigma(F))$, for $s \in \C$, define $\widetilde{f}_s \in \cS(\widetilde{\gl_n}(F))$ by 
\[
    \widetilde{f}_s := (u \cdot f_s^\rS)(\fc_\sigma(X)).
\]
We define
\[
    L_\gamma(s,\xi) := L_{\fc_\sigma^{-1}(b)}(s,\xi)
\]
and
\[
    I^\sigma_\gamma(f,\xi,s) := \begin{cases}
        I_{\fc_\sigma(x)}(\widetilde{f}_s,\xi,s) \mu(\gamma_n \gamma_{n+1}^{-1}) & n\text{ odd}  \\ I_{\fc_\sigma(x)}(\widetilde{f}_s,\xi,s) & n\text{ even}
    \end{cases}.
\]

\begin{proposition} \label{prop:local_group_FJ}
    The definition of $I^\sigma(f,\xi,s)$ is independent of the choice of $u$, and (1)-(5) of Proposition ~\ref{prop:local_group} holds (where $\rG',\rH_1,\rH_2$ are redefined as in this Subsection). 
\end{proposition}
and we also have the following lemma
\begin{lemma} \label{lem:local_group_central_rss_description_FJ}
    Let $\eta$ be a character on $\rH_2(F)$ defined by $\eta(h_{2,1},h_{2,2}) = \eta(h_{2,1})^{n+1} \eta(h_{2,2})^n$. Let $f \in \rS(\rG'(F))$, we have the following assertions:
    \begin{itemize}
        \item If $\gamma$ is regular semisimple, then
        \[
            I^\sigma_\gamma(f,\xi,s) = \int_{\rH_1(F)} \int_{\rH_2(F)} f( \gamma \cdot (h_1,h_2)) \xi(h_1) \lvert \det h_1 \rvert^s \eta(h_2) \rd h_1 \rd h_2,
        \]
        where the integral converges absolutely for any $s \in \C$.
        \item If $\gamma \in \rG'_+(F)$ (resp. $\rG'_-(F)$), then the integral
        \[
             \int_{\rH_1(F)} \int_{\rH_2(F)} f( \gamma \cdot (h_1,h_2)) \xi(h_1) \lvert \det h_1 \rvert^s \eta(h_2) \rd h_1 \rd h_2,
        \]
        is absolutely convergent for $s \in \cH_{<-1+\frac 1n}$ (resp. $s \in \cH_{>1-\frac 1n}$) and equals $I^\sigma_\gamma(f,\xi,s)$ there.  
    \end{itemize}
    In particular, in these cases, the definition of $I^\sigma_\gamma(f,\xi,s)$ does not depend on the choice of $\sigma$.
\end{lemma}

There is a notion of transfer between $f \in \rS(\rG'(F))$ and $(f^V) \in \cS(\rG^V(F))$ similar as the situation in Subsection ~\ref{subsec:transfer}. For $(\gamma,v,u) \in \rG'(F)$, choose $\sigma$ such that $\alpha(\gamma) \in \rS_n^\sigma(F)$, let $X=\fc_\sigma^{-1}(\alpha(\gamma),v,u)$. Let $b \in \rB(F)$, for $\cO \in \cO_b$ We define
\[
    c_{\gamma,\cO} := c_{X,\fc_\sigma^{-1}(\cO)} \mu(x)^{\frac{n+1}{2}} \mu(\gamma_1^{-1}\gamma_2)^{-1} \mu(2 \sigma \tau^{-1})^{\frac{n(n-1)}{2}} \mu(1-\tau^{-1} \fc_\sigma(x))^{-n+1},
\]
where $\fc_\sigma^{-1}(\cO)$ denotes the semisimple orbit in $\widetilde{\fu}^V$ corresponding to $\cO$ and $\gamma=(\gamma_1,\gamma_2)$.

\begin{proposition}
    Let $f \in \rS(\rG'(F))$ and for each $V \in \cH$ let $f^V \in \cS(\rG^V(F))$ such that $f \leftrightarrow (f^V)$. Then for any regular element $\gamma$ in $\rG'_b(F)$, we have
    \[
        I^{\sigma,\natural}_\gamma(f) = \sum_{(V,\cO) \in \cO_b} c_{\gamma,\cO} J^V_\cO(f^V).
    \]
\end{proposition}

\subsection{Global theory}

Let $E/F$ be a quadratic extension of number fields. For $f \in \rS(\rG'(\bA))$ and $T \in \fa_0$, one defines a modified kernel $K_f^T(h_1,h_2)$ as in ~\cite{BLX}*{Section 7}.

Using the same strategy as the proof of the Theorem ~\ref{thm:coarse Bessel}, one can prove the following proposition. The corresponding results on the asymptotics of the modified kernel can be found in ~\cite{BLX}*{Section 7}.
\begin{proposition} \label{prop:coarse FJ}
    Let $\xi$ be a strictly unitary character of $\bA_E^\times$.
    \begin{enumerate}
        \item For any $f \in \cS(\rG'(\bA))$, $s \in \C$ and $T$ sufficiently positive, we have
            \begin{equation} \label{eq:geometric convergence_FJ}
                  \sum_{\gamma \in \rB(F)} \int_{[\rH_1]} \int_{[\rH_2]}  \lvert K_{f,\gamma}^T(h_1,h_2)  \rvert \lvert \det h_1 \rvert^{\mathrm{Re}(s)}  \rd h_1 \rd h_2 < \infty.
            \end{equation}
    \item For $\gamma \in \rB(F)$, as a function of $T$, the integral
            \[
                I_\bullet^T(f,\xi,s) := \int_{[\rH_1]} \int_{[\rH_2]} K^T_{f,\gamma}(h_1,h_2) \xi(h_1) \lvert \det h_1 \rvert^s \eta(h_2) \rd h_1 \rd h_2
            \]
          is an exponential polynomial. If $s \not \in \{-1,1\}$, then the pure polynomial term is constant, denoted by $I_\gamma(f,\xi,s)$. For a fixed $f$ and $\xi$, $I_\gamma(f,\xi,s)$ is meromorphic on $\C \setminus \{-1,1\}$. Let $I_\gamma(f,\xi):=I_\gamma(f,\xi,0)$.
    \item For each $\gamma \in \rB(F)$ and $s \not \in \{-1,1\}$, the distribution $I_\gamma(\cdot,\xi,s)$ is continuous on $\cS(\rG'(\bA))$ and we define
        \[
            I(f,\xi,s) := \sum_{\gamma \in \rB(F)} I_\gamma(f,\xi,s).
        \]
    Where the sum on the right-hand side is absolutely convergent.
    \end{enumerate}
\end{proposition}

Similar to the results in Section ~\ref{sec:global} and ~\ref{sec:global_II}, when the test function is regular supported, the distribution $I_\gamma$ can be described in terms of the normalized orbital integral.

\begin{proposition} \label{prop:main_FJ}
    Let $f \in \cS(\rG'(\bA))$. If there exists a place $v$ of $F$ such that $f$ is of the form $f_v f^v$, with $f_v \in \cS(\rG'(F_v)), f^v \in \cS(\rG'(\bA^v))$ and $\supp(f_v) \subset \rG'_+(F_v)$ (resp. $\rG'_-(F_v)$), then (1)-(3) of Theorem ~\ref{thm:main} holds.
\end{proposition}

\begin{proposition} \label{prop:global_II_FJ}
    Let $f = f_v  f^v \in \cS(\rG'(\bA))$ with $f_v \in \cS(\rG'(F_v)), f^v \in \cS(\rG'(\bA^v))$ and $\mathrm{supp}(f_v) \subset \rG'_{\mathrm{reg}}(F_v)$, then for $b \in \rB(F)$, we have
    \begin{equation} \label{eq:global_II_main_FJ}
      I_b(f,\xi,s) = \sum_{\gamma} I^\sigma_\gamma(f,\xi,s).
    \end{equation}    
    Where $I_b$ is the distribution from relative trace formula in Proposition ~\ref{prop:coarse FJ} (2), and $I^\sigma_\gamma$ is definds as in Section ~\ref{sec:global_II}, the sum runs through all the representative of $\rH_1(F) \times \rH_2(F)$ orbits in $\rG'_b(F) \cap \rG'_{\mathrm{reg}}(F)$.
\end{proposition}

\appendix

\section{Asymptotic of modified kernel} \label{sec:Asymptotic of modified kernel}

In this appendix, we generalize the result of ~\cite{BPCZ}*{Theorem 3.3.7.1} to asymptotics of the modified kernel associated with parabolic subgroups and their Levi subgroups and introduce a Lie algebra analogue of it. We follow the strategy in ~\cite{BPCZ}*{Section 3}.

For a reductive group $G$ over $F$, and three semi-standard parabolic subgroups $R \subset S \subset Q$ of $G$, we write $\sigma_R^{S,Q}$ for the characteristic function of $X \in \fa_0$ satisfying

\begin{itemize}
    \item $\alpha(X) > 0$ for all $\alpha \in \Delta_S^R$,
    \item $\alpha(X) \le 0$ for all $\alpha \in \Delta_S^Q \setminus \Delta_S^R$,
    \item $\varpi(H)>0$ for all $\varpi \in \widehat{\Delta}_R^Q$.    
\end{itemize}

\begin{proposition} \label{prop:asympototic Q}
    Let $Q \in \cF_\mathrm{RS}$. Then for every $N>0$, there exists a continuous semi-norm $\| \cdot \|$ on $\cS(\rG'(\bA))$ such that
        \begin{equation}
            \sum_{\chi \in \fX(\rG')} \lvert K_{f,\chi}^{Q,T}(h_1,h_2) - F^{Q_{n+1}}(h_{2,n},T_{Q_{n+1}}) K_{f,Q,\chi}(h_1,h_2) \rvert \le e^{-N\|T\|} \|h_1\|_{Q_{\rH_1}}^{-N} \|h_2\|_{Q_{\rH_2}}^{-N} \|f\|
        \end{equation}
    holds for $f \in \cS(\rG'(\bA)),(h_1,h_2) \in [\rH_1]_{Q_{\rH_1}} \times [\rH_2]^\mathbbm{1}_{Q_{\rH_2}}$ and $T \in \fa_{n+1}$ sufficiently positive.
\end{proposition}

We denote by $\cT^Q_{\cF_{\mathrm{RS}},F}$ the following space of function
\[
\cT^Q_{\cF_{\mathrm{RS}},F} := \{ ({}_P\varphi) \in \prod_{ \substack{R \subset Q \\ R \in \cF_{\mathrm{RS},F}} } \cT([\GL_n]_{R_n}) \mid \varphi_R - \varphi_S \in \cS_{d_{R_{n+1}}^{S_{n+1}}}(R_n(F) \backslash \GL_n(\bA)) \text{ for any }R \subset S \subset Q  \}
\]

\begin{lemma} \label{lem:trunation Q}
    For $\underline{\varphi}=({}_P \varphi) \in \cT^Q_{\cF_{\mathrm{RS},F}}(\GL_n)$ and $g \in [\GL_n]_{Q_n}^\mathbbm{1}$ define
        \[
        \Lambda^{Q,T} \underline{\varphi}(g) = \sum_{ \substack{P \in \cF_{\mathrm{RS}} \\ P \subset Q}} \varepsilon_P^Q \sum_{\gamma \in P_n(F)\backslash Q_n(F)} \widehat{\tau}_{P_{n+1}}^{Q_{n+1}}(H_{P_{n+1}}(\gamma g)-T_{P_{n+1}}) \cdot {}_P \varphi(\gamma g).
        \]
    We also define
        \[
        \Pi^{Q,T} \underline{\varphi}(g) = F^{Q_{n+1}}(g,T) \cdot {}_{Q} \varphi(g).
        \]
    Then for any $c>0,N>0$, there exists a continuous semi-norm $\|\cdot\|$ on $\cT_{\cF_\mathrm{RS},F}^Q(\GL_n)$ such that
        \[
        \left\| \Lambda^{Q,T} \underline{\varphi} - \Pi^{Q,T} \underline{\varphi} \right\|_{\infty,N} \le e^{-c\|T\|} \| \underline{\varphi}\|
        \]
\end{lemma}

\begin{proof}
    Using the partition formula in ~\cite{Zydor20}*{Lemme 2.1}, for every $g \in [\GL_n]_{P_n}$, we have
        \[
        \sum_{ \substack{R \in \cF_{\mathrm{RS},F} \\ R \subset P}} \sum_{\delta \in R_n(F) \backslash P_n(F)} F^{R_{n+1}}(\delta g,T_{R_{n+1}}) \tau_{R_{n+1}}^{P_{n+1}}(H_{P_{n+1}}(\delta g)-T_{P_{n+1}})=1,
        \]
    together with the equation
        \[
        \widehat{\tau}_{P_{n+1}}^{Q_{n+1}} \tau_{R_{n+1}}^{P_{n+1}} = \sum_{ \substack{S \in \cF_{\mathrm{RS},F} \\ P \subset S \subset Q}} \sigma_{R_{n+1}}^{S_{n+1},Q_{n+1}}.
        \]
    For $R \subset S$ and $R,S \in \cF_{\mathrm{RS},F}$, define
        \[
        {}_{R,S} \varphi(g) = \sum_{\substack{P \in \cF_{\mathrm{RS},F}\\ R \subset P \subset S}} \epsilon_P^Q \cdot {}_P \varphi(g),
        \]
    then we obtain
        \[
        \Lambda^{Q,T} \underline{\varphi}(g) = \sum_{\substack{R,S \in \cF_{\mathrm{RS},F}\\R \subset S}} \sum_{\gamma \in R_n(F) \backslash Q_n(F)} F^{R_{n+1}}(\gamma g,T_{R_{n+1}}) \sigma_{R_{n+1}}^{S_{n+1},Q_{n+1}}(H_{R_{n+1}}(\gamma g)-T_{R_{n+1}}) {}_{R,S} \varphi(\gamma g).
        \]
    Note that $\sigma_{Q_{n+1}}^{Q_{n+1},Q_{n+1}}=1$ and $\sigma_{R_{n+1}}^{R_{n+1},Q_{n+1}}=0$ for $R \subsetneq Q$, it follows that
        \[
        \Lambda^{Q,T} \underline{\varphi}(g) - \Pi^{Q,T} \underline{\varphi}(g) = \sum_{R \subsetneq S} \sum_{\gamma \in R_n(F) \backslash Q_n(F)} F^{R_{n+1}}(\delta g,T) \sigma_{R_{n+1}}^{S_{n+1},Q_{n+1}}(H_{R_{n+1}}(\gamma g)-T_{R_{n+1}}) \cdot {}_{R,S}\varphi(\gamma g).
        \]
    Since $E_{R_n}^{Q_n}: f \mapsto (g \mapsto \sum_{\gamma \in R_n(F) \backslash Q_n(F)} f(\gamma g))$ sends $\cS^0([\GL_n]_{R_n}^\mathbbm{1})$ to $\cS^0([\GL_n]_{Q_n}^\mathbbm{1})$, it remains to show that for every $c>0$ and $N>0$, there exists a continuous semi-norm $\|\cdot \|$ on $\cT^Q_{\mathrm{RS},F}(\GL_n)$ such that
        \begin{equation} \label{eq:truncation goal}
          \lvert {}_{R,S}\varphi(g) \rvert \le e^{-c\|T\|} \|g\|_{R_n}^{-N} \|\varphi\|
        \end{equation} 
    for all $\underline{\varphi} \in \cT^Q_{\mathrm{RS},F}(\GL_n)$, $T$ sufficiently positive and $g \in R_n(F)\backslash G_n(\bA)_Q^\mathbbm{1}$ with
        \[
        F^{R_{n+1}}(g,T) \sigma_{R_{n+1}}^{S_{n+1},Q_{n+1}}(H_{R_{n+1}}(g)-T_{R_{n+1}}) \ne 0.
        \]
    We can assume $P$ is standard. Take such $g$ and take $z \in A_Q^{\rH_1,\infty}$ such that $g$ can written as $zg^1$ where $g^1 \in G_{n+1}(\bA)_Q^1$, there exists $r>0$ such that $\| g \|_{R_n} \ll \|g^1\|_{R_{n+1}}^{r_0}$. By ~\cite{BLX}*{Lemma 4.29} applied to $G=M_{Q_{n+1}}$, there exists $r>0$ such that
        \begin{align}
             e^{\| T \|} &\ll \left( \min_{\alpha \in \Delta_0^{S_{n+1}} \setminus \Delta_0^{R_{n+1}}} d_{R_{n+1},\alpha}(g^1) \right)^r \sim \left( \min_{\alpha \in \Delta_0^{S_{n+1}} \setminus \Delta_0^{R_{n+1}}} d_{R_{n+1},\alpha}(g) \right)^r \label{eq:Fsigma T} \\
              \|g\|_{R_n} &\ll \| g^1 \|_{R_{n+1}}^{r_0} \ll \left( \max_{\alpha \in \Delta_0^{S_{n+1}} \setminus \Delta_0^{R_{n+1}}} d_{R_{n+1},\alpha}(g)  \right)^{rr_0}. \label{eq:Fsigma norm}
        \end{align}     
    Fix $\alpha \in \Delta_0^{S_{n+1}} \setminus \Delta_0^{R_{n+1}}$, for any $P \in \cF_\mathrm{RS,F}$ with $\alpha \in \Delta_0^{P_{n+1}}$, there is unique $P^\alpha \in \cF_{\mathrm{RS},F}$ such that $\Delta_0^{P_{n+1}^\alpha}=\Delta_0^{P_{n+1}} \setminus \{\alpha\}$. Then there exists $N_0>0$ such that for any $r>0$, we can find a continuous semi-norm $\| \cdot \|$ on $\cT^{Q}_{\mathrm{RS},F}(\GL_n)$ with
        \[
        \lvert {}_{R,S}\varphi(g) \rvert \le \sum_{\substack{R \subset P \subset S \\ \alpha \in \Delta_0^{P_{n+1}}}} \lvert _{P}\varphi(g) - {}_{P^\alpha}\varphi(g) \rvert \le \|g\|_{R_n}^{N_0}  \sum_{\substack{R \subset P \subset S\\ \alpha \in \Delta_0^{P_{n+1}}}} d_{P_{n+1},\alpha}(g)^{-r} \|\underline{\varphi}\|.
        \]
    The equation ~\eqref{eq:Fsigma T} implies $d_{P_{n+1},\alpha} \sim d_{R_{n+1},\alpha}$, hence by ~\eqref{eq:Fsigma norm}, as $\alpha$ varies through $\Delta_0^{Q_{n+1}} \setminus \Delta_0^{P_{n+1}}$, ~\eqref{eq:truncation goal} is proved.
\end{proof}

Now we prove Proposition ~\ref{prop:asympototic Q}. 
\begin{proof}
    
Given $\psi \in \cT^0([\rH_{2,n+1}]_{Q_{\rH_{2,n+1}}})$ and $\chi \in \fX(\rG')$, consider the following sequence of maps

    \[
\begin{tikzcd}
                                               & {\cT^0([\rH_1]_{Q_{\rH_1}}) \otimes \cT^0([\rH_{2,n+1}]_{Q_{\rH_{2,n+1}}})} \arrow[r, "R(f^\vee)_{\chi^\vee} \otimes \mathrm{id}"]  & {\cT^{\Delta,Q}_{\cF_{\mathrm{RS},F}}(\rG') \otimes \cT^0([\rH_{2,n+1}]_{Q_{\rH_{2,n+1}}})} \arrow[r, "\mathrm{Res} \otimes \mathrm{id}"] & {\cT^{\Delta,Q}_{\cF_{\mathrm{RS},F}}(\rH_2)} \\
{} \arrow[r, "{\langle \cdot, \cdot\rangle }"] & {\cT^Q_{\cF_{\mathrm{RS},F}}(\rH_2)} \arrow[r, "{\Lambda^{Q,T}}"', bend left] \arrow[r, "{\Pi^{Q,T}}", bend right] & {\cS^0([\rH_2]_{Q_{\rH_2}^\mathbbm{1}})}                                                                                                  &                                              
\end{tikzcd}
    \]

Where the first map sends $\varphi \otimes \psi$ to $P \mapsto R_{\chi^\vee}(f^\vee)\varphi_{P_{\rH_1}} \otimes \psi$, the second map is restriction to $H_2 \subset \rG'$, the third map sends $({}_P \varphi) \otimes \psi$ to $P \mapsto \langle {}_P \varphi, \psi_{P_{\rH_{2,n+1}}} \rangle$, the last map is the truncation operator defined in Lemma ~\ref{lem:trunation Q}. By ~\cite{BPCZ}*{Proposition 3.4.2.1} and Lemma ~\ref{lem:trunation Q}, the images of these maps land in the target, and by the closed graph theorem, they are continuous. Let $L_{f,\chi}^{Q,T}$ be the composition of the first three maps and $\Lambda^{Q,T}$ at the last, and $P_{f,\chi}^{Q,T}$ the composition of first three maps and $\Pi^{Q,T}$ at the last. One check directly that
    \[
    L^{Q,T}_{f,\chi}(\varphi \otimes \psi)(h_{2,n})= \int_{[\rH_1]_{Q_{\rH_1}} \times [\rH_{2,n+1}]_{Q_{\rH_{2,n+1}}}} K_{f,\chi}^{Q,T}(h_1,h_{2,n},h_{2,n+1}) \varphi(h_1)\psi(h_{2,n+1}).
    \]
and
    \[
     P^{Q,T}_{f,\chi}(\varphi \otimes \psi)(h_{2,n})= F^{Q_{n+1}}(h_{2,n},T_{Q_{n+1}})\int_{[\rH_1]_{Q_{\rH_1}} \times [\rH_{2,n+1}]_{Q_{\rH_{2,n+1}}}} K_{f,Q,\chi}(h_1,h_{2,n},h_{2,n+1}) \varphi(h_1)\psi(h_{2,n+1}).
    \]
By ~\cite{BPCZ}*{Theorem 2.9.4.1.3} and Lemma ~\ref{lem:trunation Q},  for every $\varphi \otimes \psi \in \cT^0([\rH_1]_{Q_{\rH_1}}) \otimes \cT^0([\rH_{2,n+1}]_{Q_{\rH_{2,n+1}}})$ and $N>0$, we have
    \[
    \sum_{\chi \in \fX(\rG')} \left\| L_{f,\chi}^{Q,T}(\varphi \otimes \psi)-P_{f,\chi}^{Q,T}(\varphi \otimes \psi) \right\|_{\infty,N} \ll_{N,\varphi,\psi} e^{-N\|T\|}.
    \] 
By the uniformly boundedness principle, for each $\varphi \otimes \psi \in \cT^0_N([\rH_1]_{Q_{\rH_1}}) \otimes \cT^0_N([\rH_{2,n+1}]_{Q_{\rH_{2,n+1}}})$, we have
    \[
    \sum_{\chi \in \fX(\rG')} \left\| L_{f,\chi}^{Q,T}(\varphi \otimes \psi)-P_{f,\chi}^{Q,T}(\varphi \otimes \psi) \right\|_{\infty,N} \ll_N e^{-N\|T\|} \|\varphi\|_{1,-N} \|\psi\|_{1,-N}.
    \]
Apply this to $\varphi=\delta_{h_1},\psi=\delta_{h_{2,n+1}}$, Proposition ~\ref{prop:asympototic Q} is proved except the semi-norm on $f$. One then deduces it by applying the uniform boundedness principle again.
\end{proof}

There is also an analogous result on the geometric side.

\begin{proposition} \label{prop:asymptotic Q geometric}
    There exists a continuous semi-norm $\| \cdot \|$ on $\cS(\rG'(\bA))$ such that for any $f \in \cS(\rG'(\bA))$ and $N>0$, we have 
    \[
       \sum_{\gamma \in \rB(F)} \lvert K_{f,\gamma}^{Q,T}(h_1,h_2) - F^{Q_{n+1}}(h_{2,n},T_{Q_{n+1}}) K_{f,Q,\gamma}(h_1,h_2) \rvert \ll e^{-N\|T\|} \|h_1\|_{Q_{\rH_1}}^{-N} \|h_2\|_{Q_{\rH_2}}^{-N} \| f \| .
     \]
\end{proposition}

\begin{proof}
    Let $\mathbf{A}^N$ be the $N$-dimensional affine space over $F$. Choose any closed embedding $i: \rB \hookrightarrow \mathbf{A}^N$ for some $N$. We extend the function $K_{f,\gamma}$ for all $\gamma \in \bA^N$ by setting $K_{f,\gamma} = 0$ if $\gamma$ is not in the image of $i$.
    
    Fix $f$, there exists $d \in F$ such that $K_{f,\gamma}=0$ unless $\gamma \in d\cO_F^N$. Choose $u \in C_c^\infty(F_\infty^N)$ support around $0$ such that $\supp u \cap d \cO_F^N = \{ 0 \}$. Let $p$ be the composition $\rG' \to \rB \hookrightarrow \mathbf{A}^N$. For $\gamma \in (d\cO_F)^N$, define $f_\gamma(g)=f(g) \cdot u(p(g_\infty)-\gamma)$. Then by ~\cite{BLX}*{Proposition 4.30}, we have $f_\gamma \in \cS(\rG'(\bA))$ and for any continuous semi-norm $\| \cdot \|_\cS$ on $\cS(\rG'(\bA))$, the sum
        \[
            \sum_{\gamma \in \rB(F)} \| f_\gamma \|_\cS
        \]
    is finite, hence defines a continuous semi-norm on $\cS(\rG'(\bA))$ by uniform boundedness principle.
     One check directly that $K_{f_\gamma,P}(h_1,h_2) = K_{f,P,\gamma}(h_1,h_2)$ for any $P \in \cF_{\mathrm{RS}}$ and $(h_1,h_2) \in [\rH_1]_{P_{\rH_1}} \times [\rH_2]_{P_{\rH_2}}$, hence $K_{f_\gamma}^{Q,T}(h_1,h_2) = K_{f,\gamma}^{Q,T}(h_1,h_2)$. 

     Proposition ~\ref{prop:asympototic Q} implies there exists a semi-norm $\| \cdot \|_\cS$ on $\cS(\rG'(\bA)$ such that
        \[
             \lvert K_{f}^{Q,T}(h_1,h_2) - F^{Q_{n+1}}(h_{2,n},T_{Q_{n+1}}) K_{f,Q}(h_1,h_2) \rvert \le e^{-N\|T\|} \|h_1\|_{Q_{\rH_1}}^{-N} \|h_2\|_{Q_{\rH_2}}^{-N} \|f\|_\cS.
        \]
    Therefore
        \[
             \sum_{\gamma \in \rB(F)} \lvert K_{f,\gamma}^{Q,T}(h_1,h_2) - F^{Q_{n+1}}(h_{2,n},T_{Q_{n+1}}) K_{f,Q,\gamma}(h_1,h_2) \rvert \le \left( \sum_{\gamma \in \rB(F)} \| f_\gamma \|_\cS \right) e^{-N\|T\|} \|h_1\|_{Q_{\rH_1}}^{-N} \|h_2\|_{Q_{\rH_2}}^{-N}.
        \]
\end{proof}

We state the asymptotic of the modified kernel on the Levi subgroup, the proof is similar so we omit the proof.

Let $Q \in \cF_{\mathrm{RS}}$ with Levi decomposition $P=MN$. Let $f \in \cS(M(\bA))$. For $\chi \in \fX(M)$ and $T \in \fa_0$, we put
    \[
    K^{M,T}_{f,\chi}(h_1,h_2) = \sum_{\substack{P \in \cF_{\mathrm{RS}}\\P \subset Q}} \epsilon_P^Q \sum_{\substack{\gamma \in P_{M_{\rH_1}}(F) \backslash M_{\rH_1}(F)  \\ \delta \in P_{M_{\rH_2}}(F) \backslash M_{\rH_2}(F)}} \widehat{\tau}_{P_{n+1}}^{Q_{n+1}}(H_{P_{n+1}}(\delta_n h_{2,n})-T_{P_{n+1}}) K_{f,P_M,\chi}(\gamma h_1,\delta h_2).
    \]
where $(h_1,h_2) \in [M_{\rH_1}] \times [M_{\rH_2}]$. 

\begin{proposition} \label{prop:asymptotic M}
    For every $N>0$, there exists a continuous semi-norm $\| \cdot \|$ on $\cS(M(\bA))$ such that
        \begin{equation} \label{asymptotic M}
            \sum_{\chi \in \fX(M)} \lvert K_{f,\chi}^{M,T}(h_1,h_2) - F^{Q_{n+1}}(h_{2,n},T_{Q_{n+1}}) K_{f,\chi}(h_1,h_2) \rvert \le e^{-N\|T\|} \|h_1\|_{M_{\rH_1}}^{-N} \|h_2\|_{M_{\rH_2}}^{-N} \|f\|.
        \end{equation}
    holds for $f \in \cS(M(\bA)),(h_1,h_2) \in [M_{\rH_1}] \times [M_{\rH_2}]^\mathbbm{1}$ and $T \in \fa_{n+1}$ sufficiently positive.
\end{proposition}

On the Lie algebra, we have similar result
\begin{proposition} \label{prop:asympototic_Lie_algebra}
    Let $\fg$ denote either $\gl_{n+1}$ or $\widetilde{\gl_n}$. Let $Q \in \cF_{\mathrm{RS},F}$, then for every $N>0$, there exists a continuous semi-norm $\| \cdot \|$ on $\cS(\fg(\bA))$ such that
    \[
        \sum_{a \in (\fg/\GL_n)(F)} \lvert K^{Q,T}_{\varphi,a}(g) - F^{Q_{n+1}}(g,T_{Q_{n+1}}) K_{\varphi,Q,a}(g)  \rvert \le e^{-N \| T \|} \| g \|^{-N} \| \varphi \|
    \]
    for all $\varphi \in \cS(\fg(\bA))$, $g \in [\GL_n]_{Q_n}^{\mathbbm{1}}$ and $T \in \fa_{n+1}$ sufficiently positive.
\end{proposition}

\begin{proof}
    By the same argument as in the proof of Proposition ~\ref{prop:asymptotic Q geometric}, we only need to prove that there exists a continuous semi-norm $\| \cdot \|$ on $\cS(\fg(\bA))$ such that
    \[
        \lvert K_{\varphi}^{Q,T}(g) - F^{Q_{n+1}}(g,T_{Q_{n+1}})K_{\varphi,Q}(g) \rvert \le e^{-N \| T \|} \|g\|^{-N} \| \varphi \|
    \]
    holds for all $\varphi \in \cS(\fg(\bA))$ and $g \in [\GL_n]^{\mathbbm{1}}_{Q_n}$.
    
    By Lemma ~\ref{lem:trunation Q}, we are then reduced to check that the family $P \mapsto K_{\varphi,P}$ belongs to $\cT^Q_{\cF_\mathrm{RS},F}(\GL_n)$.

    We first show this for $\fg=\gl_{n+1}$. Let $R,S \in \cF_{\mathrm{RS},F}$ with $R \subset S \subset Q$. For $P \in \cF_{\mathrm{RS}}$, we will write $\fm_P$ and $\fn_P$ the the Lie algebra of $M_{P_{n+1}}$ and $N_{P_{n+1}}$ respectively. We need to show there exists $N>0$ such that for any $r>0$ and $X \in \cU(\gl_\infty)$, we have
    \[
       \lvert R(X) K_{\varphi,R}(g) - R(X) K_{\varphi,S}(g) \rvert \ll d_R^S(g)^{-r} \| g \|^N.
    \]
    Replacing $\varphi$ by $X \cdot \varphi$, we can assume $X=1$. For any $P \in \cF_{\mathrm{RS}}$, we can extend the definition of $K_{\varphi,P}$ to any $g \in [\GL_{n+1}]_{P_{n+1}}$ by the same formula as in ~\eqref{eq:defi_K_varphi,P}. After this extension, we can write
    \[
        K_{\varphi,S}(g) - K_{\varphi,R}(g) = K_{\varphi,1}(g) + K_{\varphi_2}(g),
    \]
    where
    \[
        K_{\varphi,1}(g) = \sum_{M \in \fm_R(S)} \left( \int_{\fn_S(\bA)} \varphi((M+N) \cdot g) \rd N - \int_{\fn_R(\bA)} \varphi((M+N) \cdot g) \rd N \right)
    \]
    and
    \[
        K_{\varphi,2}(g) = \sum_{M \in \fm_S(F) \setminus \fm_R(F)} \int_{\fn_S(\bA)} \varphi((M+N) \cdot g) \rd N.
    \]
    We have
    \[
        K_{\varphi,1}(g) = -K_{\varphi,S}(g) + (K_{\varphi,S})_{R_{n+1}}(g),
    \]
    here $-_{R_{n+1}}$ is the constant term operator. By ~\cite{BPCZ}*{Proposition 2.5.14.1}, there exists $N>0$ such that for any $r>0$, we have
    \begin{equation} \label{eq:proof_lie_asympototic_1}
     \lvert K_{\varphi,1}(g) \rvert \ll_{r} d_{R_{n+1}}^{S_{n+1}}(g)^{-r} \|g\|^N
    \end{equation}
    holds for any $g \in R_{n+1}(F) \backslash \GL_{n+1}(\bA)$. Thus it remains to show that there exists $N>0$ such that for any $r>0$ and $g \in R_{n+1}(F) \backslash \GL_{n+1}(\bA)$ we have
    \begin{equation} \label{eq:proof_lie_asympototic_2}
        \lvert K_{\varphi,2}(g) \rvert \ll_r d_{R_{n+1}}^{S_{n+1}}(g)^{-r} \|g\|^N.
    \end{equation}
    After replacing $R_{n+1}$ by its conjugation, we can assume $R_{n+1}$ is standard. We then can assume $g$ is in the Siegel set $\fs^{R_{n+1}}$ of $R_{n+1}$, where we recall that any $g \in \fs^{R_{n+1}}$ can be written as $ac$, where $a \in A_{n+1}^\infty$ such that $\langle \alpha, a \rangle > t$ for any $\alpha \in \Delta_0^{R_{n+1}}$ and some $t>0$, and $c$ lies in some compact subset $C$. We now proceed to show that there exists $N>0$ such that for any $r>0$ and $N_1>0$, there exists a continuous semi-norm $\| \cdot \|$ on $\cS(\gl_{n+1}(\bA))$ such that
    \begin{equation} \label{eq:proof_lie_asympototic_3}
         \sum_{M \in \fm_S(F) \backslash \fm_R(F)} \varphi( (M+Y) \cdot g ) \ll \| Y \|^{-N_1} \| g \|^N d_{R_{n+1}}^{S_{n+1}}(g)^{-r} \| \varphi \|
    \end{equation}
    holds for any $g \in \fs^{R_{n+1}}$, $Y \in \fn_S(\bA)$. Note that $\eqref{eq:proof_lie_asympototic_3}$ would imply ~\eqref{eq:proof_lie_asympototic_2}. To show $\eqref{eq:proof_lie_asympototic_3}$, after replacing $\| \cdot \|$ by $\sup_{c \in C} \| R(c) \cdot \|$, we can assume $g = a \in A_{n+1}^\infty$ with $a \in A_{n+1}^\infty$ such that $\langle \alpha, a \rangle > t$ for any $\alpha \in \Delta_0^{R_{n+1}}$. Regarding elements of $\gl_{n+1}(\bA)$ as $(n+1) \times (n+1)$ matrices, the action of $a$ preserve each entries. Therefore ~\eqref{eq:proof_lie_asympototic_3} directly follows from the following two facts:
    \begin{itemize}
        \item For any $N$ sufficiently large, we have
        \[
         \sum_{x \in F} \| ax \|^{-N} \ll \|a\|^{N}
        \]
        holds for all $a \in \bA^\times$.
        \item Fix $t>0$, for any $r>0$ there exists $N>0$ such that
        \[
         \sum_{x \in F \setminus \{0\}} \| ax \|^{-N} \ll |a|^{-r}
        \]
        holds for all $a \in \R_{>t} \subset \bA^\times$.
    \end{itemize}
    We thus finish the proof when $\fg = \gl_{n+1}$. For the case $\fg = \widetilde{\gl_n}$, choose a non-negative $\varphi' \in \cS(\bA)$ with $\varphi'(0) > 0$. Given $\varphi \in \cS(\widetilde{\gl_n}(\bA))$, put $\varphi_1 \in \cS(\gl_{n+1}(\bA))$ by
    \[
        \varphi_1 \begin{pmatrix}
            A & v \\ u & d
        \end{pmatrix} = \varphi(A,v,u) \varphi'(d).
    \]
    Then there exists a constant $C>0$ such that $K_{\varphi_1,P}(g) = C K_{\varphi,P}(g)$. Therefore, the result follows from the case for $\fg = \gl_{n+1}$.
\end{proof}

\end{document}